\documentclass[12]{article}
\usepackage{amsmath,amssymb,amsthm,color}
\usepackage{color}
\usepackage{subfigure}
\usepackage[pdftex]{graphicx}
\DeclareGraphicsExtensions{.pdf,.mps,.png,.jpg}
\usepackage{subfigure}
\usepackage{url}
\usepackage{cite}

\graphicspath{{./}{figures/}}

\newcommand{\R}{\mathbb{R}}
\newcommand{\M}{\mathbb{M}}

\newcommand{\Sphere}{\mathbb{S}}

\newcommand{\lp}{\left(}
\newcommand{\rp}{\right)}

\newcommand{\bO}{\mathcal{O}}
\newcommand{\ep}{\varepsilon}

\newcommand{\vect}[1]{\mathbf{#1}}

\newcommand{\vf}{\vect{f}}
\newcommand{\vg}{\vect{g}}
\newcommand{\vx}{\vect{x}}
\newcommand{\vy}{\vect{y}}
\newcommand{\vP}{\vect{P}}
\newcommand{\vp}{\vect{p}}
\newcommand{\vH}{\vect{H}}
\newcommand{\vI}{\vect{I}}
\newcommand{\vT}{\vect{T}}
\newcommand{\vL}{\vect{L}}
\newcommand{\vn}{\vect{n}}
\newcommand{\bfe}{\vect{e}}
\newcommand{\coef}{c}
\newcommand{\uu}{u_X}

\newcommand{\dt}{\Delta t}
\newcommand{\vxi}{\boldsymbol{\xi}}

\newcommand{\laps}{\Delta_{\M}}
\newcommand{\Mlap}{\Delta_{\M}}
\newcommand{\Mgrad}{\nabla_{\M}}
\newcommand{\grad}{\nabla}
\newcommand{\MEgrad}{\vP\nabla}
\newcommand{\Mdiv}{\Mgrad\cdot}
\newcommand{\pder}[1]{\dfrac{\partial}{\partial#1}}
\newcommand{\pgradx}{\mathcal{G}^{x}}
\newcommand{\pgrady}{\mathcal{G}^{y}}
\newcommand{\pgradz}{\mathcal{G}^{z}}

\newcommand{\sgrad}{\nabla_{\M}}
\newcommand{\ds}{\displaystyle}
\newcommand{\proj}{\mathbf{P}}

\newtheorem{theorem}{Theorem}
\newtheorem{lemma}[theorem]{Lemma}
\newtheorem{proposition}[theorem]{Proposition}

\title{A High-Order Kernel Method for Diffusion and Reaction-Diffusion Equations on Surfaces}
\author{Edward Fuselier\\
Department of Mathematics and Computer Science\\
High Point University\\
High Point, NC 27262\\
\url{efuselie@highpoint.edu}\\
\\
Grady B. Wright
\thanks{Corresponding author. Research supported by grant DMS-0934581, and DMS-0540779 from the National Science Foundation.}
\\
Department of Mathematics \\
Boise State University \\
Boise, ID 83725-1555\\
\url{gradywright@boisestate.edu}
}

\begin{document}

\maketitle

\begin{abstract}
In this paper we present a high-order kernel method for numerically solving diffusion and reaction-diffusion partial differential equations (PDEs) on smooth, closed surfaces embedded in $\R^d$.   For two-dimensional surfaces embedded in $\R^3$, these types of problems have received growing interest in biology, chemistry, and computer graphics to model such things as diffusion of chemicals on biological cells or membranes, pattern formations in biology, nonlinear chemical oscillators in excitable media, and texture mappings.  Our kernel method is based on radial basis functions (RBFs) and uses a semi-discrete approach (or the method-of-lines) in which the surface derivative operators that appear in the PDEs are approximated using collocation.  The method only requires nodes at ``scattered'' locations on the surface and the corresponding normal vectors to the surface.  Additionally, it does not rely on any surface-based metrics and avoids any intrinsic coordinate systems, and thus does not suffer from any coordinate distortions or singularities.  We provide error estimates for the kernel-based approximate surface derivative operators and numerically study the accuracy and stability of the method.   Applications to different non-linear systems of PDEs that arise in biology and chemistry are also presented.  
% Our kernel method appears to be the first high-order accurate method for the numerical solution of diffusion and reaction-diffusion equations on surfaces using on

\medskip

\noindent {\bf Key words:}  radial basis functions $\cdot$ mesh-free $\cdot$ manifold $\cdot$ collocation $\cdot$ method-of-lines $\cdot$ pattern formation $\cdot$ Turing patterns $\cdot$ spiral waves.

\medskip

\noindent {\bf Mathematics Subject Classification (2000)}  
58J45 \and 35K57 \and 41A05 \and 41A25 \and 41A30 \and 41A63 \and 65D25 \and 65M20 \and 65M70 \and 46E22 \and 35B36
\end{abstract}
			
%%%%%%%%%%%%%%%%%%%%%%%%%%%%%%%%%%%%%%%%%%%%%%%%%%%%%%%%%%%%%%%%%%%%%%%%%%%
\section{Introduction}
%%%%%%%%%%%%%%%%%%%%%%%%%%%%%%%%%%%%%%%%%%%%%%%%%%%%%%%%%%%%%%%%%%%%%%%%%%%
Kernel methods such as those based on radial basis functions (RBFs) are becoming increasingly popular for numerically solving partial differential equations (PDEs) because they are geometrically flexible, algorithmically accessible, and can be highly accurate.  Since Kansa's pioneering work~\cite{Kan90_2}, there have been many successful applications of kernel methods to various types of PDEs defined on planar regions in two and higher dimensions (see, for example,~\cite[Ch. 38--45]{Fasshauer:2007} and the references therein).   More recently, these methods have been developed and applied to PDEs defined on the surface of a sphere~\cite{QTLG2005,FlyerWright07,FoPi08,FlyerLehto2010,FlyerWright09}.

In this paper, we present a high-order kernel method using RBFs for numerically solving certain PDEs defined on more general surfaces.   In particular, we focus on diffusion and reaction-diffusion equations on smooth, closed embedded submanifolds $\M \subset \R^d$.   In the case of two species, the prototypical form of the latter type of these PDEs is 
\begin{align}
\begin{split}
\frac{\partial u}{\partial t} =& \delta_u \laps u + f_u(t,u,v), \\
\frac{\partial v}{\partial t} =& \delta_v \laps v + f_v(t,u,v), 
\end{split}
\label{eq:reaction_diffusion_system}
\end{align}
where $u,v:\M \longrightarrow \R$, $\delta_u, \delta_v \geq 0$, $f_u,f_v$ are (possibly non-linear) scalar functions, and $\laps$ is the Laplace-Beltrami operator for the surface.  For two dimensional surfaces embedded in three dimensional space, systems like \eqref{eq:reaction_diffusion_system} have received growing interest to model such things as diffusion of chemicals on biological cells or membranes~\cite{Schwartz:2005,Sbalzarini:2006}, pattern formations in biology~\cite{ChaplainEtAl2001,VareaEtAl1999}, nonlinear chemical oscillators in excitable media~\cite{DavydovEtAl2000,Grindrod08041991,ManzEtAl2003}, and texture mappings in computer graphics~\cite{Turk:1991}.  

In the past decade, many methods have been developed for PDEs like \eqref{eq:reaction_diffusion_system} on surfaces.  Nearly all of these techniques can be classified into two types, \emph{intrinsic methods} and \emph{embedded, narrow-band methods}.  The former methods use coordinates intrinsic to the surface and a surface-based mesh to discretize the differential operators (e.g.~\cite{Dziuk1988,DziukElliot2007,DziukElliot2007b,LiuXuZhang2008,Xu,Calhoun:2009,Stam:2003,LandsbergVoigt2010}), while the latter methods extend the entire PDE in $\R^3$ in a narrow band around the surface and then modify the differential operators so that the solution is restricted to the surface (e.g.~\cite{BertalmioEtAl2001,BergdorfEtAl2010,Deckelnick01042010,Schwartz:2005,Sbalzarini:2006,Adalsteinsson:2003,Greer2006,RuuthMerriman2008,MacDondaldRuuth2008,MacDondaldRuuth2009,Piret2012}). Intrinsic methods have the benefit that the resulting discretization scheme is consistent with the dimension of the original problem.  However, properly dealing with the inherent coordinate distortions or singularities that arise in the metric terms of the surface differential operators can be difficult, and these methods are generally limited to low orders of accuracy.  Embedded, narrow-band methods have the benefit that the surface differential operators are posed in extrinsic coordinates so that all coordinate singularities can be avoided and standard methods such as finite differences on 3D Cartesian grids, or finite elements on 3D unstructured meshes can be used.  Additionally, the surface can be naturally represented using standard level-set methods, and thus quite topologically complicated surfaces can be handled.  However, these methods require consistent extensions of the initial data on the surface to the embedding space, which can be non-trivial.  Also, some of these methods  (e.g.~\cite{BertalmioEtAl2001}) lead to degenerate surface differential operators since, for example, they allow diffusion to occur only in directions tangential to the surface.  This makes it difficult to use implicit discretization techniques for time-dependent problems.  Additionally, the need to implement artificial boundary conditions in the embedding space can lead to accuracy degradations~\cite{Greer2006}.  Finally, all embedded methods have an added computational expense since they solve the equations in a dimension at least one greater than that in which the equations are posed.  This added expense can grow quite significantly depending on the surface, the order of the differential operators, and the order of accuracy of the method~\cite{MacDondaldRuuth2008}.

The present kernel method combines many of the benefits of both the intrinsic and narrow band methods.  It uses a semi-discrete approach (or the method-of-lines) in which the surface derivative operators (e.g. surface Laplacian) that appear in the PDEs are approximated using collocation.  These approximations are made using RBF interpolation on surfaces~\cite{FuselierWright2011}.  In this way, the method is similar to the ones used in~\cite{FlyerWright07,FoPi08,FlyerLehto2010,FlyerWright09} for the surface of a sphere.   However, the new method applies to more general surfaces, including those that can be defined parametrically or implicitly (as the level-set of a function).  Since the method is based on RBFs, it allows the computational nodes to be placed at ``scattered'' locations on the surface (see Figure \ref{fig:ex_surfaces} for examples).  Furthermore, it does not rely on any surface-based metrics or intrinsic coordinate systems, thus avoiding any coordinate distortions or singularities.   In this regard, our kernel method is similar to the embedded, narrow-band methods discussed above.   The difference is that our method approximates the surface differential operators directly on the surface instead of having to extend quantities off the surface and do the approximations in $\mathbb{R}^3$.  Thus, like the intrinsic methods, our method is posed in the same dimension as the original problem.  This ability to directly approximate the differential operator also means that no degeneracies arise from having to modify the equations so that the solutions remain restricted to the surface.  Furthermore, computational efficiencies are gained by not having to extend into the full three dimensional space to solve the problems.

Recently, a different kernel method based on RBFs was introduced by Piret for solving PDEs on surfaces, which is called the Orthogonal Gradients (OGr) method \cite{Piret2012}. This method shares many similarities to the embedded, narrow band method known as the Closest Point Method \cite{RuuthMerriman2008,MacDondaldRuuth2008,MacDondaldRuuth2009} and is different from our method in a number of ways. For example, the OGr method requires expanding into the embedding space ($\R^3$) to obtain discrete surface differential operators, which increases the computational complexity. Further, a differentiation matrix is constructed by enforcing that the first and second normal derivatives to the surface of a certain RBF interpolant vanish. This means in particular that one must analytically calculate the derivatives of the normal vectors to the surface, which requires that a smooth representation of the surfaces is known. In contrast, our method only requires nodes at scattered locations on the surface and the corresponding normal vectors to the surface, and does not require expanding into $\R^3$. Furthermore, the discrete operators for the present method are constructed from the inversion of a single traditional positive-definite kernel interpolation matrix, which is guaranteed to be invertible. Additionally, we provide a theory for the convergence of the discrete surface differential operators.

The application of kernel methods (including RBFs) to problems on general surfaces is still in its infancy.  The goal of this paper is thus to investigate some of the theoretical and applied aspects of our kernel method and to establish its applicability.  As a result, we (a) present error estimates for approximating various surface differential operators; (b) numerically investigate the stability and accuracy of the method, showing how solutions can be computed to very high orders of accuracy; and (c) apply the method to two relevant problems in biology and chemistry.

The remainder of the paper is organized as follows.  In next section, we briefly review kernel interpolation with RBFs.  We follow this with a short presentation in Section \ref{sec:operators} on how to formulate surface differential operators in Cartesian coordinates.   Section \ref{sec:kernel_method} describes how these surface differential operators are discretized in the form of differentiation matrices and presents the kernel method in method-of-lines form.   Making heavy use of the results from~\cite{FuselierWright2011}, we provide error estimates for the discrete surface derivatives in Section \ref{sec:convergence} (and proofs in Appendix \ref{proofappendix}).   Section \ref{sec:eigen_stability} numerically investigates the eigenvalue stability of the kernel method.   We numerically demonstrate the convergence the method for the forced scalar diffusion equation to two different surfaces in Section \ref{sec:numerical_results}.  Applications of the method to simulations of pattern formations in a Turing system and to spiral waves in excitable media are presented in Section \ref{sec:applications}.  We conclude with some comments on potential future enhancements of the method in Section \ref{sec:concluding_remarks}.

Before continuing to the main body of the paper, we pause to make some relevant remarks on what follows.  While there are many types of classifications of surfaces or manifolds, in our work, we use the term surface or manifold to always refer to a smooth embedded submanifold of $\R^d$ with no boundary.  We use the notation $\M$ to denote the manifold.  For simplicity, we present our method for manifolds of dimension 2 embedded in $\R^3$.  However, it can be naturally extended to manifolds of codimension 1 embedded in $\R^d$.  Although not pursued here, more technical, but straightforward extensions are possible for manifolds of higher codimension.

\begin{figure}[t]
\centering
\begin{tabular}{ccc}
\includegraphics[width=0.31\textwidth]{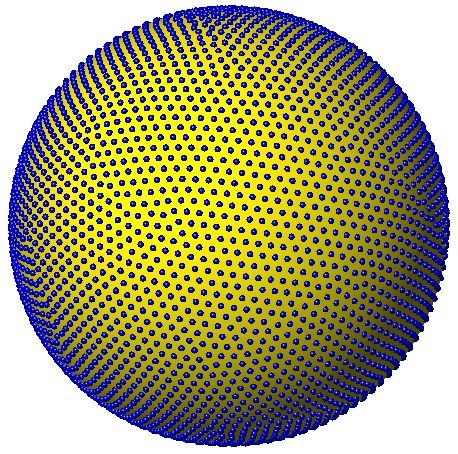} & 
\includegraphics[width=0.32\textwidth]{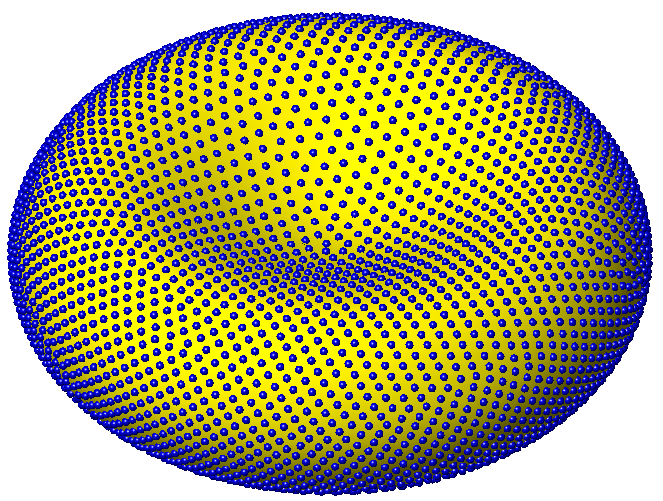} &
\includegraphics[width=0.31\textwidth]{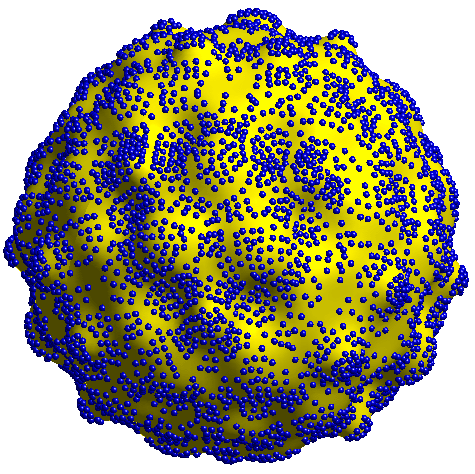}
\\
(a) Unit sphere & (b)  Red blood cell & (c) ``Bumpy sphere'' \\
\includegraphics[width=0.31\textwidth]{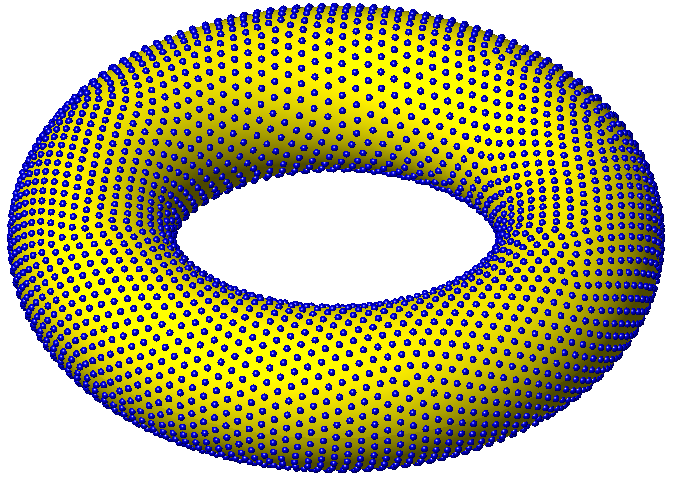} & 
\includegraphics[width=0.33\textwidth]{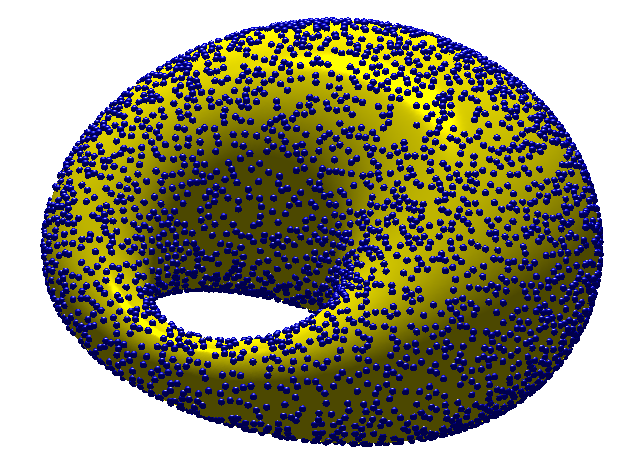} &
\includegraphics[width=0.33\textwidth]{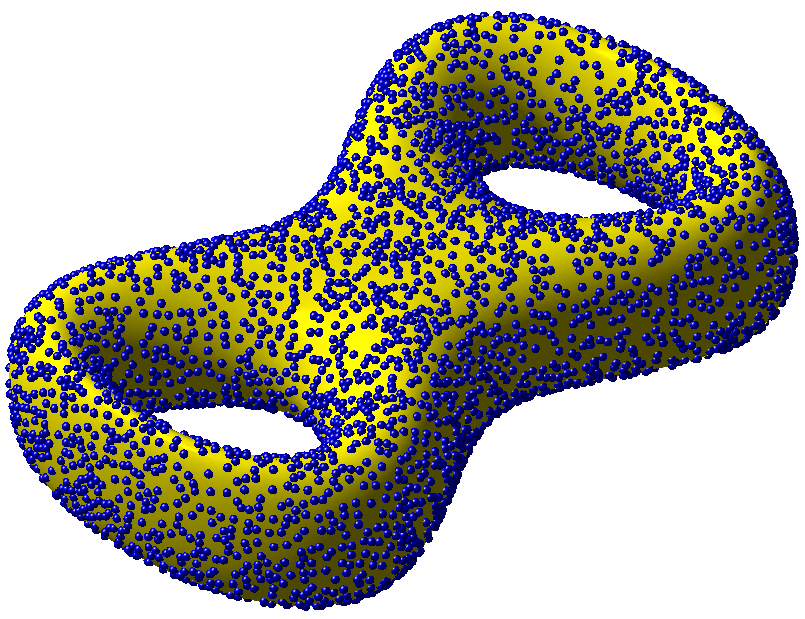}
\\
(d) Torus & (e) Dupin's cyclide & (f) Bretzel2
\end{tabular}
\caption{Example surfaces and ``scattered'' nodes (represented as small solid spheres) used in the numerical experiments.  A description of the surfaces and the node sets is given in Appendix \ref{apndx:Surfaces}.\label{fig:ex_surfaces}}
\end{figure}

%%%%%%%%%%%%%%%%%%%%%%%%%%%%%%%%%%%%%%%%%%%%%%%%%%%%%%%%%%%%%%%%%%%%%%%%%%%%%%%%%%%%%%%%
\section{Brief overview of kernel interpolation}\label{sec:overview}
%%%%%%%%%%%%%%%%%%%%%%%%%%%%%%%%%%%%%%%%%%%%%%%%%%%%%%%%%%%%%%%%%%%%%%%%%%%%%%%%%%%%%%%%
Fundamental to the construction of our approximate surface differential operators are kernel interpolants.  Let $\Omega \subseteq \R^d$ and $\phi:\Omega\times\Omega\rightarrow \R$ be a continuous function, which we will refer to as a \emph{kernel}.  Given a set of ``scattered'' nodes $X = \{\vx_j\}_{j=1}^N\subset \Omega$ and a continuous target function $f:\R^d\rightarrow\R$ sampled at $X$, a kernel interpolant takes the general form
\begin{align}
I_{\phi}f(\vx) = \sum_{j=1}^{N} c_j \phi(\vx,\vx_j),\; \vx\in\Omega,
\label{eq:kernel_interp}
\end{align}
where $c_j$ are determined by requiring $I_{\phi}f\bigr|_{X} = f\bigr|_{X}$.  It is well-known that there are many kernels $\phi$ for which a unique solution to this problem exists. In particular, if for any non-zero vector $b\in\R^{N}$, $\phi$ satisfies
\begin{align}
\sum_{i=1}^N\sum_{j=1}^N b_i\phi(\vx_i,\vx_j)b_j > 0,
\label{eq:pos_def}
\end{align}
then there exists a unique solution to \eqref{eq:kernel_interp}.  We call $\phi$ that satisfy \eqref{eq:pos_def} for all finite node sets $X\subset\Omega$ \emph{positive definite kernels} on $\Omega$.  

In this study we focus on the popular subclass of positive definite kernels called \emph{radial basis functions} (RBFs), which have the property that $\phi(\vx,\vy)=\phi(\|\vx - \vy\|)$, where $\|\cdot\|$ is the standard Euclidean norm. (Strictly speaking, there are RBFs that are only conditionally positive definite~\cite[Ch. 7]{Fasshauer:2007}, but our focus will be on the positive definite ones).  For these radial kernels \eqref{eq:kernel_interp} becomes
\begin{align}
I_{\phi}f(\vx) = \sum_{j=1}^{N} c_j \phi(\|\vx-\vx_j\|)
\label{eq:rbf_interp}
\end{align}
and the interpolation constraints can be expressed as the following linear system:
\begin{align}
\begin{bmatrix}
\phantom{c} & & \phantom{c} \\
\phantom{\vdots} & A_X & \phantom{\vdots}  \\
\phantom{c} & & \phantom{c}
\end{bmatrix}
\underbrace{
\begin{bmatrix}
c_1 \\
\vdots \\
c_N \\
\end{bmatrix}}_{\ds \coef_{f}}
=
\underbrace{
\begin{bmatrix}
f(\vx_1) \\
\vdots \\
f(\vx_N)
\end{bmatrix}}_{\ds f_X},
\label{eq:rbf_lin_sys}
\end{align}
where $(A_X)_{i,j} = \phi(\|\vx_i - \vx_j\|)$.  For a positive definite $\phi$, this system is positive definite and hence non-singular.

We will use the following two popular RBFs in this study:
\begin{align}
\text{\underline{Mat\'ern}:}&\quad \phi_{\nu}(r) = C_{\nu}(\ep r)^{\nu-d/2}K_{\nu-d/2}(\ep r),\; \nu > d/2,\; \ep > 0, \label{eq:matern}\\
\text{\underline{IMQ}:}&\quad \phi(r) = \frac{1}{\sqrt{1 + (\ep r)^2}},\; \ep > 0, \label{eq:imq}
\end{align}
which are positive definite on $\R^d$. In \eqref{eq:matern}, $C_\nu = 2^{1-(\nu-d/2)}/\Gamma(\nu-d/2)$ and $K_{\beta}$ is the modified Bessel function of the second kind of order $\beta$.  The Mat\'ern family of kernels are piecewise smooth, with the smoothness controlled by the parameter $\nu$.  For example, with $\nu = \frac{d+3}{2}$, $\phi_{\nu}\in C^{2}(\R^d)$ and with $\nu=\frac{d+5}{2}$, $\phi_{\nu}\in C^{4}(\R^d)$.  In contrast, the inverse multiquadric (IMQ) kernel is $C^{\infty}(\R^d)$.  Infinitely smooth kernels like the IMQ are used often in solving PDEs, and we have included results for the Mat\'ern family to illustrate the different convergence rates that are possible with finitely smooth kernels.

%The parameter $\ep$ is called the ``shape parameter'' since it can be used for changing the kernels from peaked (large $\ep$) to flat (small $\ep$). \textcolor{blue}{In the limit that $\epz$ (i.e. the basis function are completely flat), the RBF interpolant \eqref{eq:rbf_interp} based on the IMQ kernel (and many other infinitely smooth kernels) converges to a polynomial interpolant~\cite{DriscollAndFornberg2002,FornbergWrightLarsson2004,LarssonAndFornberg2005,Schaback2005,FornbergPiret2007}.  Thus, all classical spectral collocation methods can be reproduced in this limit when the nodes in $X$ are arranged accordingly (e.g. Gauss-Chebyshev nodes for Chebyshev methods, nodes on the unit circle for Fourier methods, and nodes on the unit sphere for spherical harmonic methods).  For piecewise smooth Mat\'ern kernels, the interpolant \eqref{eq:rbf_interp} converges to a polyharmonic spline interpolant as $\epz$~\cite{SongEtAl2011}.  Thus, different spline collocation methods in 1-D can be reproduced in this limit.\footnote{Do we say too much about the shape parameter... thoughts?}}

Typically, kernel, and in particular RBF, interpolation is applied to problems where the data are scattered in $\R^d$ or on the surface of the sphere.  The present authors have recently proved that favorable error estimates can be achieved when RBF interpolation is applied to reconstruction problems on more general manifolds $\M$ of $\R^d$\cite{FuselierWright2011}.  It is worth noting, that the interpolants for these reconstruction problems are not modified from their original form \eqref{eq:rbf_interp}, i.e. distances are still measured as straight line distances, and not as distances intrinsic to the manifold.  Thus, the interpolation method needs only limited knowledge of the underlying manifold; it only requires computing the solution to the linear system \eqref{eq:rbf_lin_sys}.  In what follows, we will use \eqref{eq:rbf_interp} to approximate the surface gradient and surface divergence over a given set of nodes $X\subset\M$ on the manifold, and combine these to obtain an approximation for the surface Laplacian.  For this construction, we will only need a point set $X\subset\M$ and the normals to $\M$ at each point in $X$.

We conclude with some remarks on the parameter $\ep$ in \eqref{eq:matern} and \eqref{eq:imq}, which is called the ``shape parameter'' since it can be used for changing the kernels from peaked (large $\ep$) to flat (small $\ep$).  In general, better accuracy is obtained for smaller values of $\ep$~\cite[Ch. 16--17]{Fasshauer:2007,FornbergWrightLarsson2004}.  However, the standard way of computing the interpolant by solving \eqref{eq:rbf_lin_sys} becomes increasingly ill-conditioned.  While several types of stable algorithms have been developed for bypassing this ill-conditioning~\cite{FornbergWright2004,FasshauerMcCourt2012,FornbergPiret2007,FornbergLarssonFlyer2011}, these algorithms generally break down when the nodes $X$ fall on a lower dimensional manifold of $\R^d$, unless the surface topology is taken into account in the algorithm (as done on the sphere in~\cite{FornbergPiret2007}).   Modifications to these algorithms are underway to address this issue~\cite{LarssonEtAl2012}, but they are not yet generally available.  As a result, in this study we will not give much attention to issues of the shape parameter, but will instead choose it in the applications to ensure the linear system \eqref{eq:rbf_lin_sys} is reasonably well-conditioned.

%%%%%%%%%%%%%%%%%%%%%%%%%%%%%%%%%%%%%%%%%%%%%%%%%%%%%%%%%%%%%%%%%%%%%%%%%%%%%%%%%%%%%%%%
\section{Continuous surface differential operators}\label{sec:operators}
%%%%%%%%%%%%%%%%%%%%%%%%%%%%%%%%%%%%%%%%%%%%%%%%%%%%%%%%%%%%%%%%%%%%%%%%%%%%%%%%%%%%%%%%
Here we assume that $\M$ is a smooth embedded submanifold of $\R^d$ with no boundary and $\text{dim}(\M)=d-1$. In the descriptions of the surface differential operators on $\M$ that follows, we will focus on the case of $d=3$ for notational simplicity.  The extension to other $d$ should be obvious.  Additionally, all expressions are given in extrinsic (or Cartesian) coordinates.

For any point $\vx=(x,y,z)$ on $\M$, we denote the normal vector to $\M$ at $\vx$ as $\vn=(n^x,n^y,n^z)$ and the vector space of tangent vectors to $\M$ at $\vx$ as $T_{\vx} \M$.  The surface gradient on $\M$ at $\vx$ is then given as
\begin{align}
\Mgrad := \proj\nabla = \lp \vI - \vn\vn^{T} \rp \nabla
\label{eq:Mgrad1}
\end{align}
Here the $3$-by-$3$ matrix $\proj$ projects vectors in $\R^3$ to $T_\vx\M$.  Letting $\bfe^x$, $\bfe^y$, $\bfe^z$, be the standard unit vectors in $x$, $y$, and $z$ directions in $\R^3$, we can re-write \eqref{eq:Mgrad1} in component form as
\begin{align}
\Mgrad := 
\begin{bmatrix}
\lp \bfe^x \cdot \vP \rp \grad  \\
\lp \bfe^y \cdot \vP \rp \grad  \\
\lp \bfe^z \cdot \vP \rp \grad 
\end{bmatrix}
=
\begin{bmatrix}
\lp \bfe^x  - n^x \vn \rp \cdot \grad \\
\lp \bfe^y  - n^y \vn \rp \cdot \grad \\
\lp \bfe^z  - n^z \vn \rp \cdot \grad
\end{bmatrix}
= 
\begin{bmatrix}
\vp^x \cdot \grad \\
\vp^y \cdot \grad \\
\vp^z \cdot \grad \\
\end{bmatrix}
= 
\begin{bmatrix}
\pgradx \\
\pgrady \\
\pgradz \\
\end{bmatrix}.
\label{eq:Mgrad}
\end{align}
With this notation, the surface divergence of a smooth vector field $\vf = (f^x,f^y,f^z):\M\rightarrow \R^3$ at a point $\vx\in\M$ can be expressed as
\begin{align}
\Mdiv \vf := \lp \MEgrad \rp \cdot \vf =  \pgradx f^x + \pgrady f^y + \pgradz f^z.
\label{eq:Mdiv}
\end{align}
Finally, the Laplace-Beltrami operator on $\M$ at $\vx$ can be written using \eqref{eq:Mgrad} and \eqref{eq:Mdiv} as:
\begin{equation}
\Mlap := \Mdiv\Mgrad = \lp \MEgrad \rp \cdot \lp \MEgrad \rp = 
\pgradx\pgradx + \pgrady\pgrady + \pgradz\pgradz.
\label{eq:Mlap}
\end{equation}
%
%This can be simplified further to
%
%\begin{align*}
%\Mlap := \Mgrad \cdot \grad - (\Mgrad\cdot \vn)(\vn \cdot \grad) = \lp \MEgrad \rp\cdot \grad - \lp\lp \MEgrad \rp\cdot\vn\rp\lp\vn \cdot \grad \rp,
%\end{align*}
%\textcolor{red}{Weren't there problems with this; should we double check?} which may be more useful than \eqref{eq:Mlap} when the normal vectors $\vn$ to $\M$ have a simple expression (as in the case of the sphere). In what follows, we will use the Laplace-Beltrami operator in the form of \eqref{eq:Mlap} since our method works by replacing the continuous operator \eqref{eq:Mgrad} with a discrete version based on the kernel interpolant \eqref{eq:rbf_interp}.
In what follows, we will approximate the Laplace-Beltrami operator in the form of \eqref{eq:Mlap} by replacing the continuous operator \eqref{eq:Mgrad} with a discrete version based on the kernel interpolant \eqref{eq:rbf_interp}.

%%%%%%%%%%%%%%%%%%%%%%%%%%%%%%%%%%%%%%%%%%%%%%%%%%%%%%%%%%%%%%%%%%%%%%%%%%%%%%%%%%%%%%%%
\section{Kernel collocation technique in method-of-lines form}\label{sec:kernel_method}
%%%%%%%%%%%%%%%%%%%%%%%%%%%%%%%%%%%%%%%%%%%%%%%%%%%%%%%%%%%%%%%%%%%%%%%%%%%%%%%%%%%%%%%%
The discrete approximations to the surface differential operators that will be used in the method-of-lines (MOL) are based on applying the projected gradient \eqref{eq:Mgrad} to the kernel interpolant \eqref{eq:rbf_interp}.  We illustrate how to construct these approximations by first applying the projected gradient to a radial kernel.  We then show how to express the discrete operators as \emph{differentiation matrices}.  This is followed by a description of the MOL formulation with these discrete operators for the diffusion equation.  As in the previous section, we formulate everything in $\R^3$, with an obvious extension to other dimensions.

%To describe the method-of-lines (MOL) formulation, we need to illustrate how to construct the discrete approximations to the surface differential operators.  To do this we first start by applying the projected gradient to a radial kernel.  We then show how to construct the discrete operators in terms of \emph{differentiation matrices}.  This is followed by a description of the MOL formulation for a prototypical example. As in the previous section, we formulate everything in terms of $\R^3$, with an obvious extension to other dimensions.

%%%%%%%%%%%%%%%%%%%%%%%%%%%%%%%%%%%%%%%%%%%%%%%
\subsection{Projected gradient of a radial kernel}
%%%%%%%%%%%%%%%%%%%%%%%%%%%%%%%%%%%%%%%%%%%%%%%
Let $\vx$ and $\vx_{j}$ be points on $\M$, and let $r_j(\vx) = \|\vx - \vx_j\|$ denote the Euclidean distance from $\vx$ to $\vx_{j}$, i.e. $r_j(\vx)=\sqrt{(x-x_{j})^{2}+(y-y_{j})^{2}+(z-z_{j})^{2}}$.
With this notation, a radial kernel $\phi:\R^3\times\R^3\rightarrow \R$ centered at $\vx_{j}$ is given by $\phi(r_j(\vx))$.  Under the assumption that $\phi(r_j(\vx))$ is differentiable at $\vx$, a simple application of the chain rule gives the gradient of $\phi(r_j(\vx))$ as
\begin{equation}
\nabla \phi(r_j(\vx)) = 
\begin{bmatrix}
\pder{x}\phi(r_j(\vx)) \\
\pder{y}\phi(r_j(\vx)) \\ 
\pder{z}\phi(r_j(\vx))
\end{bmatrix}
=
\begin{bmatrix}
(x-x_j) \\
(y-y_j) \\
(z-z_j)
\end{bmatrix}
\dfrac{\phi'(r_j(\vx))}{r_j(\vx)}
= (\vx - \vx_j)\dfrac{\phi'(r_j(\vx))}{r_j(\vx)},
\label{eq:Cgradphi}
\end{equation}
where we have used $'$ to denote differentiation with respect to $r$.  Note that the apparent singularity in the above formula when $r_j(\vx_j)=0$ will analytically cancel since $\phi$ is assumed to be an even function with at least one continuous derivative.  The surface gradient of $\phi(r_j(\vx))$ can then be obtained from \eqref{eq:Mgrad}.  After applying \eqref{eq:Mgrad}, the components of the surface gradient of $\phi(r_j(\vx))$ are
\begin{align}
\pgradx \phi(r_j(\vx)) &= \lp(1-n^x n^x)(x-x_j) - n^x n^y (y-y_j) - n^x n^z (z-z_j)\rp\dfrac{\phi'(r_j(\vx))}{r_j(\vx)}, \label{eq:pgradx_phi}\\
\pgrady \phi(r_j(\vx)) &= \lp(1-n^y n^y)(y-y_j) - n^x n^y (x-x_j) - n^y n^z (z-z_j)\rp\dfrac{\phi'(r_j(\vx))}{r_j(\vx)}, \label{eq:pgrady_phi}\\
\pgradz \phi(r_j(\vx)) &= \lp(1-n^z n^z)(z-z_j) - n^x n^z (x-x_j) - n^y n^z (y-y_j)\rp\dfrac{\phi'(r_j(\vx))}{r_j(\vx)}. \label{eq:pgradz_phi}
\end{align}
We now have all the components necessary to build the RBF approximation of the surface gradient $\vP\nabla$ of a scalar quantity defined on $\M$.

%%%%%%%%%%%%%%%%%%%%%%%%%%%%%%%%%%%%%%%%%%%%%%%
\subsection{Discrete operators in terms of differentiation matrices}
%%%%%%%%%%%%%%%%%%%%%%%%%%%%%%%%%%%%%%%%%%%%%%%
Let $f:\M\rightarrow \R$ be some differentiable target function known at a set of ``scattered'' node locations $X=\{\vx_j\}_{j=1}^N\subset \M$.  The first step in constructing a discrete approximation to the surface gradient of $f$ on $X$ is to construct an interpolant to $f$ of the form \eqref{eq:rbf_interp}.  In the notation of this section, the interpolant takes the form
\begin{align}
I_{\phi}f(\vx) = \sum_{j=1}^{N} \coef_j\phi(r_j(\vx)),
\end{align}
where the coefficients $\coef_j$ are determined by collocation.  We next apply the projected gradient operator \eqref{eq:Mgrad} to $I_{\phi}f$ and evaluate at the nodes in $X$.  Focusing on the $\pgradx$ component of this operator and using \eqref{eq:pgradx_phi}, we obtain
\begin{align*}
\lp\pgradx I_{\phi}f(\vx)\rp\bigr|_{\vx=\vx_i} &= \sum_{j=1}^{N} \coef_j \lp\pgradx \phi(r_j(\vx))\rp\bigr|_{\vx=\vx_i}, \\
&= \sum_{j=1}^{N} \coef_j \underbrace{\lp(1-n^{x}_i n^{x}_i)(x_i-x_j) - n^{x}_{i} n^{y}_{i} (y_i-y_j) - n^{x}_{i} n^{z}_{i} (z_i-z_j)\rp\dfrac{\phi'(r_j(\vx_i))}{r_j(\vx_i)}}_{\ds (B_X^x)_{i,j}},
\end{align*}
where $i=1,\ldots,N$ and $(n^x_i,n^y_i,n^z_i)$ is the normal vector at $\vx_i$. We can write this approximation in terms of a matrix vector product using the notation from the linear system \eqref{eq:rbf_lin_sys} as follows:
\begin{align}
\lp\pgradx I_{\phi}f\rp\bigr|_{X} &= B^{x}_X\coef_{f} = (B^x_X A_X^{-1})f_X=G^{x}_{X}f_X.  \label{eq:dm_gx}
\end{align} 
The matrix $G^{x}_{X}$ is an $N$-by-$N$ differentiation matrix that represents the discrete RBF approximation to the $x$-component of the projected gradient operator over the set of nodes in $X$.  We can similarly obtain the discrete approximations to the $y$- and $z$-components of the projected gradient operator as follows:
\begin{align}
\lp\pgrady I_{\phi}f\rp\bigr|_{X} &= (B^y_X A_X^{-1})f_X = G^{y}_{X}f_X, \label{eq:dm_gy}\\
\lp\pgradz I_{\phi}f\rp\bigr|_{X} &= (B^z_X A_X^{-1})f_X = G^{z}_{X}f_X, \label{eq:dm_gz}
\end{align}
where the entries in $B^y_X$ and $B^z_X$ are given by
\begin{align*}
(B^y_X)_{i,j} &= \lp(1-n_i^y n_i^y)(y_i-y_j) - n_i^x n_i^y (x_i-x_j) - n_i^y n_i^z (z_i-z_j)\rp\dfrac{\phi'(r_j(\vx_i))}{r_j(\vx_i)}, \\
(B^z_X)_{i,j} &= \lp(1-n_i^z n_i^z)(z_i-z_j) - n_i^x n_i^z (x_i-x_j) - n_i^y n_i^z (y_i-y_j)\rp\dfrac{\phi'(r_j(\vx_i))}{r_j(\vx_i)},
\end{align*}
for $i,j=1,\ldots,N$.  The full discrete approximation to the surface gradient operator $\Mgrad$ is given in terms of \eqref{eq:dm_gx}--\eqref{eq:dm_gz} as the $3N$-by-$N$ matrix
\begin{align}
G_X = 
\begin{bmatrix}
G_X^x \\ G_X^y \\ G_X^z
\end{bmatrix}.
\label{eq:Mgrad_dm}
\end{align}

A discrete approximation to the surface divergence operator \eqref{eq:Mdiv} can also be constructed from \eqref{eq:dm_gx}--\eqref{eq:dm_gz}.  Let $\vf = (f^x,f^y,f^z):\M\rightarrow \R^3$ be a smooth vector field and $\vf_X = \begin{bmatrix} f^x_X & f^y_X & f^z_X\end{bmatrix}^T$ denote the $3N$-by-$1$ vector of samples 
of $\vf$ on $X$.  Then the discrete approximation of the surface divergence of $\vf$ on $X$ is given by
\begin{align}
D_X \vf_X = 
\begin{bmatrix}
G_X^x & G_X^y & G_X^z
\end{bmatrix}\vf_X =
G_X^x f^x_X + G_X^y f^y_X + G_X^z f^z_X.
\label{eq:Mdiv_dm}
\end{align} 

Finally we can combine \eqref{eq:Mgrad_dm} and \eqref{eq:Mdiv_dm} to obtain a discrete approximation to the Laplace-Beltrami operator \eqref{eq:Mlap}:
\begin{align}
L_X = D_X G_X = G_X^x G_X^x + G_X^y G_X^y + G_X^z G_X^z.
\label{eq:Mlap_dm}
\end{align}
%Using the notation from \eqref{eq:dm_gx}--\eqref{eq:dm_gz} this can be expressed as follows:
%\begin{align}
%L_X = (B_X^x A_X^{-1} B_X^x + B_X^y A_X^{-1} B_X^y + B_X^z A_X^{-1} B_X^z)A_X^{-1}.
%\label{eq:Mlap_dm2}
%\end{align}
The $N$-by-$N$ differentiation matrix $L_X$ represents our discrete RBF approximation to the surface Laplacian and is what we will use in the applications that follow.  It is important to note that the construction of $L_X$ only requires a node set $X\subset \M$ and the normal vectors to $\M$.  Additionally, $L_X$ is constructed using extrinsic coordinates and is thus free from any coordinate singularities (e.g. the pole singularity for the unit sphere $\Sphere^2$).  In the case of $\M=\Sphere^2$, other studies have taken a different approach at approximating the surface Laplacian using kernel collocation (e.g.~\cite{QTLG2005,WrightFlyerYuen}).  These studies apply the surface Laplacian directly instead of approximating it with projected gradient as we have done.  For more general manifolds, it is non-trivial to work out the surface Laplacian in closed form, which is why we are advocating the use of the discrete approximation \eqref{eq:Mlap_dm}.  

In Section \ref{sec:convergence}, we show that our discrete approximations \eqref{eq:Mgrad_dm}--\eqref{eq:Mlap_dm} converge to the continuous surface derivatives \eqref{eq:Mgrad}--\eqref{eq:Mlap}, respectively, as the nodes in $X$ increases and ``fill'' the manifold $\M$.

Although the computational cost of constructing the discrete surface differential operators is $\bO(N^3)$, these constructions are a preprocessing step and can be done once and stored.  Applying these surface differential operators to a vector requires $\bO(N^2)$ operations.  In Section \ref{sec:concluding_remarks}, we make some remarks on how these cost may be reduced.

\subsection{MOL formulation}
To describe the MOL formulation of our kernel method, we use the equation for diffusion of a scalar quantity $u$ on a surface with a (non-linear) forcing term:
\begin{align}
\frac{\partial u}{\partial t} = \delta \laps u + f(t,u).
\label{eq:diffusion_scalar}
\end{align}
Here $\delta>0$ is the diffusion coefficient, $f(t,u)$ is the forcing term, and an initial value of $u$ at time $t=0$ is given.  Letting $X=\{\vx_j\}_{j=1}^N\subset\M$ and $\uu\in\R^N$ denote the vector containing the samples of $u$ at the points in $X$, our method for \eqref{eq:diffusion_scalar} takes the form
\begin{align}
\frac{d}{dt} \uu = \delta L_X \uu + f\lp t,\uu\rp,
\label{eq:kernel_mol}
\end{align}
where $L_X$ is the differentiation matrix \eqref{eq:Mlap_dm} corresponding to our discrete approximation of $\laps$.  This is a system of $N$ coupled ODEs and, provided it is stable, can be advanced in time with a suitably chosen time-integration method (see the Sections \ref{sec:numerical_results} and \ref{sec:applications} for examples).   In Section \ref{sec:eigen_stability}, we numerically investigate the eigenvalue stability of the method. 

%%%%%%%%%%%%%%%%%%%%%%%%%%%%%%%%%%%%%%%%%%%%%%%%%%%%%%%%%%%%%%%%%%%%%%%%%%%%
\section{Convergence results for the discrete surface derivatives}\label{sec:convergence}
%%%%%%%%%%%%%%%%%%%%%%%%%%%%%%%%%%%%%%%%%%%%%%%%%%%%%%%%%%%%%%%%%%%%%%%%%%%%

In this section we present the convergence results for our discrete differential operators. Given that these operators are essentially obtained by differentiating interpolants, our error estimates will rely on the approximation power of kernel interpolants. Recently, Sobolev error estimates for a wide variety of kernel interpolants (constructed from RBF kernels restricted to surfaces) were made available for smooth, embedded submanifolds of $\R^d$ \cite{FuselierWright2011}, and these results can be used to derive the bounds in this section. Precisely stated theorems and their proofs are given in Appendix \ref{proofappendix} for the interested reader. A summary of the results will be presented after we establish the necessary notation.

%%%%%%%%%%%%%%%%%%%%%%%%%%%%%%%%%%%%%%%%%%%%%%%
\subsection{Definitions and notation}\label{sec:definitions}
%%%%%%%%%%%%%%%%%%%%%%%%%%%%%%%%%%%%%%%%%%%%%%%

As is typical in scattered data approximation, the error estimates are given in terms of the \emph{mesh norm} $h$ and \emph{separation radius} $q$ of the point set $X$ on $\M$, which are given by
\begin{align*}
h:=\sup_{\vx\in\M}\min_{\vx_j\in X}d_{\M}(\vx,\vx_j) \hspace{.5in} q:=\frac{1}{2}\min_{\underset{\vx_j\neq \vx_i}{\vx_j\in X}}d_{\M}(\vx_i,\vx_j),
\end{align*}
where $d_{\M}(\vx,\vx_j)$ is the geodesic distance between $\vx$ and $\vx_j$ on $\M$, i.e. it is the arclength of the shortest curve on $\M$ connecting $\vx$ and $\vx_j$. The uniformity of the data is also important - this is measured by the \emph{mesh ratio} $\rho:=h/q$.

Every positive definite kernel is the reproducing kernel for a space of continuous functions, often called the \emph{native space} (see  \cite[Chapter 13]{Fasshauer:2007} or \cite[Chapter 10]{Wendland:2004} for details). We will denote the native space of a scalar kernel $\phi$ with domain $\Omega\subseteq\R^d$ by $\mathcal{N}_{\phi}(\Omega)$. When determining the approximation power of a given kernel, there is a rich theory to draw from when the target function lies within $\mathcal{N}_\phi(\Omega)$. 

The native spaces associated with all of the kernels we consider here are subsets of, and sometimes equal to, Sobolev spaces. We denote by $H^{t}(\R^d)$ the space of functions whose distributional derivatives up to order $t$ are all in $L_{2}(\R^d)$, with the usual norm. We will ultimately consider Sobolev spaces on a manifold $\M$, where $\M$ is a $2$-dimensional smooth embedded submanifold of $\R^3$. We denote by $H^t(\M)$ the space of functions on $\M$ whose projections onto $\R^{2}$ by coordinate charts are all in the Sobolev space $H^t(\R^{2})$ (see \cite[Section 2]{FuselierWright2011}). Since our manifold is embedded in $\R^3$, a natural choice for vector-valued functions on $\M$ is $\mathbf{H}^{t}(\M):= (H^t(\M))^3$. The norm for a function $\mathbf{f} = (f^x,f^y,f^z)\in \mathbf{H}^{t}(\M)$ is given by 
\[\|\mathbf{f}\|_{\mathbf{H}^{t}(\M)} := \left(\|f^x\|_{H^{t}(\M)}^2 + \|f^y\|_{H^{t}(\M)}^2 + \|f^z\|_{H^{t}(\M)}^2\right)^{1/2}.\]

The kernels that we use throughout the paper give rise to spaces of continuous Sobolev functions. Thus from this point on we assume that for some $\tau > 3/2$ the kernel $\phi$ satisfies
\begin{align}
\widehat{\phi}(\omega)\leq C(1 + \|\omega\|^2)^{-\tau},\label{eq:fastdecay}
\end{align}
where $\widehat{\phi}$ is the Fourier transform of $\phi$ on $\R^3$ and $C$ is a constant.\footnote{The condition $\tau > 3/2$ ensures that functions within the kernel's native space are continuous on $\R^3$.} Radial kernels whose Fourier transforms decay exponentially, such as the Gaussian or IMQ, obviously satisfy this. For finitely smooth kernels, we will require the stronger condition that
\begin{align}
c (1 + \|\omega\|^2)^{-\tau}\leq \widehat{\phi}(\omega)\leq C(1 + \|\omega\|^2)^{-\tau},\label{eq:algdecay}
\end{align}
where $c$ and $C$ are positive constants. Note that in particular the Mat\'ern kernels \eqref{eq:matern} satisfy \eqref{eq:algdecay} since~\cite[p.133]{Wendland:2004}
\begin{align*}
\widehat{\phi}_{\nu}(\omega) = (1 + \|\omega\|^2)^{-\nu}.
\end{align*}
It can be shown that kernels satisfying \eqref{eq:fastdecay} obey $\mathcal{N}_\phi(\M)\subseteq H^{s}(\M)$, and that those satisfying \eqref{eq:algdecay} give $\mathcal{N}_\phi(\M)=H^{s}(\M)$ \cite[Theorem 3.3]{FuselierWright2011}, where $s=\tau - 1/2$. The parameter $s$ governs the smoothness of the native spaces intrinsic to the surface $\M$, so the error estimates that follow will often be given in terms of $s$.  

%%%%%%%%%%%%%%%%%%%%%%%%%%%%%%%%%%%%%%%%%%%%%%%
\subsection{Summary of convergence results}
%%%%%%%%%%%%%%%%%%%%%%%%%%%%%%%%%%%%%%%%%%%%%%%

We need to define continuous analogues to our discrete differential operators. Recall that the operators $G_X$ and $D_X$ are designed by applying surface differential operators to an RBF interpolant. This prompts the following definitions:
\[G_\M f := \nabla_\M I_{\phi}f,\mbox{\hspace{.5in}} D_{\M}\mathbf{f} := \nabla_\M \cdot I_{\Phi}\mathbf{f},\mbox{\hspace{.5in}} L_{\M}f := D_{\M} G_\M f,\]
where $I_\Phi \mathbf{f} = (I_\phi f^x, I_\phi f^y,I_\phi f^z)$ is the vector-valued kernel interpolant of $\mathbf{f}$. For any continuous functions $f$ and $\mathbf{f}$, the functions above sampled on $X$ agree with the associated discrete differential operator acting on the vector $f_X$ or $\mathbf{f}_X$, e.g. $(L_\M f)|_X = L_X f_X$. To determine how well $L_X$ approximates $\Delta_{\M}$, for example, we will derive estimates for the error $\|L_{\M}f - \Delta_{\M}f\|$ over the entire domain.

The error will be measured in the $L_2(\M)$ and $L_\infty(\M)$ norms, and we do this for a few reasons. Even though error is often measured empirically with an $\infty$-type norm, the current theoretical $\infty$-norm error estimates for many kernel approximation problems typically give slower rates than those observed experimentally. The observed $\infty$-norm rates tend to be closer to the $L_2$-type convergence rates, which are often optimal. Thus one can probably expect the faster $L_2(\M)$ convergence rates in practice. With that, here is our first approximation result.\\

\noindent \textbf{Result 1: First-order discrete operators.} Suppose the kernel $\phi$ satisfies \eqref{eq:fastdecay} with $\tau > 3/2 + 1$ and define $s = \tau - 1/2$. Then the following holds for all $f\in N_{\phi}(\M)$ and $\vf \in (N_\phi(\M))^3$.
\begin{eqnarray*}
\|G_\M f-\nabla_{\M} f\|_{L_2(\M)}& \leq &\mathcal{O}(h^{s-1}), \quad
\|G_\M f-\nabla_{\M} f\|_{L_\infty(\M)} \leq \mathcal{O}(h^{s-2}), \\
\|D_{\M}\mathbf{f}-\nabla_{\M}\cdot \mathbf{f}\|_{L_2(\M)}& \leq &\mathcal{O}(h^{s-1}), \quad \|D_{\M}\mathbf{f}-\nabla_{\M}\cdot \mathbf{f}\|_{L_\infty(\M)} \leq \mathcal{O}(h^{s-2}).
\end{eqnarray*}
For a proof, see Proposition \ref{prop:discretegraddivallkernels}, Appendix \ref{proofappendix}. Note that for kernels such as the Gaussian and IMQ, which satisfy \eqref{eq:fastdecay} for all $\tau > 0$, this implies convergence faster than any fixed algebraic order for target functions within the native space.

A key assumption in the above result is that the target function is within $\mathcal{N}_\phi(\M)$. When the kernel is less smooth, we have the ability to provide error rates for many target functions \emph{not} within $\mathcal{N}_\phi(\M)$ - the only requirement is that the targets be in a Sobolev space smooth enough so that first-order derivatives are continuous (since $\text{dim}(\M) = 2$, this means that $f\in H^\beta(\M)$ for some $\beta > 2$). In this case the uniformity of the data sites, measured by the mesh ratio $\rho$, enters the error estimates. Proposition \ref{prop:discretegraddiv} in Appendix \ref{proofappendix} implies the following.\\

\noindent \textbf{Result 2: First-order discrete operators with finitely smooth kernels.} Suppose $\phi$ satisfies \eqref{eq:algdecay} with $\tau > 3/2 + 1$, define $s = \tau - 1/2$, and let $\beta$ satisfy $s\geq \beta > 2$. Then the following holds for all $f\in H^{\beta}(\M)$ and $\vf \in \vH^\beta(\M)$: 
\begin{eqnarray*}
\|G_\M f-\nabla_{\M} f\|_{L_2(\M)}& \leq &\mathcal{O}(h^{\beta-1}\rho^{s - \beta}), \quad
\|G_\M f-\nabla_{\M} f\|_{L_\infty(\M)} \leq \mathcal{O}(h^{\beta-2}\rho^{s - \beta}), \\
\|D_{\M}\mathbf{f}-\nabla_{\M}\cdot \mathbf{f}\|_{L_2(\M)}& \leq &\mathcal{O}(h^{\beta-1}\rho^{s - \beta}),  \quad \|D_{\M}\mathbf{f}-\nabla_{\M}\cdot \mathbf{f}\|_{L_{\infty}(\M)} \leq \mathcal{O}(h^{\beta-2}\rho^{s - \beta}).
\end{eqnarray*}
In addition, for very smooth target functions\footnote{See the discussion preceeding Proposition \ref{prop:discretegraddiv} in Appendix \ref{proofappendix} for details.} these kernels give the following error rates:
\begin{eqnarray*}
\|G_\M f-\nabla_{\M} f\|_{L_2(\M)}& \leq &\mathcal{O}(h^{2s-1}), \quad
\|G_\M f-\nabla_{\M} f\|_{L_\infty(\M)} \leq \mathcal{O}(h^{2s-2}), \\
\|D_{\M}\mathbf{f}-\nabla_{\M}\cdot \mathbf{f}\|_{L_2(\M)}& \leq &\mathcal{O}(h^{2s-1}),  \quad \|D_{\M}\mathbf{f}-\nabla_{\M}\cdot \mathbf{f}\|_{L_\infty(\M)} \leq \mathcal{O}(h^{2s-2}).
\end{eqnarray*}

Finding bounds on the discrete Laplacian is less straightforward - the repeated process of ``interpolate, then differentiate'' complicates matters. This will ultimately result in more factors of $\rho$ in the estimates, even in the case of very smooth target functions. Unfortunately error bounds for the discrete Laplacian constructed from infinitely smooth kernels are currently unavailable. However, we will be able to provide high-order convergence rates in the discrete Laplacian for any kernel satisfying \eqref{eq:algdecay}. The following comes from Theorems \ref{theorem:discretelaplaceest_rough} and \ref{theorem:discretelaplaceest_smooth}, Appendix \ref{proofappendix}. \\

\noindent \textbf{Result 3: Discrete Laplacian with finitely smooth kernels.} Suppose $\phi$ satisfies \eqref{eq:algdecay} with $\tau > 3/2 + 2$, define $s = \tau - 1/2$, and let $\beta$ satisfy $s\geq \beta > 3$. Then the following holds for all $f\in H^{\beta}(\M)$ and $\vf \in \vH^\beta(\M)$:
\begin{eqnarray*}
\|L_{\M}f-\laps f\|_{L_2(\M)}& \leq &\mathcal{O}(h^{\beta-2}\rho^{2(s - \beta)+1}),
 \; \|L_{\M}f-\laps f\|_{L_\infty(\M)} \leq \mathcal{O}(h^{\beta-3}\rho^{2(s - \beta)+1}).
\end{eqnarray*}
In addition, for very smooth target functions\footnote{See Theorem \ref{theorem:discretelaplaceest_smooth} for details.} these kernels give the following error rates:
\begin{eqnarray*}
\|L_{\M}f-\laps f\|_{L_2(\M)}& \leq &\mathcal{O}(h^{2s-2}\rho),\quad \|L_{\M}f-\laps f\|_{L_\infty(\M)} \leq \mathcal{O}(h^{2s-3}\rho).
\end{eqnarray*}

For solving time-dependent PDEs, one often is only concerned with approximating derivatives at the nodes since typically these are the only locations the numerical solution is known. If our point set $X$ is always chosen in such a way that $\rho$ is bounded by a fixed constant and $f$ is any target function in a Sobolev space smooth enough to guarantee continuous second-order derivatives (that is, $f\in H^\beta(\M)$ with $\beta > 3$), Result 3 guarantees that 
\[ \|L_X f_X -(\Delta_\M f)|_{X}\|_{l_\infty} \leq \|L_{\M}f - \Delta_\M f \|_{L_\infty(\M)} \leq \mathcal{O}(h^{\beta - 3})\rightarrow 0.\]
%Thus the discrete Laplacian applied to $f_X$ will converge to $(\Delta_\M f)|_X$. 
Now that we have established spatial accuracy, we need to discuss the stability of numerical PDE solvers based on these discrete differential operators.

%%%%%%%%%%%%%%%%%%%%%%%%%%%%%%%%%%%%%%%%%%%%%%%%%%%%%%%%%%%%%%%%%%%%%%%%%%%%%%%%%%%%%%%%
\section{Eigenvalue stability} \label{sec:eigen_stability}
%%%%%%%%%%%%%%%%%%%%%%%%%%%%%%%%%%%%%%%%%%%%%%%%%%%%%%%%%%%%%%%%%%%%%%%%%%%%%%%%%%%%%%%
A necessary condition for stability of the MOL approach described in Section \ref{sec:kernel_method} is that the eigenvalues of the differentiation matrices $L_X$ in \eqref{eq:kernel_mol} must be in the stability domain of the ODE solver used for advancing the system in time.  We propose using either fully-implicit backward differentiation formulae (BDF) methods in the case where there is no forcing terms, and semi-implicit BDF methods (SBDF) in the case of forcing terms (see Sections \ref{sec:numerical_results} and \ref{sec:applications} for details).  This means that, at the very least, all eigenvalues must be in the left half plane.   Unfortunately, the construction of the discrete approximation of $\laps$ does not guarantee that this property will hold for $L_X$.  Fortunately however, we have found from numerous numerical experiments that provided the mesh norm $h$ is small enough (i.e. the surface is well-discretized), then it appears that all of the eigenvalues do in fact lie in the left half plane.  To illustrate these observations, in Figures \ref{fig:egvls}(a)--(f) the eigenvalues of $L_X$ are plotted for the surfaces and node sets depicted in Figure \ref{fig:ex_surfaces}.  Results are shown for both the IMQ kernel and the $\nu=7$ Mat\'ern kernel, which is $C^{10}(\R^3)$.  While not presented here, similar results were obtained for the Mat\'ern kernels with $\nu < 7$.   The shape parameters, which are listed in the figure, have been chosen for this experiment so that the condition numbers of \eqref{eq:rbf_lin_sys} are around $10^{10}$.  We find that even when the nodes are quite scattered (as in the case of the Dupin's cyclide and Bretzel2) the eigenvalues all lie in the left half plane in these examples.
%Given the results from the previous section on the convergence of the discrete surface Laplacian to the continuous Laplacian, the numerical results on the eigenvalues makes sense. 

\begin{figure}[t]
\centering
\begin{tabular}{ccc}
\includegraphics[width=0.32\textwidth]{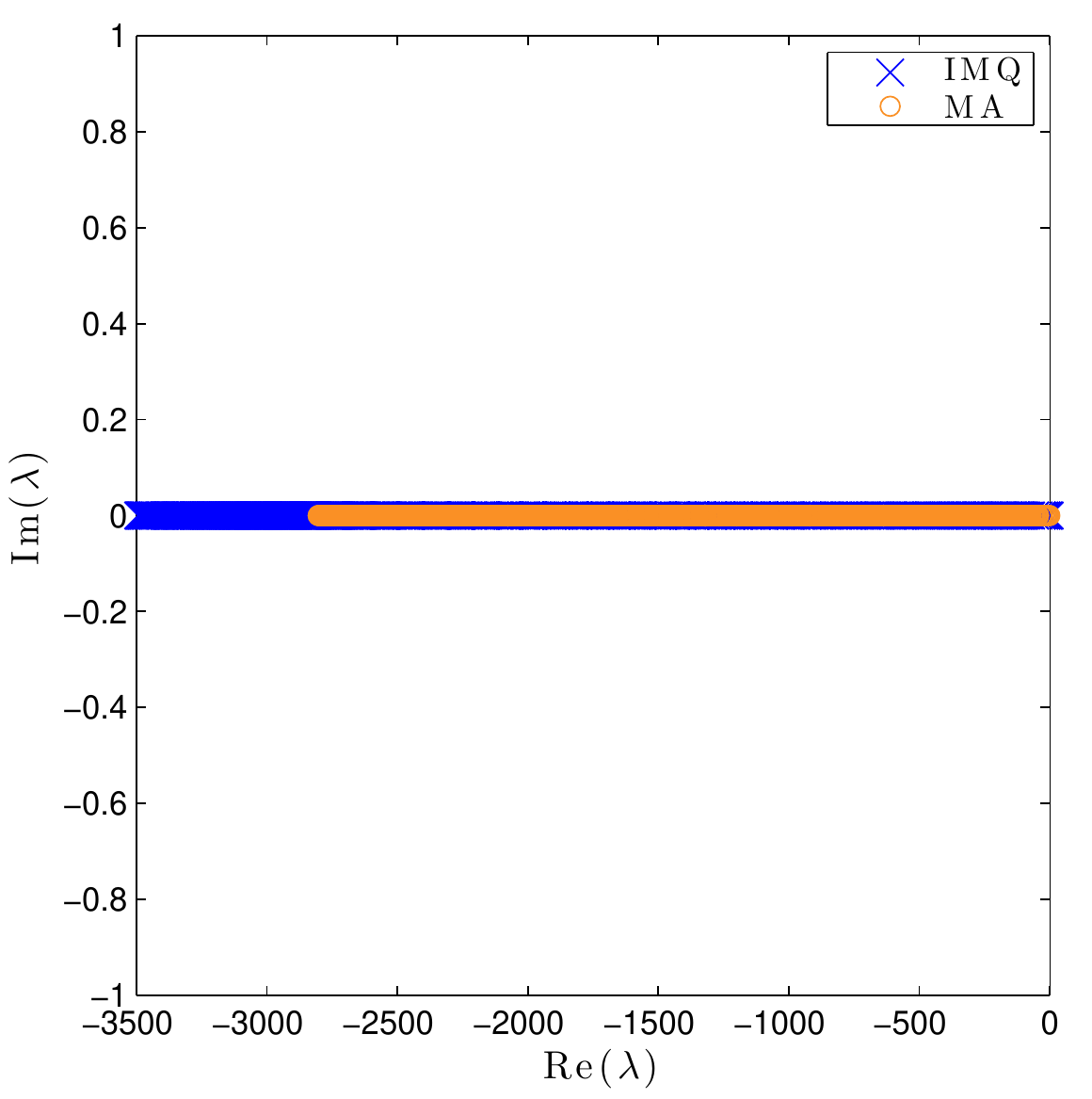} & 
\includegraphics[width=0.32\textwidth]{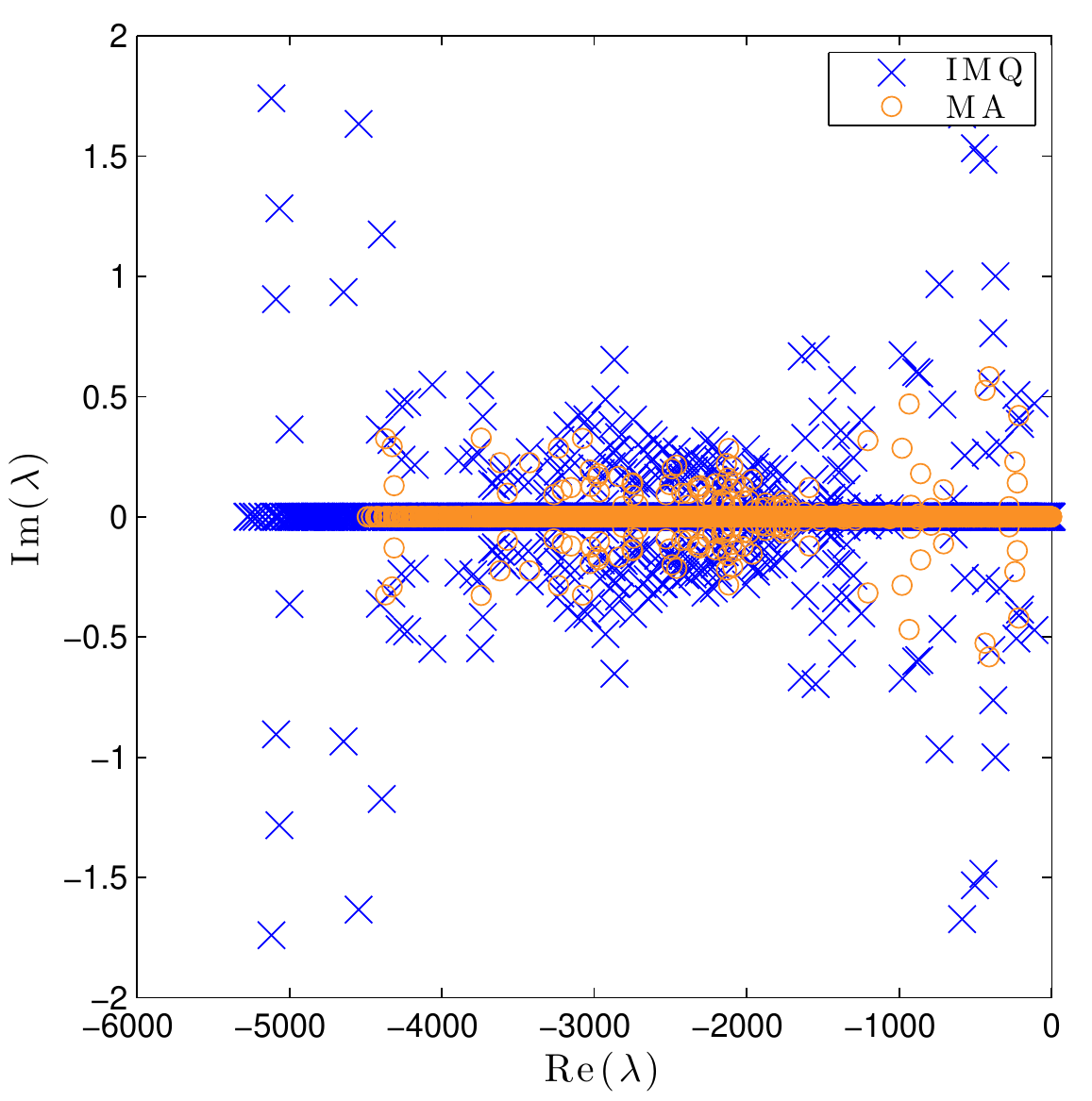} &
\includegraphics[width=0.32\textwidth]{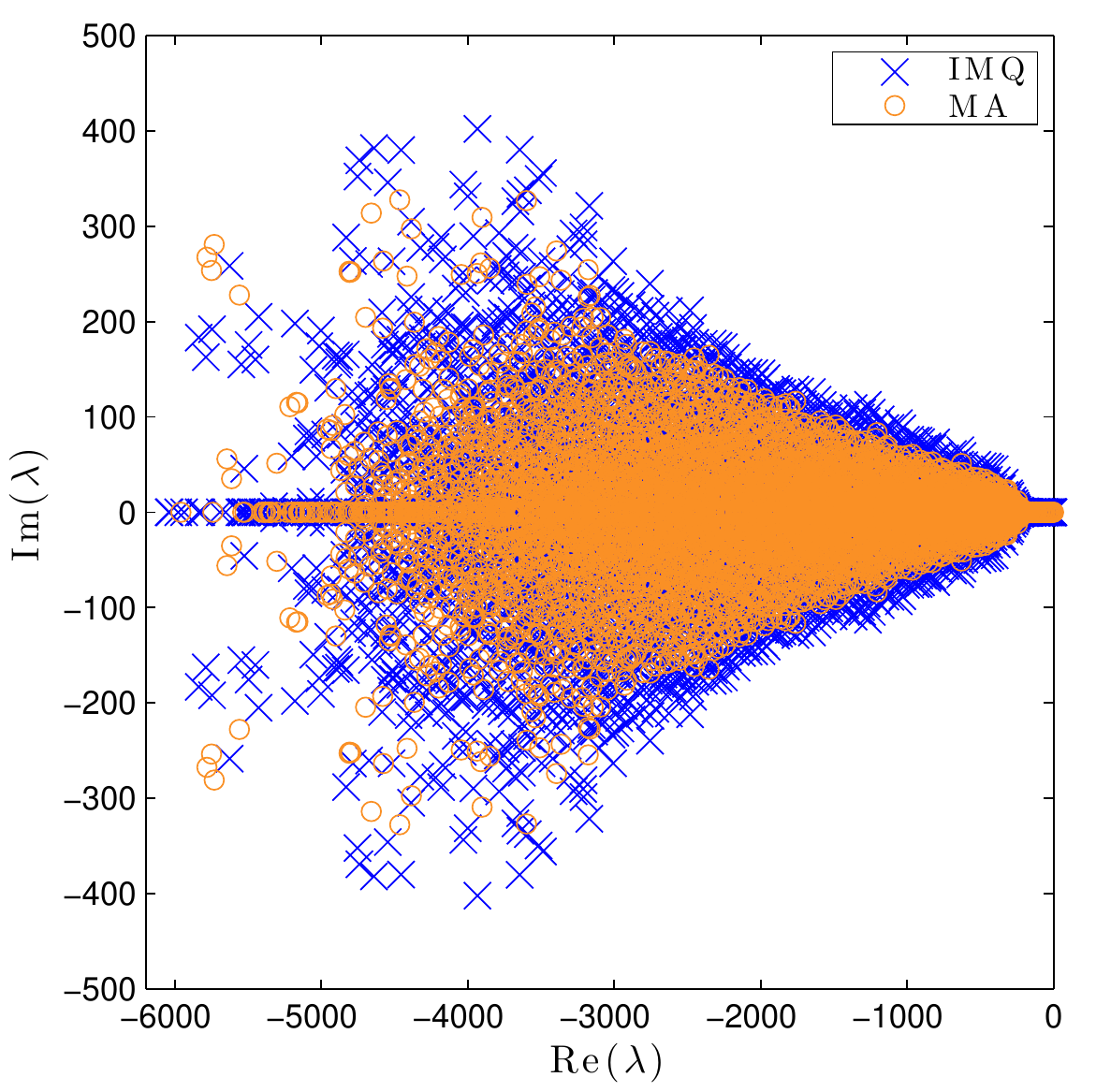}
\\
(a) Unit sphere & (b)  Red blood cell & (c) ``Bumpy sphere'' \\
\includegraphics[width=0.32\textwidth]{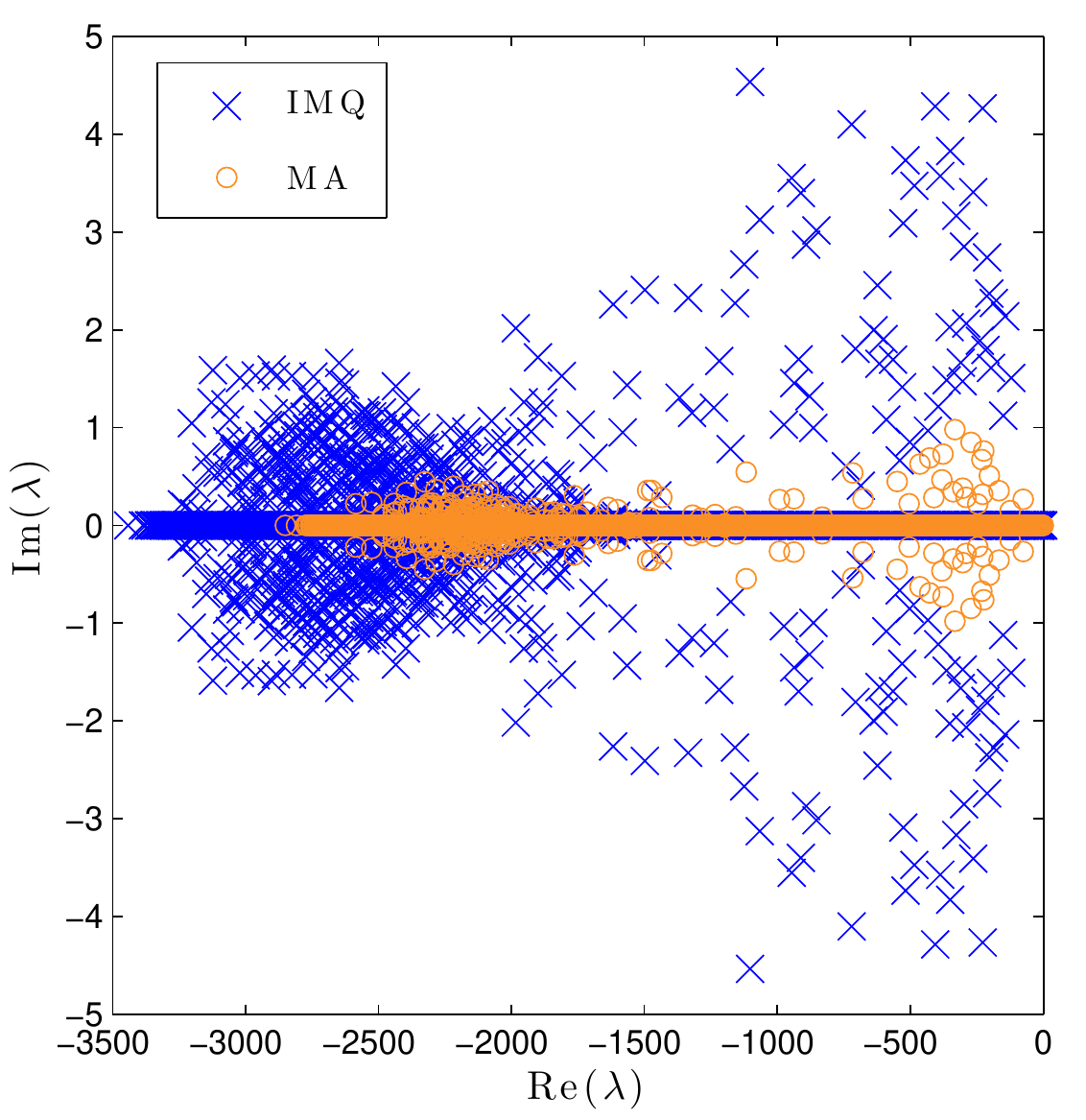} & 
\includegraphics[width=0.32\textwidth]{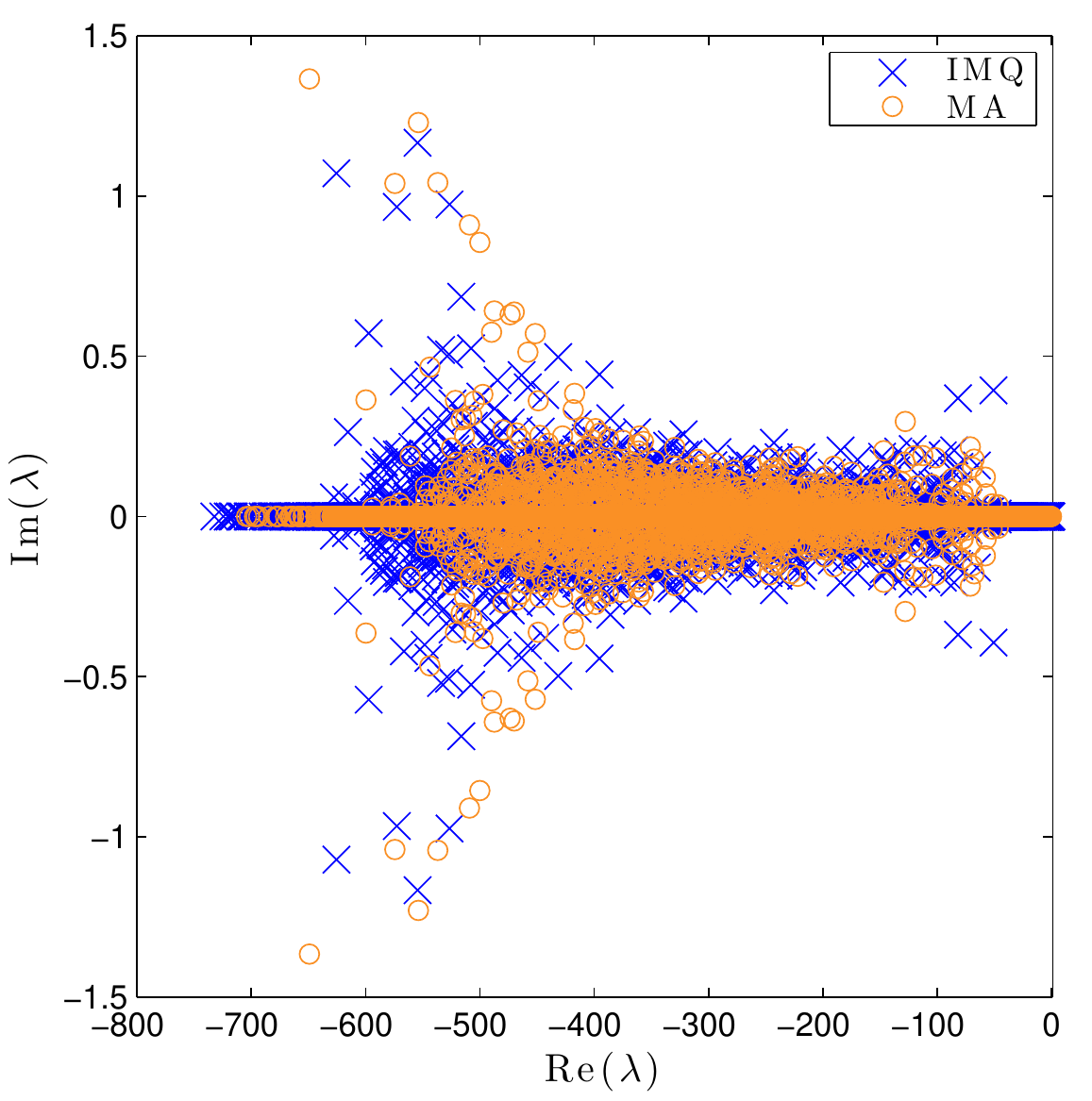} &
\includegraphics[width=0.32\textwidth]{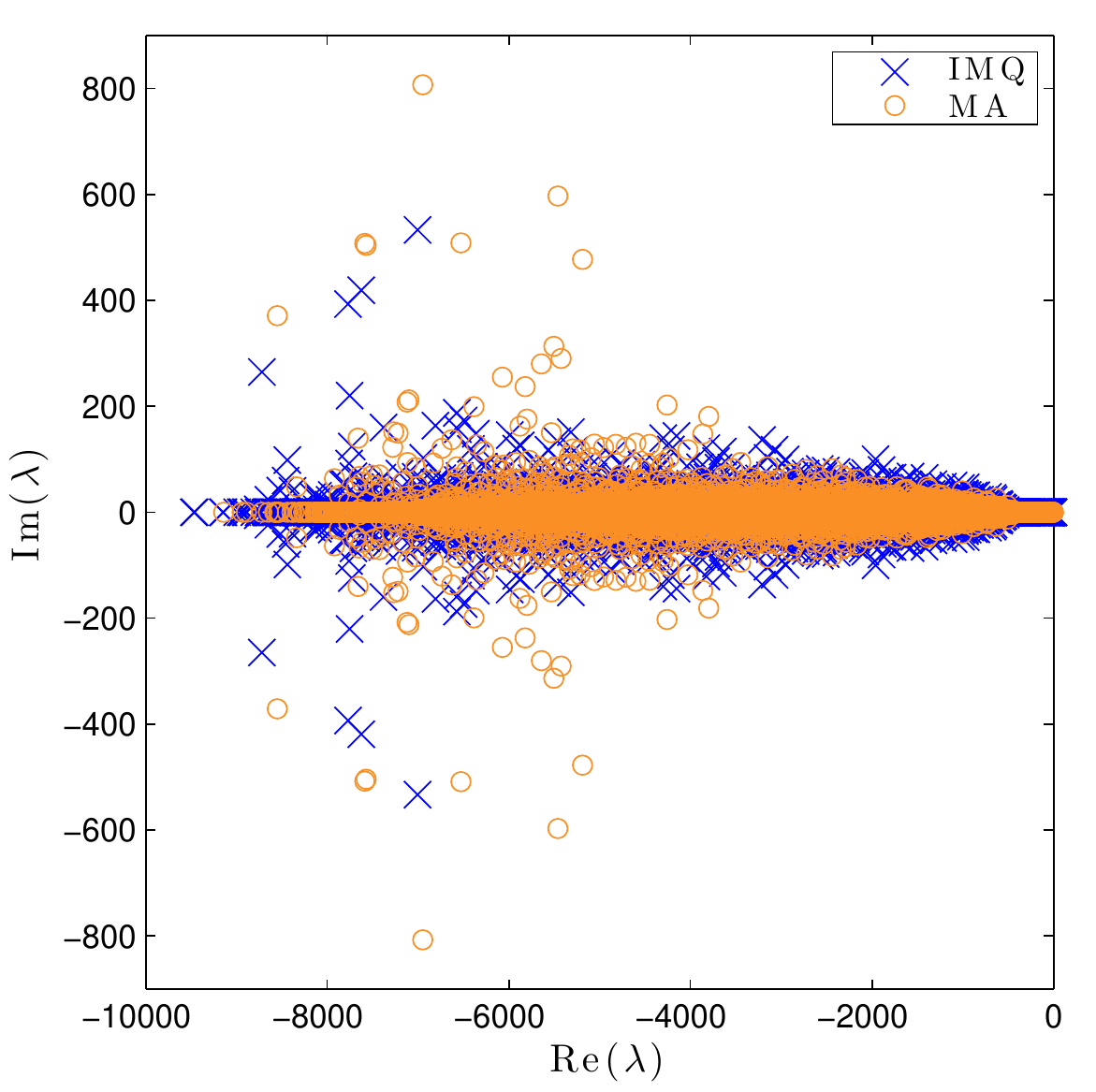}
\\
(d) Torus & (e) Dupin's cyclide & (f) Bretzel2
\end{tabular}
\caption{Eigenvalues of the discrete Laplacian for the surfaces and node sets shown in  Figure \ref{fig:ex_surfaces}.  $\times$'s correspond to the IMQ kernel with shape parameters chosen as (a) $\ep=2.8$, (b) $\ep=4$, (c) $\ep=5$, (d) $\ep=2.8$, (e) $\ep=2$, and (f) $\ep=6.5$, while $\circ$'s correspond to the $\nu=7$ Mat\'ern (MA) kernel with shape parameters chosen as (a) $\ep=8$, (b) $\ep=12$, (c) $\ep=15$, (d) $\ep=8$, (e) $\ep=6$, and (f) $\ep=16$. \label{fig:egvls}}
\end{figure}

For node sets that are not sufficiently dense, we have found cases where some eigenvalues of the differentiation matrix shift over to the right half-plane.  To illustrate this, we repeat the above experiment for the ``bumpy sphere'' with only $N=1035$ nodes instead of $N=5256$.   Figure \ref{fig:unstable_egvls}(a) shows this reduced node set on the ``bumpy sphere'', while part (b) of this figure shows the corresponding eigenvalues of the discrete Laplacian for both the $\nu=7$ Mat\'ern and IMQ kernels.  The shape parameters here were chosen so that the condition numbers of \eqref{eq:rbf_lin_sys} are similar to those from the $N=5256$ experiments.  We can see from this latter figure that the majority of the eigenvalues fall in the left half plane for both kernels, but that there are now a few that fall in the right half plane.   Fortunately, these ``unstable'' eigenvalues appear to correspond to high frequency eigenvectors as illustrated in part (c) of Figure \ref{fig:unstable_egvls} for the Mat\'ern kernel (a similar result holds for the IMQ kernel and the results are thus omitted).  Unphysical eigenmodes of this type also appear in other contexts of kernel methods for PDEs~\cite{FlyerWright07,FlyerWright09,FornbergLehto2011}.   For these applications, Fornberg and Lehto~\cite{FornbergLehto2011} describe a  global hyperviscosity stabilization procedure for selectively shifting the unstable eigenvalues to the left half plane without a reduction in accuracy.   Given the similarity of the unstable eigenmodes in the current problem, it seems reasonable to expect this hyperviscosity procedure to also be applicable in the present context.   In the numerical experiments of this study that follow in the next section, all eigenvalues of discrete Laplacian lie in the left half plane.  We leave exploration of stabilization via hyperviscosity to a separate study.

\begin{figure}[t]
\centering
\begin{tabular}{ccc}
\includegraphics[width=0.30\textwidth]{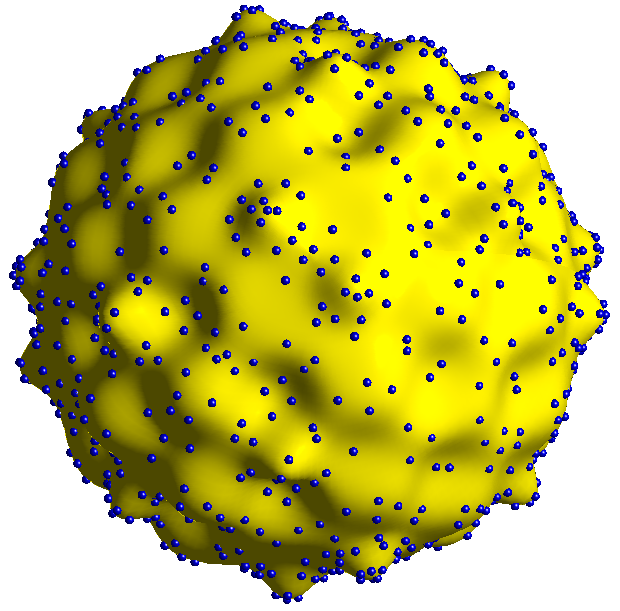} & 
\includegraphics[width=0.38\textwidth]{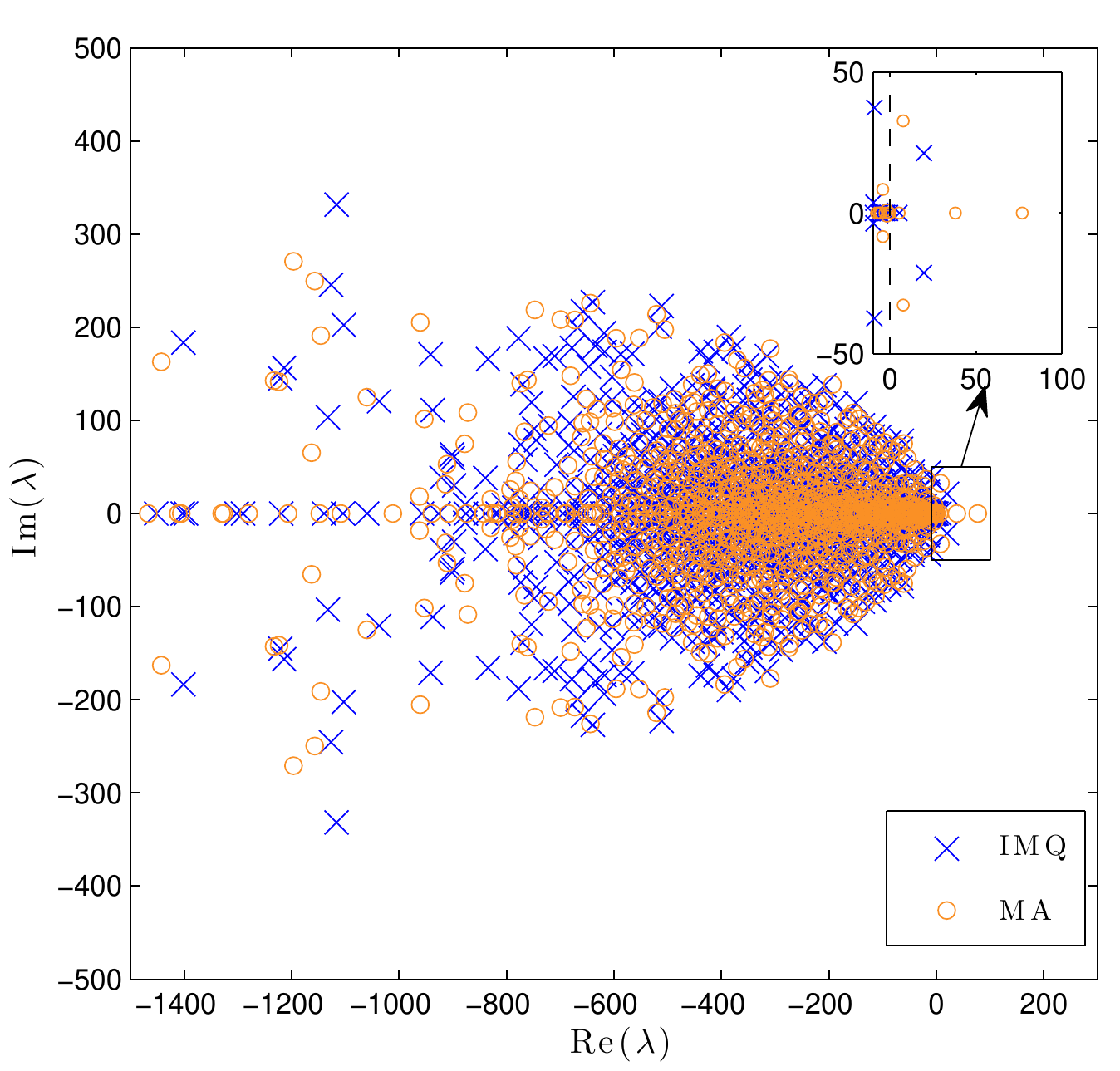} &
\includegraphics[width=0.28\textwidth]{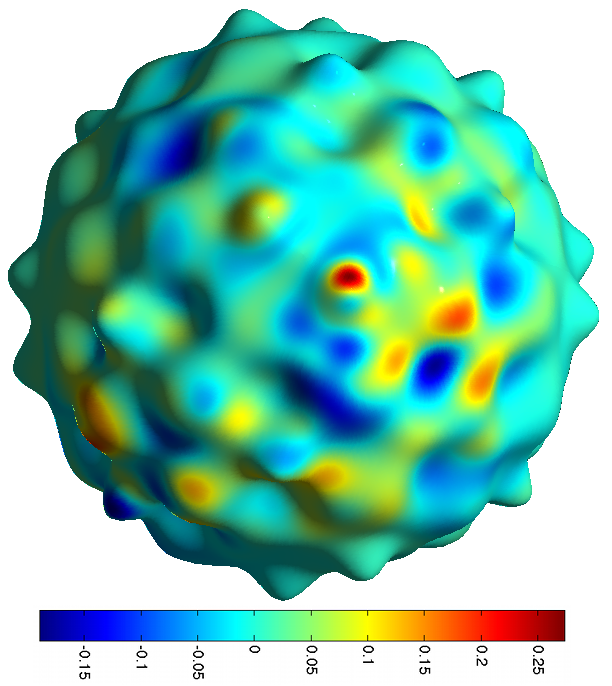}
\\
(a)  & (b)  & (c) 
\end{tabular}
\caption{(a) ``Bumpy sphere'' with $N=1035$ nodes.  (b) Eigenvalues of the discrete Laplacian corresponding to part (a).  $\times$'s correspond to the IMQ kernel with shape parameter $\ep=2$, while $\circ$'s correspond to the $\nu=7$ Mat\'ern (MA) kernel with shape parameter $\ep=6$.  (c) Eigenvector corresponding to the eigenvalue of the discrete Laplacian with the largest real part shown in part (b) using the $\nu=7$ Mat\'ern kernel.\label{fig:unstable_egvls}}
\end{figure}

%%%%%%%%%%%%%%%%%%%%%%%%%%%%%%%%%%%%%%%%%%%%%%%%%%%%%%%%%%%%%%%%%%%%%%%%%%%%%%%%%%%%%%%%
\section{Numerical convergence results}\label{sec:numerical_results}
%%%%%%%%%%%%%%%%%%%%%%%%%%%%%%%%%%%%%%%%%%%%%%%%%%%%%%%%%%%%%%%%%%%%%%%%%%%%%%%%%%%%%%%%
In this section we demonstrate the convergence of our numerical method for approximating the scalar diffusion equation with forcing \eqref{eq:diffusion_scalar} on two surfaces: the unit sphere $\Sphere^2$ and the torus shown in Figure \ref{fig:ex_surfaces}(d).  In all the tests we use the standard fourth-order backward differentiation formulae (BDF4) method for advancing the approximate solution in time.  The time-step is set to $\dt = 10^{-4}$ to ensure that spatial errors dominate over temporal errors. 

Results for three kernels of varying smoothness are presented to illustrate the different convergence rates that are possible.  The first two are the $\nu=4$ and $\nu=6$ kernels from the Mat\'ern family \eqref{eq:matern}, which are $C^{4}(\R^3)$ and $C^{8}(\R^3)$, respectively.  The third kernel is the infinitely smooth IMQ \eqref{eq:imq}.  For the Mat\'ern kernels, the analysis from Section \ref{sec:convergence} applies so we expect an algebraic rate of decay in the approximate solutions.  While we don't have estimates on the approximation of the surface Laplacian for the IMQ kernel, the estimates from Result 1 in Section \ref{sec:convergence} on the discrete surface gradient and divergence suggest that this infinitely smooth kernel should give rates of approximation that are faster than any polynomial order (provided the underlying target is sufficiently smooth).  We fix the shape parameters of the kernels in all the tests as follows: $\ep = 4$ for $\nu=4$ Mat\'ern, $\ep=8$ for $\nu=8$ Mat\'ern, and $\ep=3$ for IMQ.  These values have not been optimized and are chosen simply to illustrate the convergence rates of our method.

Finally, we measure the errors in the numerical solutions using both the discrete two-norm and max-norm.   To define the discrete two-norm we use an approximation of the continuous $L_2(\M)$-norm based on quadrature rules defined at the collocation nodes.  We use the following abuse of notation to denote our discrete two-norm:
\begin{equation*}
\|f\|_{L_2(\M)} := \lp\int_{\M} [f(\vx)]^{2} d\vx\rp^{1/2} \approx \lp\sum_{i=1}^N w_i [f(\vx_i)]^2\rp^{1/2} := \|f\|_{\ell_2},
\end{equation*}
where $\{w_i\}_{i=1}^{N}$ are quadrature weights for the surface integral at the collocation nodes $\{\vx_i\}_{i=1}^{N}$ on the manifold.  These weights are computed using a second-order accurate midpoint rule approximation on the surface.  We use the standard definition for discrete max-norm and denote it with the standard notation of $\ell_{\infty}$.

%%%%%%%%%%%%%%%%%%%%%%%%%%%%%%%%%%%%%%%%%%%%%%%%%%%%%%%%%%%%%
% \subsection{Approximation of surface Laplacian}
%%%%%%%%%%%%%%%%%%%%%%%%%%%%%%%%%%%%%%%%%%%%%%%%%%%%%%%%%%%%%

%%%%%%%%%%%%%%%%%%%%%%%%%%%%%%%%%%%%%%%%%%%%%%%%%%%%%%%%%%%%%
\subsection{Diffusion on the unit sphere with forcing}
%%%%%%%%%%%%%%%%%%%%%%%%%%%%%%%%%%%%%%%%%%%%%%%%%%%%%%%%%%%%%
For the first example, we use a test problem first presented in~\cite{Calhoun:2009}. For this test, an artificial solution to \eqref{eq:diffusion_scalar} is specified (with $\delta = 1$) and the forcing function $f$ is chosen so that this solution is maintained for all time.   The solution is given by
\begin{align}
u(t,\vx) = \exp(-5t) \sum_{k=1}^{23}  \exp\lp-10\arccos(\vxi_k \cdot \vx)\rp \label{eq:sph_ex_solution},
\end{align}
where $\vxi_k$, $k=1,\ldots,23$ are randomly placed points on the surface of the sphere.  The term in the sum is a Gaussian centered at $\vxi_k$ with the distance measured using the geodesic distance.  The solution is clearly $C^{\infty}(\Sphere^2)$ and its value at $t=0$ is displayed in Figure \ref{fig:error_diffusion_sphere}(a).  In our results, the forcing function corresponding to \eqref{eq:sph_ex_solution} is computed analytically and evaluated implicitly in time.  

\begin{figure}[t]
\centering
\begin{tabular}{cc}
\includegraphics[width=0.4\textwidth]{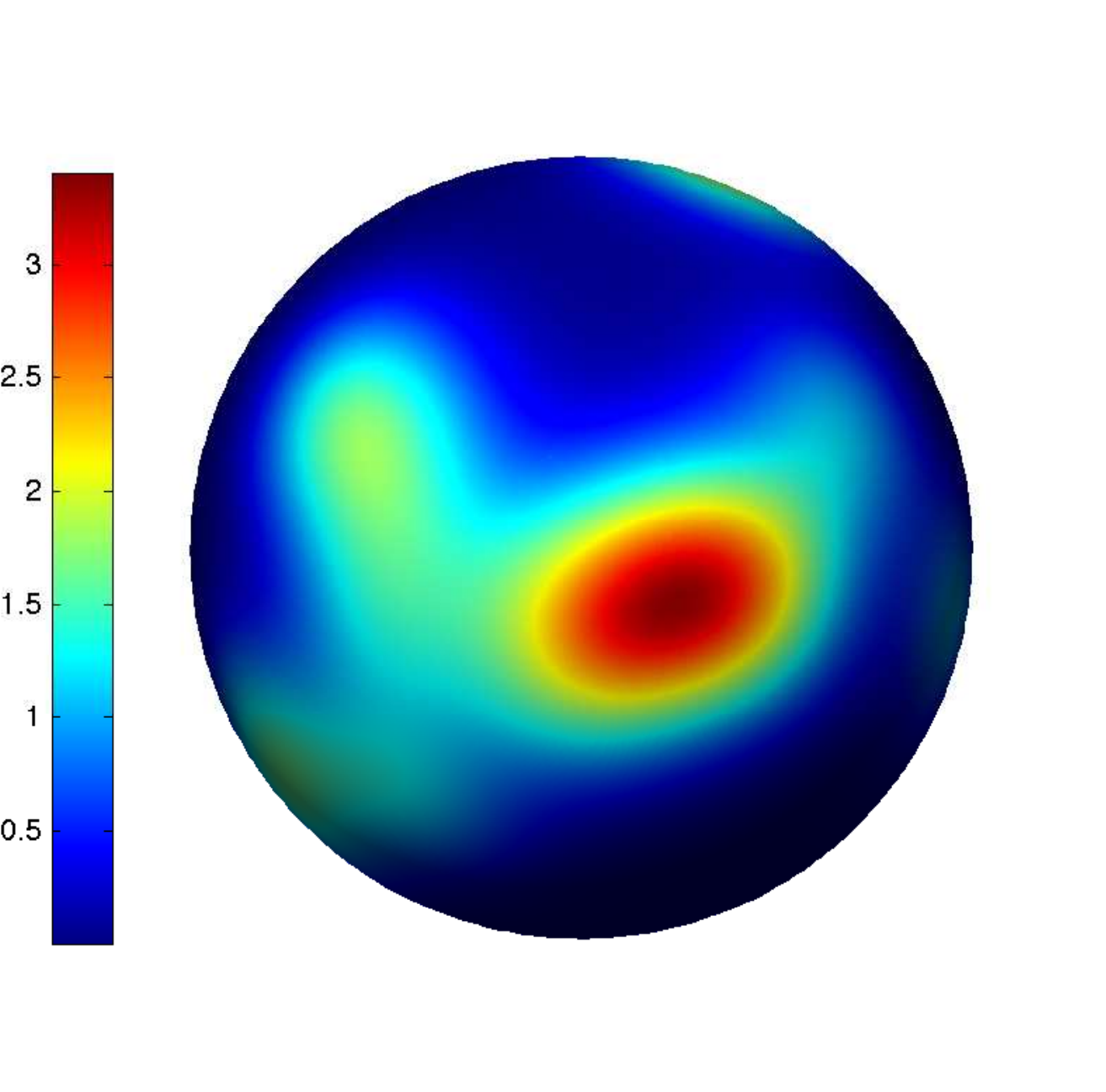} & 
\includegraphics[width=0.4\textwidth]{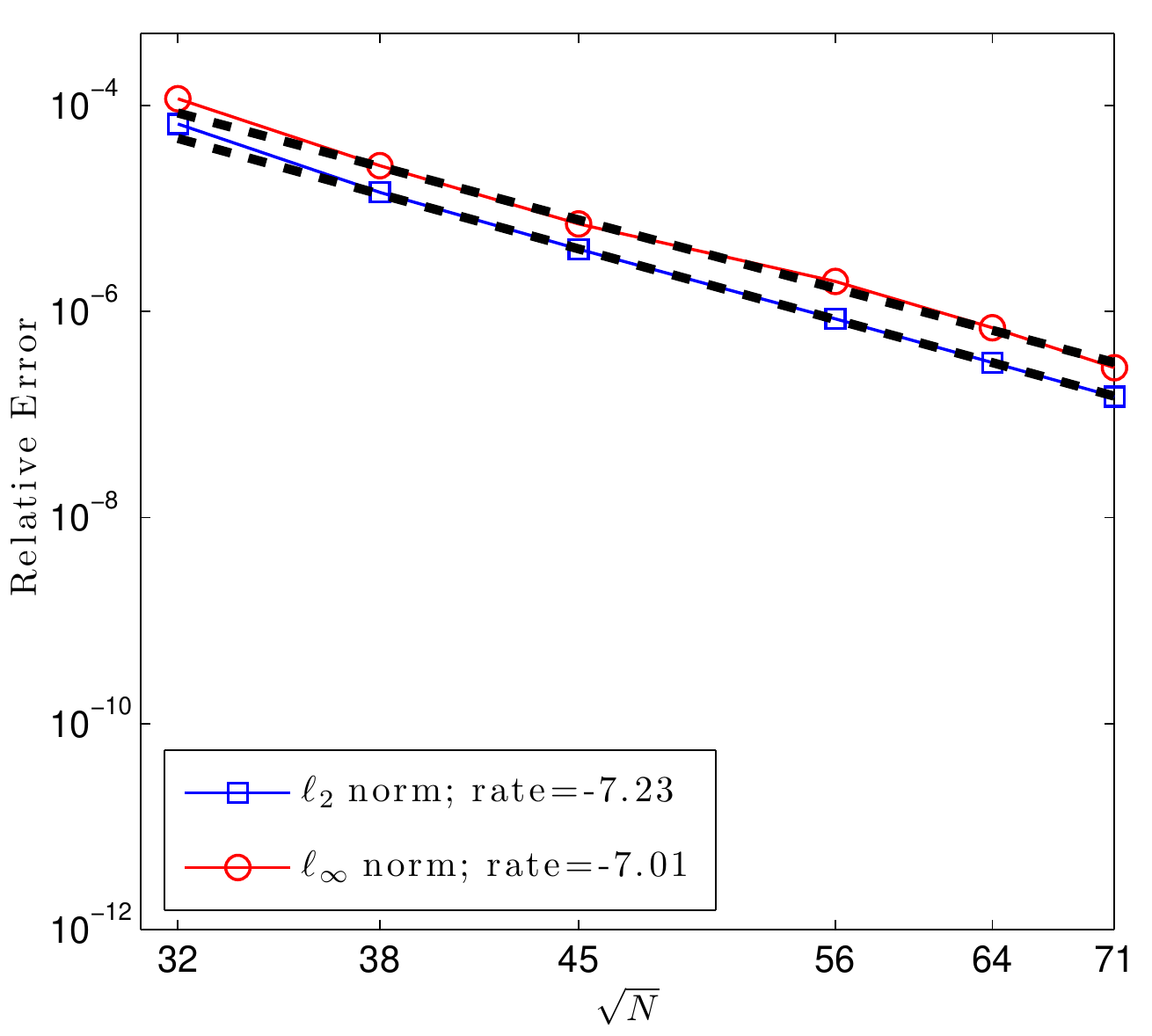} \\
(a) Initial condition & (b) Mat\'ern $\nu=4$ \\ 
\includegraphics[width=0.4\textwidth]{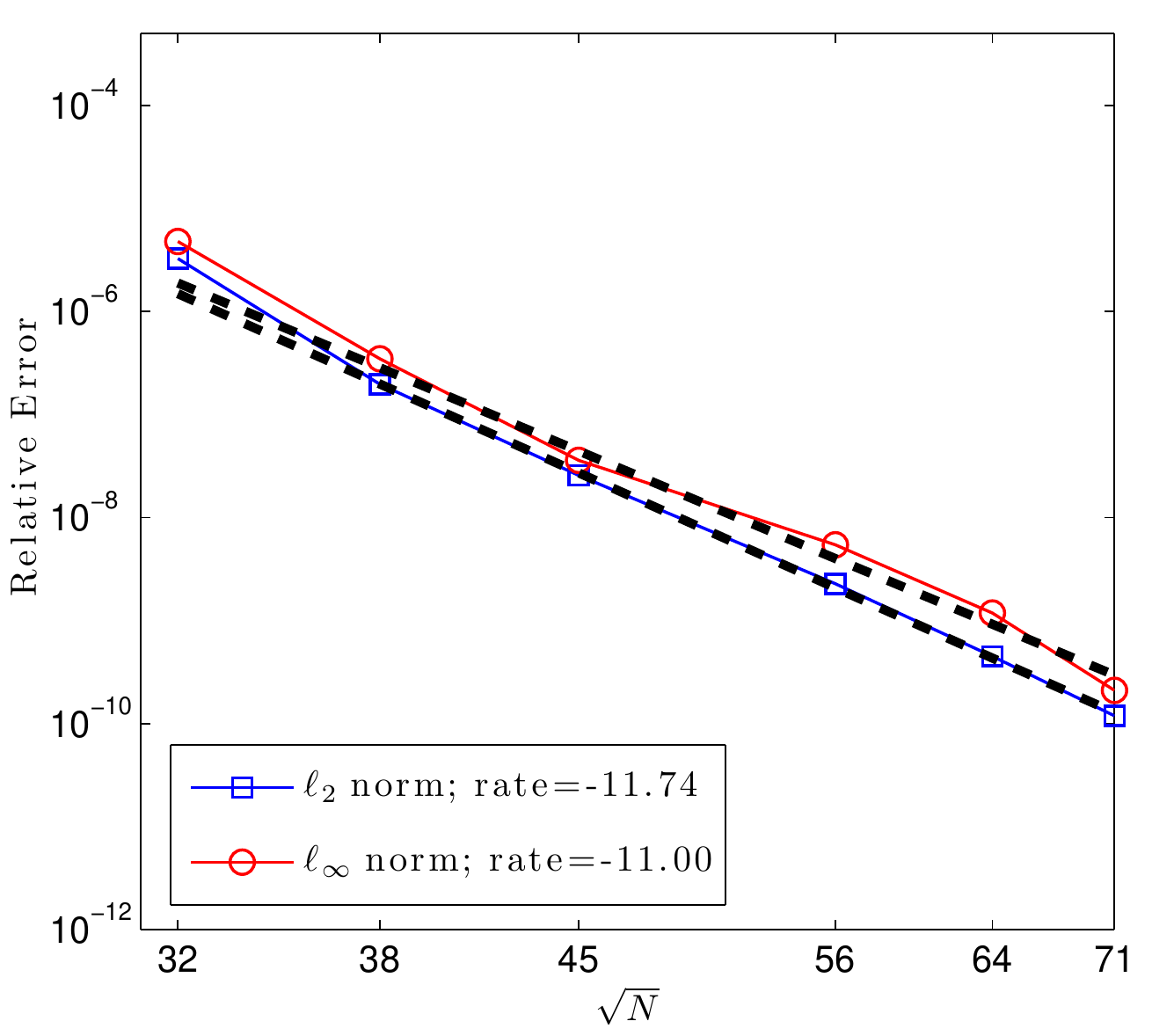} &
\includegraphics[width=0.4\textwidth]{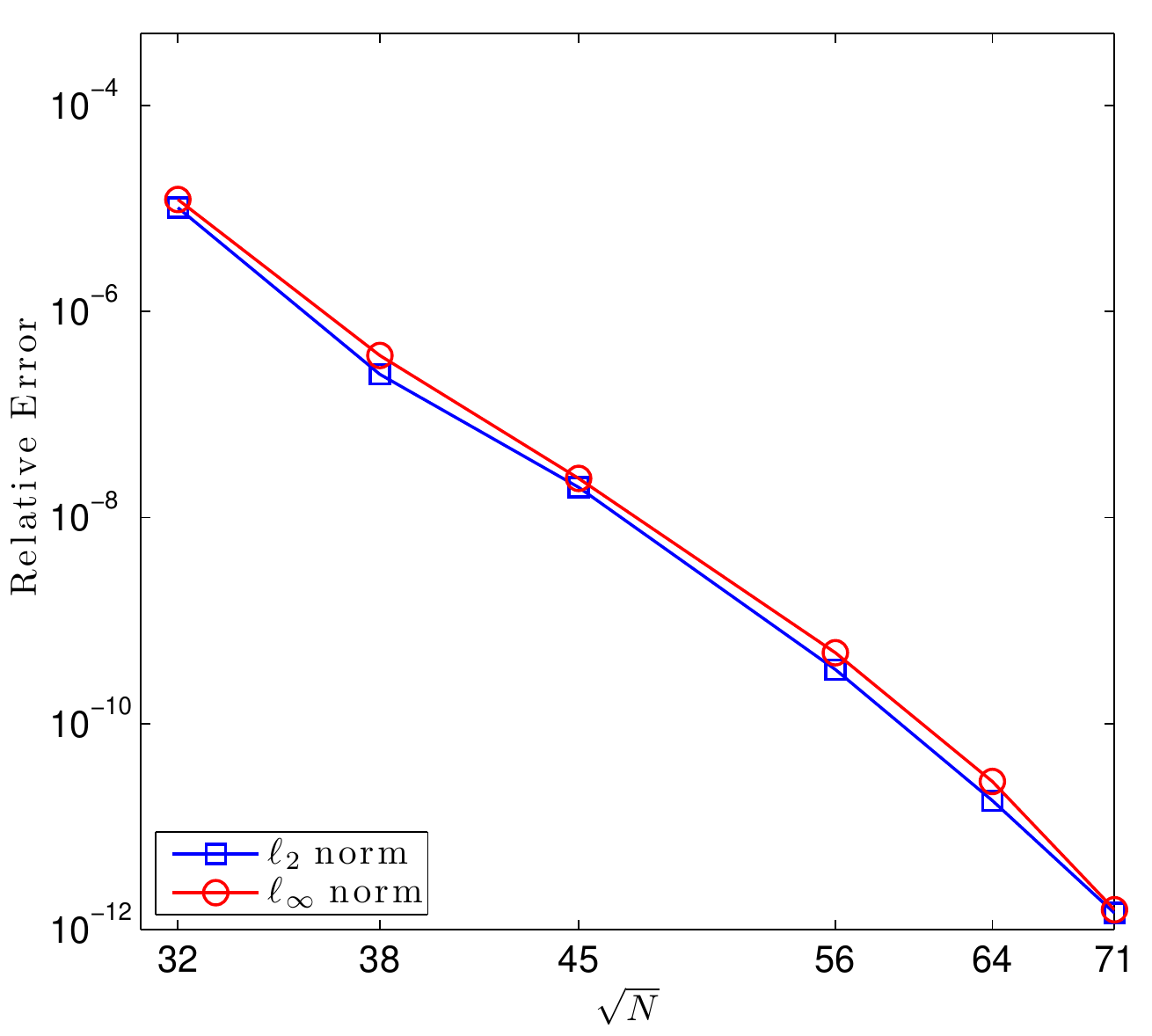} \\
(c) Mat\'ern $\nu=6$  & (d) IMQ 
\end{tabular}
\caption{(a) Initial condition for the forced diffusion problem on the unit sphere.  (b)--(c) Convergence results for the forced diffusion problem for three different radial kernels.  These last three plots are displayed on a log-log scale with the horizontal axis being the square root of the number of collocation nodes $N$, which satisfies $\sqrt{N} \sim 1/h$.   Black dashed lines in (b) and (c) are the lines of best fit to the last 4 values and the slopes of the lines are given in the legends. \label{fig:error_diffusion_sphere}}
\end{figure}

The test calls for comparing the errors in the approximate solution of \eqref{eq:diffusion_scalar} at time $t = 0.2$ at various spatial resolutions.  For our spatial resolutions, we use minimum energy (ME) node sets on the unit sphere discussed in Appendix \ref{apndx:sphere}.  The different sizes of the node sets are $N=1024, 1444, 2025, 3136, 4096$, and $5041$, which are all perfect squares.  Figures \ref{fig:error_diffusion_sphere}(b)--(d) display the results of the convergence study for these node sets and the three respective radial kernels.  Since the mesh norm for the ME nodes satisfies $h \sim 1/\sqrt{N}$, we plot the relative errors vs. $\sqrt{N}$.  All three convergence plots are on a $\log$-$\log$ scale. 

Figures \ref{fig:error_diffusion_sphere}(b) and (c) illustrate the algebraic convergence rate of the approximate solution for the two Mat\'ern kernels.  The black dashed line in these figures shows the lines of best fit to the data and the slope of the different lines is given in the legend.   Since the solution is $C^{\infty}(\Sphere^2)$, we expect the estimates from Result 3 to apply to the approximations of the surface Laplacian in this test problem.  Note that node sets we are using are quasi-uniform, so the mesh ratio that appears in Result 3 will have little or no effect here. For the $\nu=4$ Mat\'ern kernel this result predicts the approximate surface Laplacian to converge like $\bO(h^5)$ in the $\ell_2$-norm and $\bO(h^4)$ for the $\ell_{\infty}$-norm, while for the $\nu=6$ Mat\'ern kernel these convergence rates should be $\bO(h^9)$ and $\bO(h^8)$.  We see from Figures \ref{fig:error_diffusion_sphere}(b) and (c), however, that the measured rates of convergence for the numerical solutions to the forced diffusion problem \eqref{eq:diffusion_scalar} are much higher than the predicted rates for the surface Laplacian.  One reason for these higher rates may be that the estimates from Result 3 apply to the entire manifold, while we are only measuring the errors at the collocation nodes (which are the only locations where the approximate solution is known).  We may therefore be benefitting from some form of super-convergence, which is known to occur at the nodes in certain periodic spline methods~\cite{MR972466}.  A more detailed study of this phenomenon is currently under way~\cite{FuselierWrightSuperConverg}.

Figure \ref{fig:error_diffusion_sphere}(d) displays the results for the IMQ kernel.  The results in this plot suggest that the convergence of the approximate solution is converging at a rate faster than any polynomial order.   Further numerical investigations not presented here seem to indicate that the solution is converging at an exponential rate.  This type of convergence has also been observed for collocation methods based on infinitely smooth kernels for the scalar transport and nonlinear shallow water wave equations on the surface of the sphere~\cite{FlyerWright07,FlyerWright09}.

%%%%%%%%%%%%%%%%%%%%%%%%%%%%%%%%%%%%%%%%%%%%%%%%%%%%%%%%%%%%%
\subsection{Diffusion on a torus with forcing}
%%%%%%%%%%%%%%%%%%%%%%%%%%%%%%%%%%%%%%%%%%%%%%%%%%%%%%%%%%%%%
For the second example, we again consider computing the solution of \eqref{eq:diffusion_scalar}, but with the torus shown in Figure \ref{fig:ex_surfaces}(d) and described in Appendix \ref{appndx:torus}.  Like the previous test, we specify the solution and then choose the forcing function so that this solution is maintained for all time.  The solution is given by  
\begin{align}
u(t,\vx) = \frac18 \exp(-5t)x \lp x^4 - 10x^2 y^2 + 5 y^4\rp \lp x^2 + y^2 - 60 z^2 \rp,
\label{eq:torus_test_func}
\end{align}
where $\vx$ is a point on the Torus; see Figure \ref{fig:error_diffusion_torus}(a) for a plot of $u(0,\vx)$.  We compute the forcing function analytically and evaluate it implicitly in the BDF4 scheme.  A key ingredient to computing the forcing function is determining the surface Laplacian of $u(t,\vx)$.  For convenience, we provide the equation for this below:
\begin{multline}
\laps u(t,\vx) =  -\frac{3}{8\rho^2}\exp(-5t)\left[x (x^4 - 10 x^2 y^2 + 5 y^4)\right.\\
\left.(10248 \rho^4 - 34335 \rho^3 + 41359 \rho^2 - 21320 \rho  + 4000)\right],
\end{multline}
where $\rho = \sqrt{x^2 + y^2}$.

As in the previous test, we compare the errors in the approximate solution of \eqref{eq:diffusion_scalar} at time $t = 0.2$ at various spatial resolutions.  We also use ME node sets on the torus for the spatial discretization as discussed in Appendix \ref{appndx:torus}.  The sizes of the different node sets are $N=500, 750, 1000, 2000, 3000$, and $4000$.  The results of the convergence study for these node sets and the three radial kernels are displayed in Figures \ref{fig:error_diffusion_torus}(b)--(c).  As in the previous test, the mesh norm for the ME nodes on the torus satisfies $h \sim 1/\sqrt{N}$, so we again plot the relative errors vs. $\sqrt{N}$.

\begin{figure}[tbh]
\centering
\begin{tabular}{cc}
\includegraphics[width=0.4\textwidth]{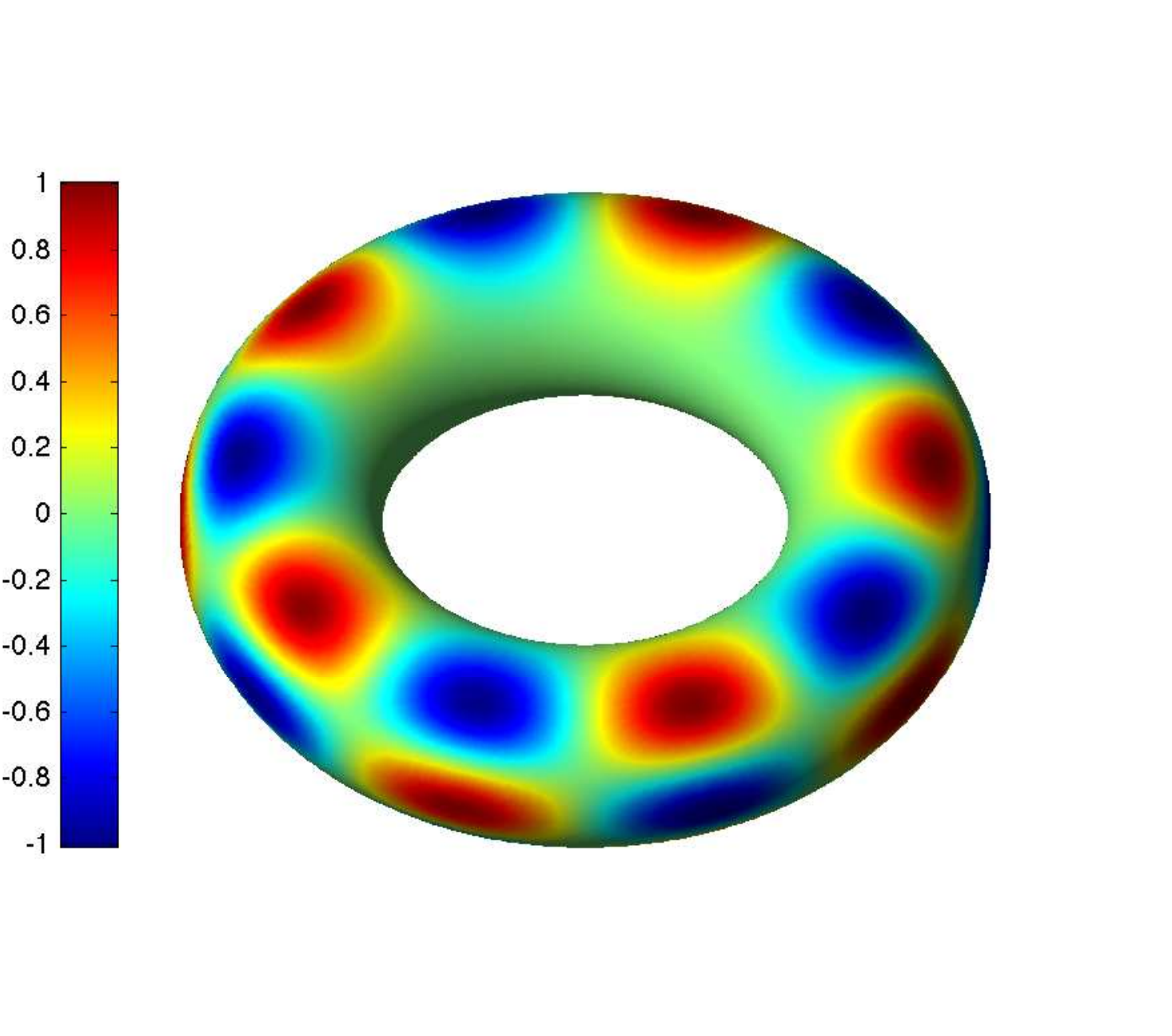} & 
\includegraphics[width=0.4\textwidth]{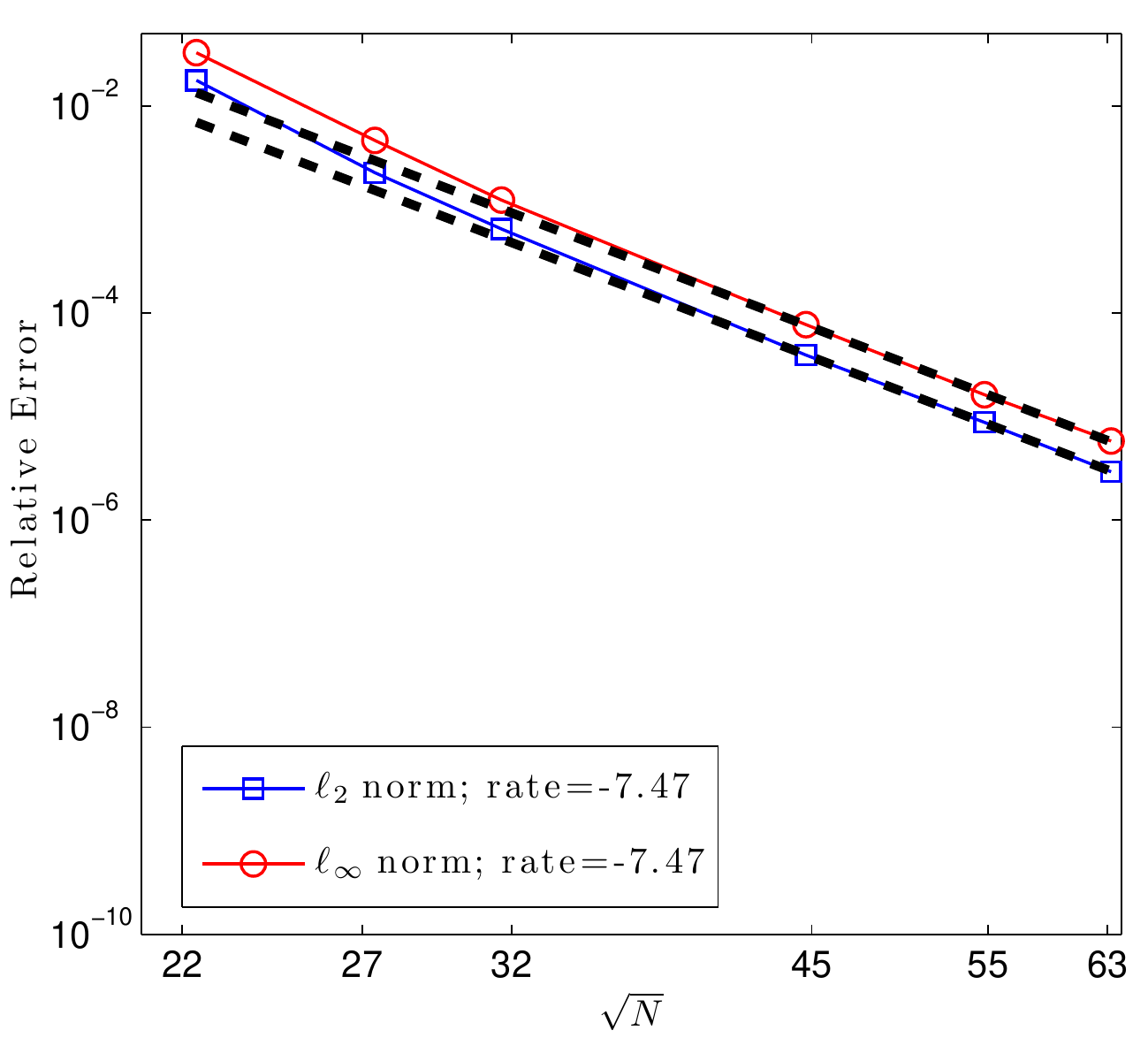} \\
(a) Initial condition & (b) Mat\'ern $\nu=4$ \\ 
\includegraphics[width=0.4\textwidth]{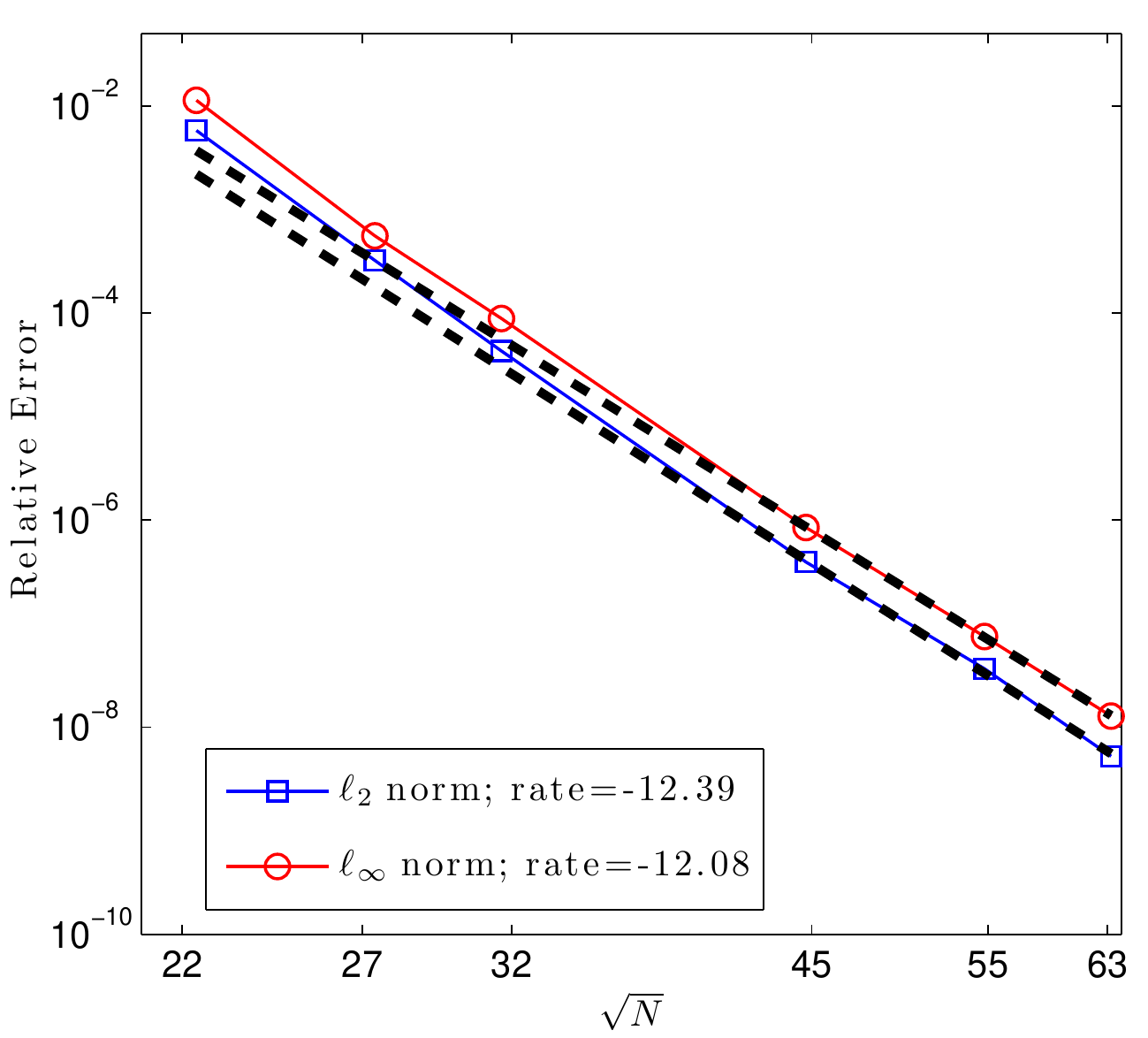} &
\includegraphics[width=0.4\textwidth]{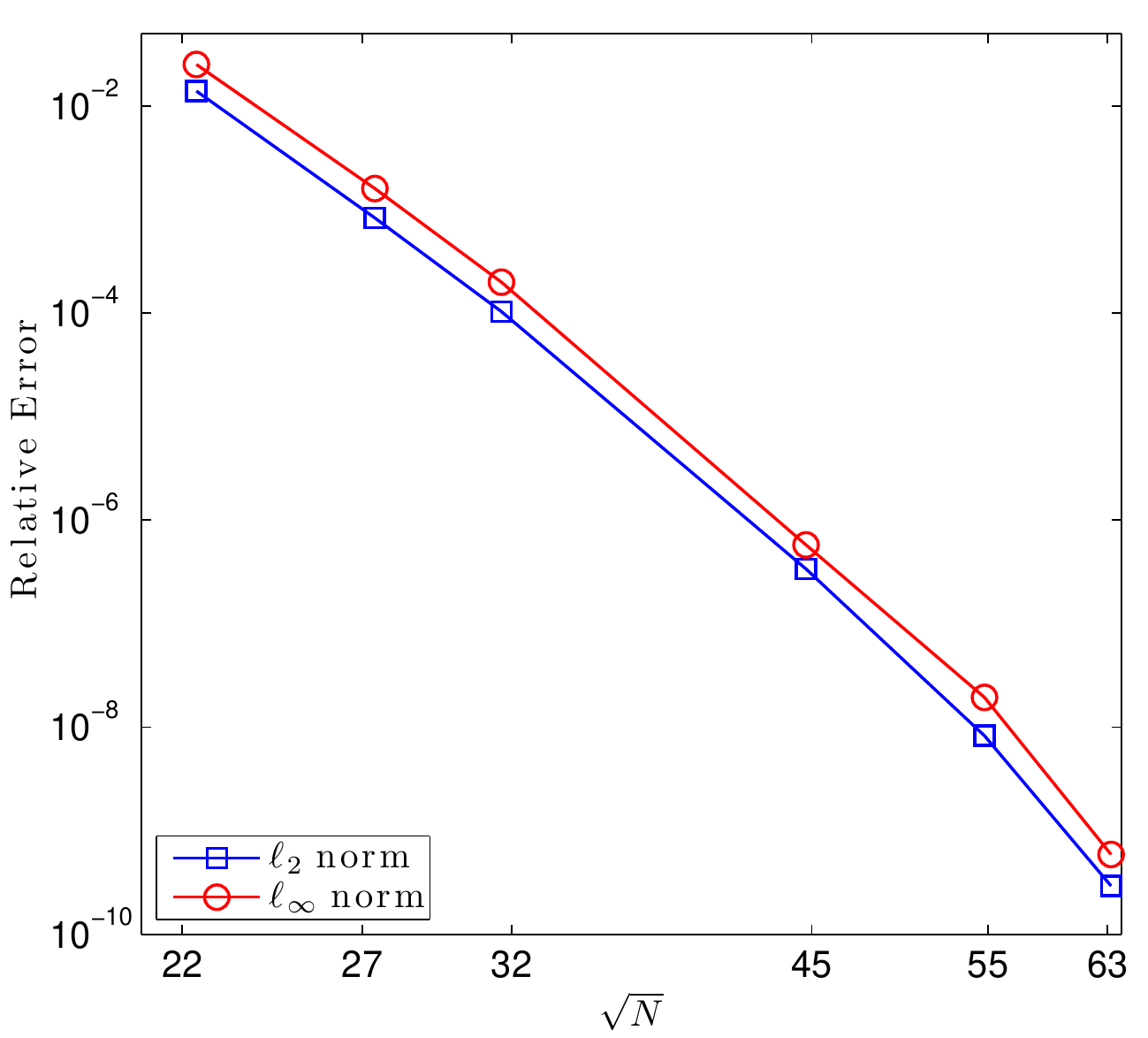} \\
(c) Mat\'ern $\nu=6$  & (d) IMQ 
\end{tabular}
\caption{(a) Initial condition for the forced diffusion problem on a torus.  (b)--(c) Convergence results for the forced diffusion problem for three different radial kernels.  These last three plots are displayed on a log-log scale with the horizontal axis being the square root of the number of collocation nodes $N$, which satisfies $\sqrt{N} \sim 1/h$.   Black dashed lines in (b) and (c) are the lines of best fit to the last 3 values and the slopes of the lines are given in the legends. \label{fig:error_diffusion_torus}}
\end{figure}

Figures \ref{fig:error_diffusion_torus}(b) and (c) display the results for the $\nu=4$ and $\nu=6$ Mat\'ern kernels, respectively.   Like the previous test, we clearly see the algebraic rates of convergence for these kernels, and the measured rates of convergence of the approximate solutions for both kernels are much higher than the predicted rates of convergence for the discrete Laplacian from Result 3.    We again believe these higher rates may be related to our restriction of measuring the errors at the collocation nodes.  Finally, we note that the measured convergence rates for the current test problem and the previous one are comparable, with only a slightly higher observed rate for the current one.

Figure \ref{fig:error_diffusion_torus}(c) displays the results for the IMQ kernel.  The results indicate the approximate solution is converging at a rate higher then any polynomial degree, which is entirely inline with the previous results.  While not presented here, a further analysis of the numerical results indicate that the rate of convergence is again exponential for this kernel.

\section{Applications}\label{sec:applications}
In the applications we consider systems of PDEs of the form \eqref{eq:reaction_diffusion_system}.   We use the third-order, semi-implicit, backward differentiation formulae (SBDF3) method discussed and analyzed in~\cite{AscherRuuthWetton95} for advancing the semi-discrete systems in time.  This scheme treats the diffusion terms implicitly and the (non-linear) reaction terms explicitly.
% and is given by
%\begin{gather}
%\begin{split}
%\%frac{11}{6} \uu^{k+1} - 3\uu^{k} + & \frac{3}{2}\uu^{k-1} - \frac{1}{3}\uu^{k-2} = \\ & \dt \delta_u L_X \uu^{k+1} + \dt \left[3 f_u(t_k,\uu^%{k},\uv^{k}) - 3 f_u(t_{k-1},\uu^{k-1},\uv^{k-1}) + f_u(t_{k-2},\uu^{k-2},\uv^{k-2})\right], \\
%\frac{11}{6} \uv^{k+1} - 3\uv^{k} + & \frac{3}{2}\uv^{k-1} - \frac{1}{3}\uv^{k-2} = \\ & \dt \delta_v L_X \uv^{k+1} + \dt \left[3 f_v(t_k,\uu^{k},%\uv^{k}) - 3 f_v(t_{k-1},\uu^{k-1},\uv^{k-1}) + f_v(t_{k-2},\uu^{k-2},\uv^{k-2})\right],
%\end{split}
%\label{eq:kernel_sbdf3}
%\end{gather}
%where $\;k=2,3,\ldots$, $\dt$ is the time-step, $t_i = i\dt$, and $\uu^{i}$ represents the approximate solution at time $i\dt$.  
We bootstrap the simulation with one step of the first order SBDF followed by one step of the second order SBDF (see~\cite{AscherRuuthWetton95} for details on these schemes).  Upon initial LU decompositions of the linear systems that result from the implicit discretization, each time-step requires $\bO(N^2)$ operations.

We only present results for the IMQ kernel using the same shape parameters as used in the stability tests (refer to the caption of Figure \ref{fig:egvls} for these values).  However, similar results were obtained when using other shape parameters and when using the Mat\'ern kernels.  

\subsection{Turing patterns}
Since Turing's classical paper~\cite{Turing1952} that suggested how certain non-linear models of reaction and diffusion can lead to stable, heterogeneous pattern formations, there has been an explosion of research in reaction-diffusion-type models for various kinds of morphogenesis~\cite{murray2003mathematical,Kondo24092010,MianiOthmer2000}.  Most of the theoretical and numerical investigations of these models have been restricted to flat planar domains, but there is growing interest in studying these models on more general surfaces to understand how curvature may effect pattern formation.  Additionally, there has been interest in the computer graphics community for using Turing models to generate interesting textures and patterns on surfaces~\cite{Turk:1991}.  

\begin{figure}[t!]
\centering
\begin{tabular}{cc}
\includegraphics[width=0.24\textwidth]{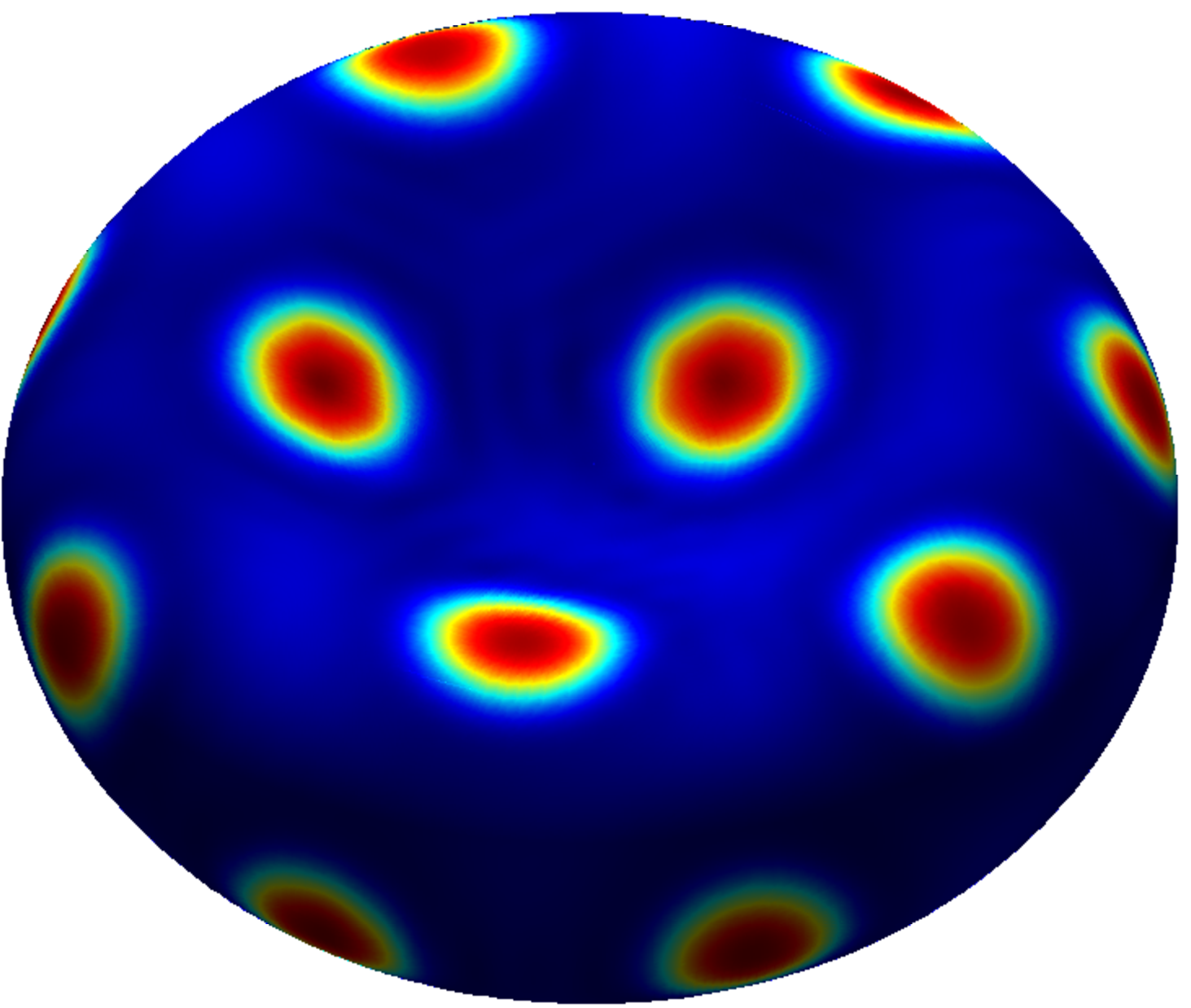} & \includegraphics[width=0.24\textwidth]{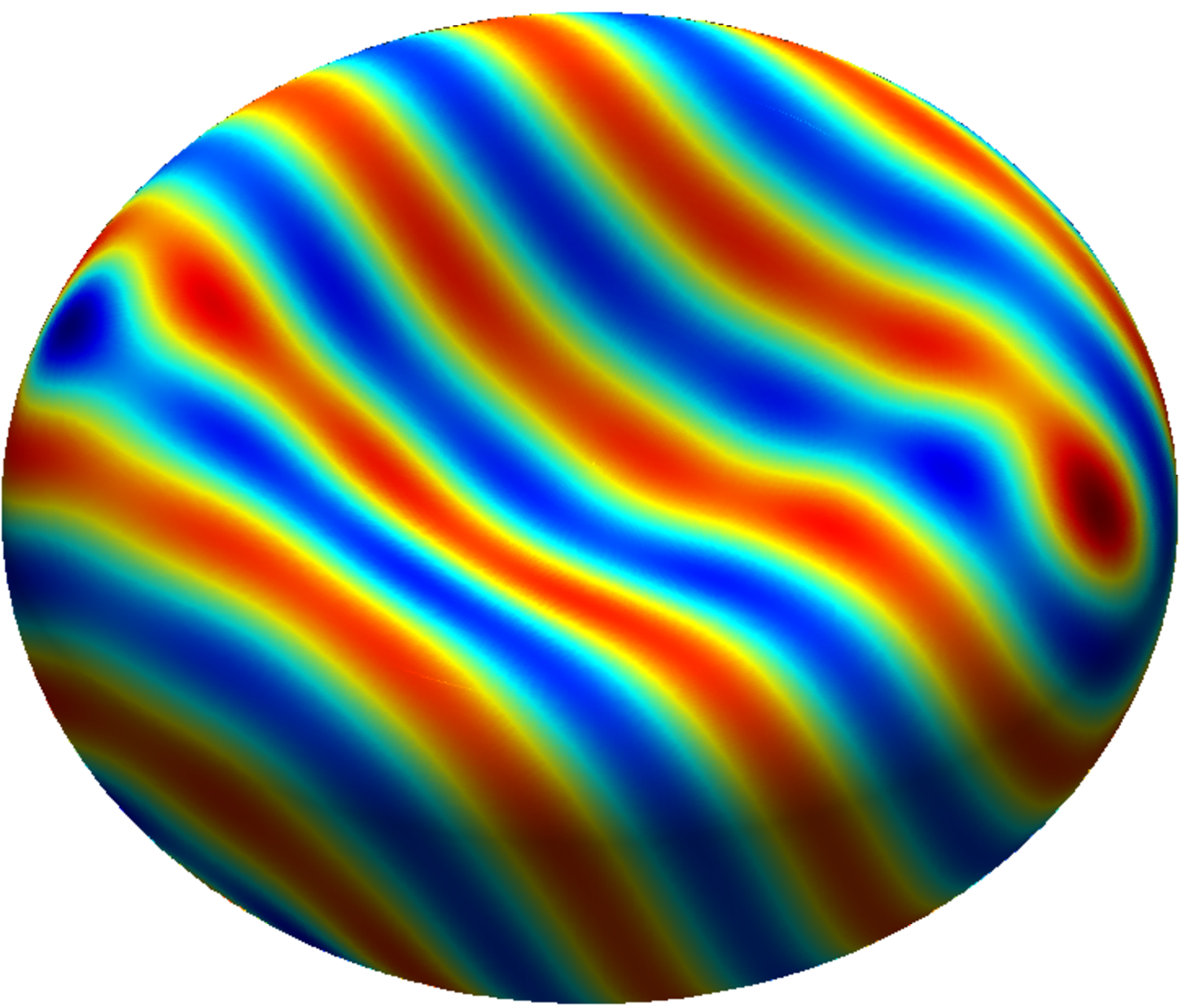} \\
\includegraphics[width=0.26\textwidth]{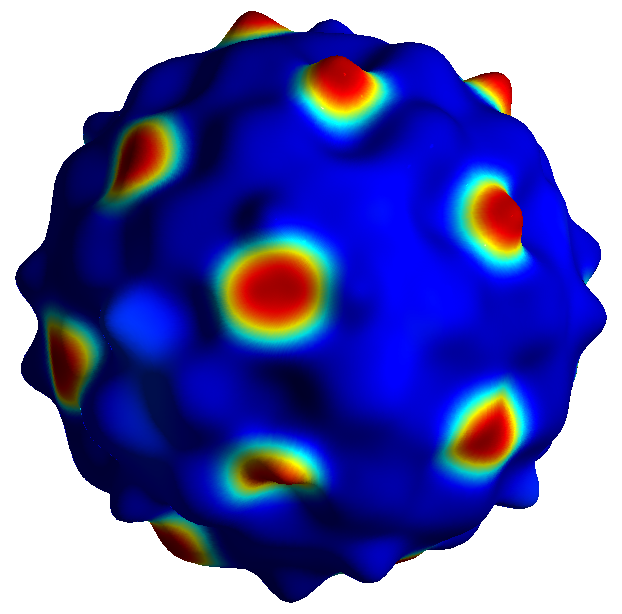} & \includegraphics[width=0.26\textwidth]{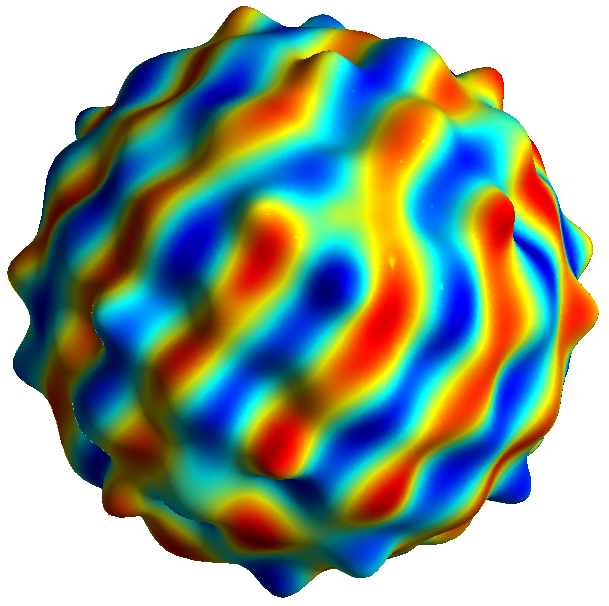} \\
\includegraphics[width=0.28\textwidth]{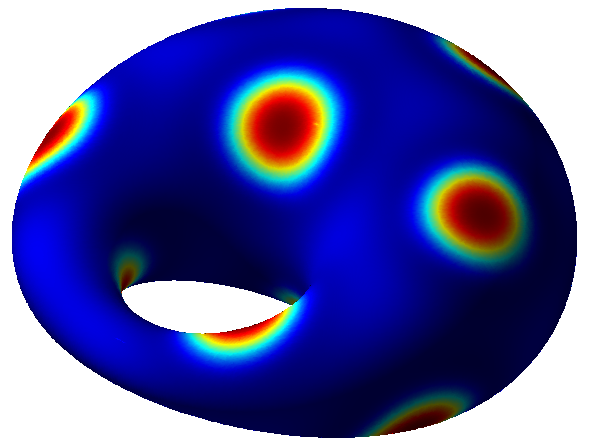} & \includegraphics[width=0.28\textwidth]{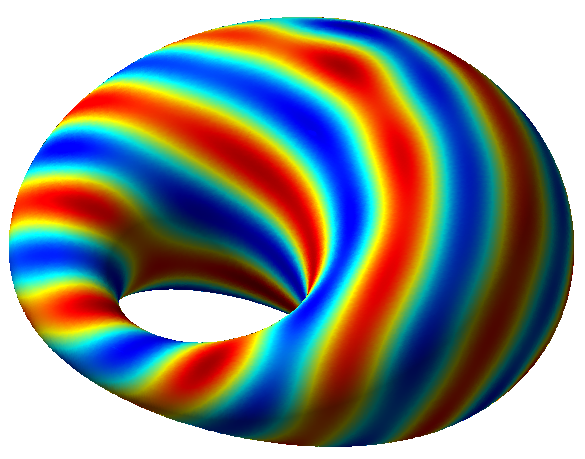} \\ 
\includegraphics[width=0.32\textwidth]{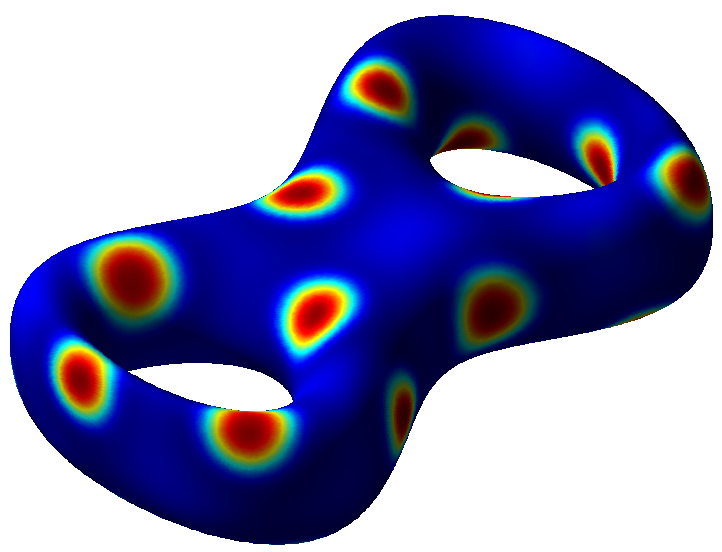} &
\includegraphics[width=0.32\textwidth]{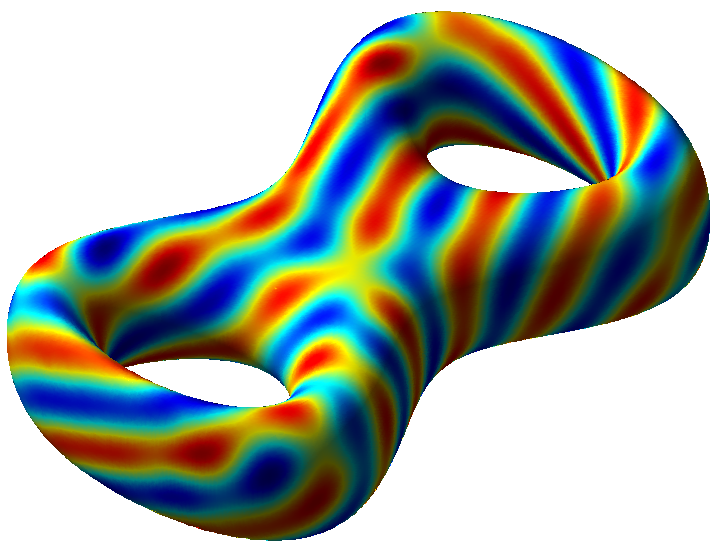}
\end{tabular}
\caption{Turing spot and stripe patterns computed from the model \eqref{eq:turing} on various surfaces.  The pseudocolor plots are for the activator  $u$ once steady state is reached.  In all plots red corresponds to a high concentration of $u$ and blue to a low concentration.  For all simulations the time-step was set to $\dt=0.01$; parameters used to get the different patterns are given in Table \ref{tbl:turing}.\label{fig:turing_spots}}
\end{figure}

We illustrate how our method can be applied to these types of pattern formation problems by applying it to the Turing system (a linearized Brusselator model) from~\cite{BarrioEtAl1999,VareaEtAl1999}:
\begin{align}
\begin{split}
\frac{\partial u}{\partial t} =& \delta_u\laps u + \underbrace{\alpha u (1 - \tau_1 v^2) + v(1 - \tau_2 u)}_{\ds f_u(u,v)}, \\
\frac{\partial v}{\partial t} =& \delta_v\laps v + \underbrace{\beta v \lp 1 + \frac{\alpha\tau_1}{\beta} u v\rp + u(\gamma  + \tau_2 v)}_{\ds f_v(u,v)}. 
\end{split}
\label{eq:turing}
\end{align}
Here $u$ and $v$ are morphogens with $u$ the ``activator'' and $v$ is the ``inhibitor''.  If $\alpha=-\gamma$ then $(u,v)=(0,0)$ is a unique equilibrium point of this system.  By changing the diffusivity rates of $u$ and $v$ an instability can form that leads to different pattern formations.  The cubic coupling parameter $\tau_1$ favors the formation of stripes, while the quadratic coupling parameter $\tau_2$ favors the formation spots~\cite{BarrioEtAl1999}.  The spot pattern formations are more robust than stripes and take far less time to reach ``steady-state''.  The model \eqref{eq:turing} and similar models have been previously used to illustrate the applicability of some of the other numerical methods for PDEs on surfaces~\cite{Calhoun:2009,BergdorfEtAl2010,RuuthMerriman2008,BertalmioEtAl2001,MacDondaldRuuth2009,Turk:1991,Piret2012}.  

Figure \ref{fig:turing_spots} shows the results from our numerical solutions of \eqref{eq:turing} on various surfaces using parameters that lead to both spot and stripe patterns.  Similar to the experiments on the surface of the sphere in~\cite{VareaEtAl1999}, we set the initial values of $u$ and $v$ to random values between $-0.5$ and $0.5$ in a thin strip around the ``equator'' of each surface and $u=v=0$ elsewhere.  Values for all parameters were also motivated from the values used in~\cite{VareaEtAl1999} and are listed in Table \ref{tbl:turing}.   A time-step of $\dt=0.01$ was used in all the simulations and the numerical solutions were computed until steady-state was reached.  The resulting stripe and spot patterns in Figure \ref{fig:turing_spots} are qualitatively similar to those obtained from other numerical methods~\cite{Calhoun:2009,BergdorfEtAl2010,RuuthMerriman2008,BertalmioEtAl2001,MacDondaldRuuth2009,Turk:1991}.

\begin{table}[tbh]
\centering
\begin{tabular}{|c||cccccc|}
\hline
Surface/Pattern & $\delta_v$ & $\alpha$ & $\beta$ & $\gamma$ & $\tau_1$ & $\tau_2$ \\
\hline
\hline
Red blood cell/spots & $4.5\cdot 10^{-3}$ & 0.899 & -0.91 & -0.899 & 0.02 & 0.2 \\
Red blood cell/stripes & $2.1\cdot 10^{-3}$ & 0.899 & -0.91 & -0.899 & 3.5 & 0 \\
Bumpy sphere/spots & $4.5\cdot 10^{-3}$ & 0.899 & -0.91 & -0.899 & 0.02 & 0.2 \\
Bumpy sphere/stripes & $2.1\cdot 10^{-3}$ & 0.899 & -0.91 & -0.899 & 3.5 & 0 \\
Dupin's cyclide/spots & $4.5\cdot 10^{-2}$ & 0.899 & -0.91 & -0.899 & 0.02 & 0.2 \\
Dupin's cyclide/stripes & $1.89\cdot 10^{-2}$ & 0.899 & -0.91 & -0.899 & 3.5 & 0 \\
Bretzel2/spots & $2.1\cdot 10^{-3}$ & 0.899 & -0.91 & -0.899 & 0.02 & 0.2 \\
Bretzel2/stripes & $8.87\cdot 10^{-4}$ & 0.899 & -0.91 & -0.899 & 3.5 & 0 \\
\hline
\end{tabular}
\caption{Values of the parameters of \eqref{eq:turing} used in the numerical experiments shown in Figure \ref{fig:turing_spots}.\label{tbl:turing}  In all cases $\delta_u=0.516\delta_v$.}
\end{table}

\subsection{Spiral waves in excitable media}
\begin{figure}[h!]
\centering
\begin{tabular}{c@{\hspace{0.5in}}c}
\includegraphics[width=0.25 \textwidth]{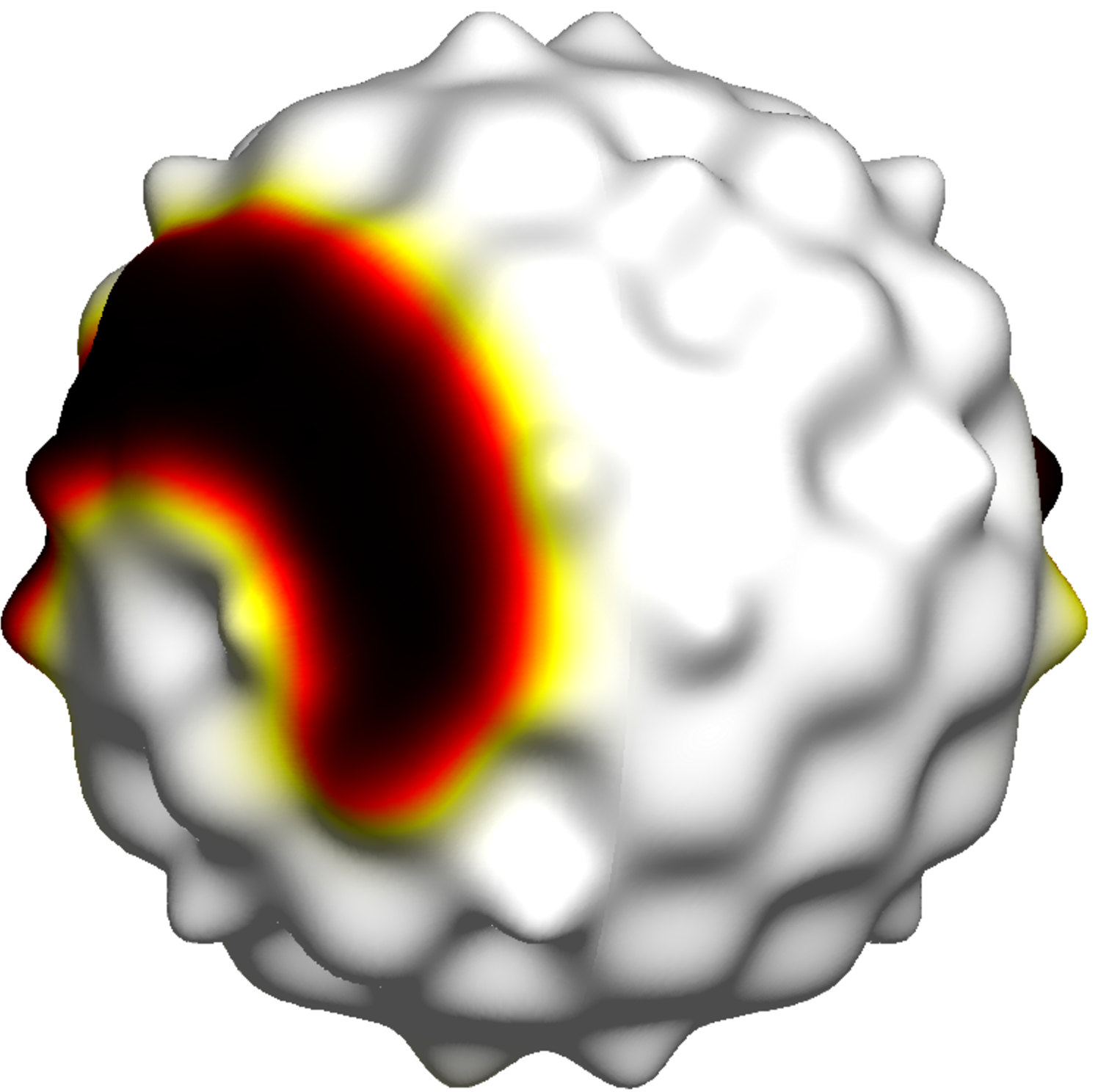}  & \includegraphics[width=0.27\textwidth]{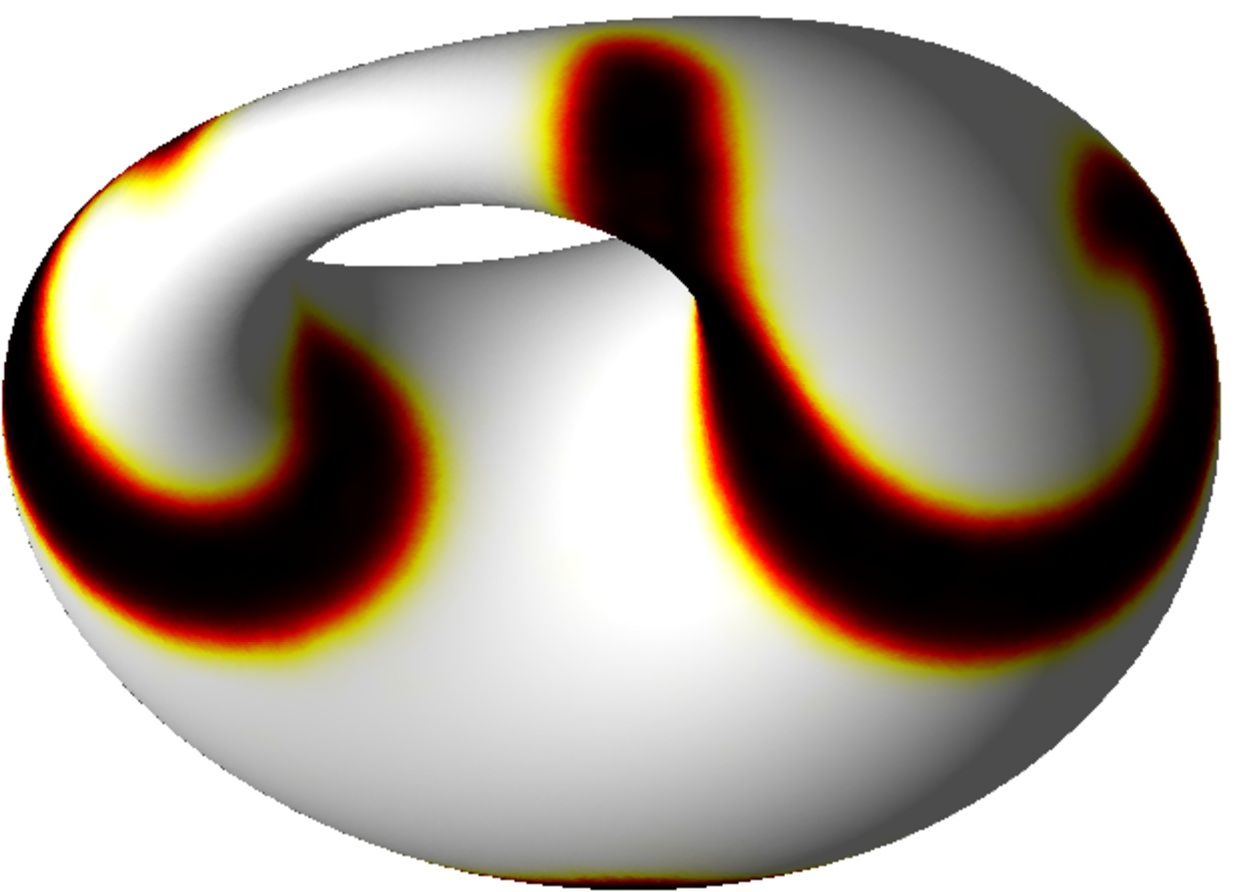}\\
\includegraphics[width=0.25\textwidth]{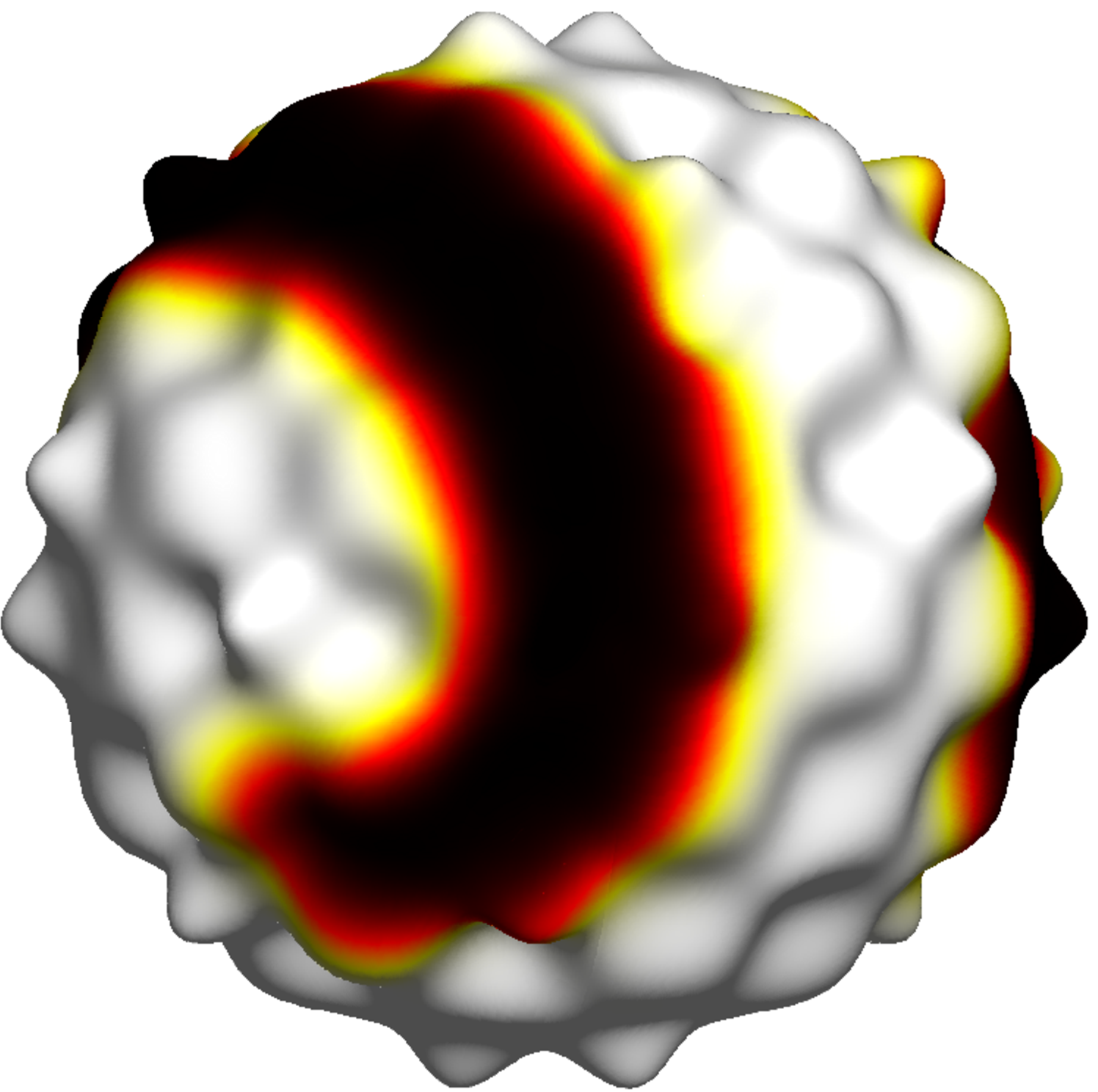}  & \includegraphics[width=0.27\textwidth]{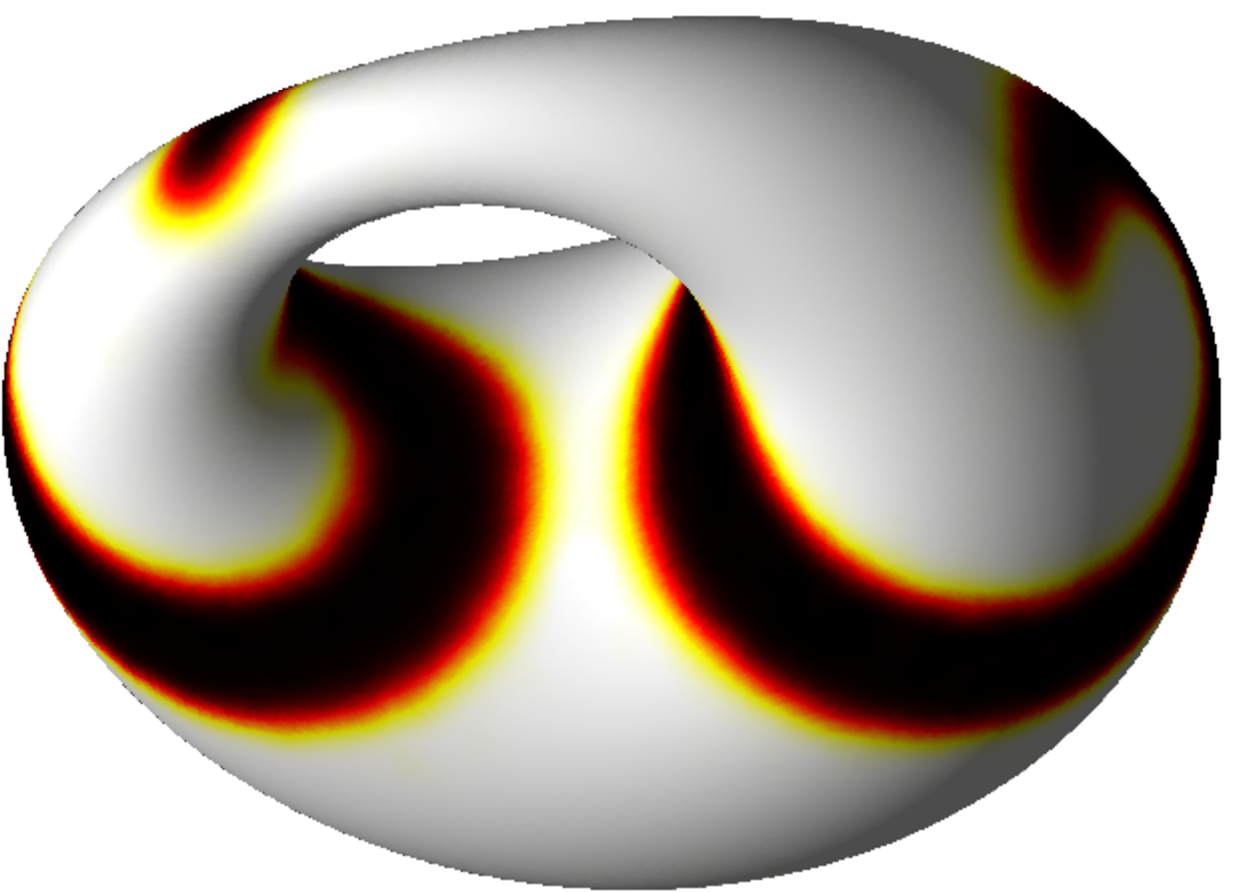}\\
\includegraphics[width=0.25\textwidth]{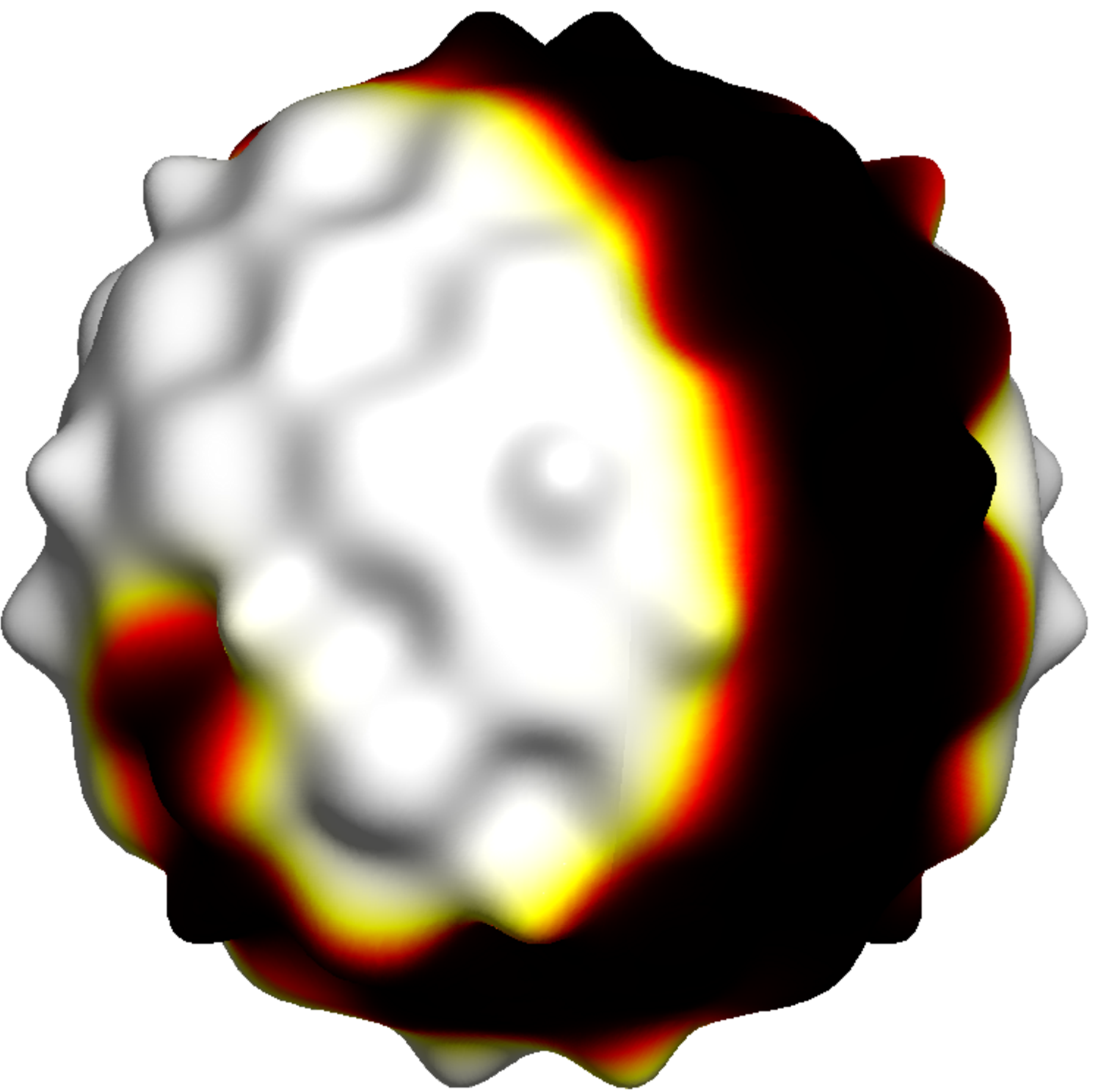}  & \includegraphics[width=0.27\textwidth]{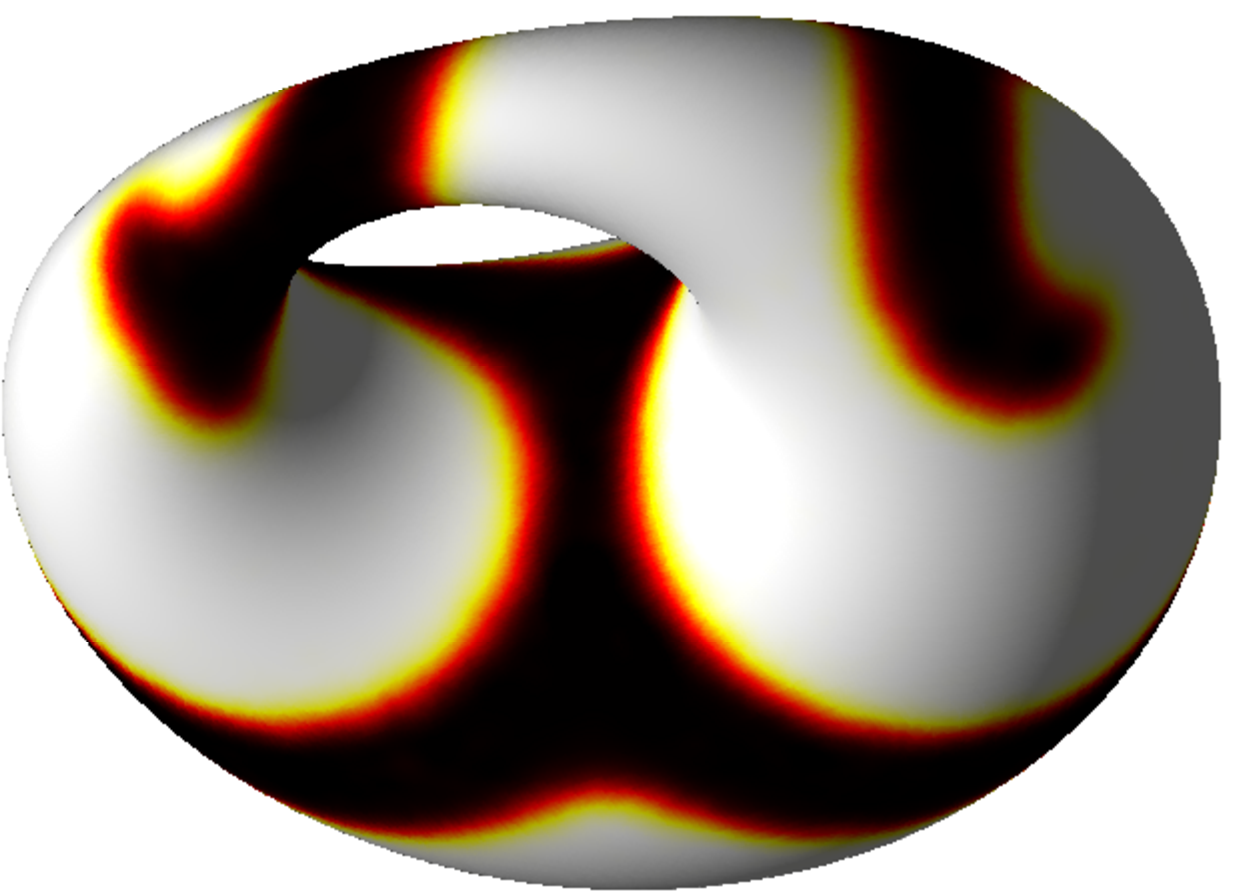}\\
\includegraphics[width=0.25\textwidth]{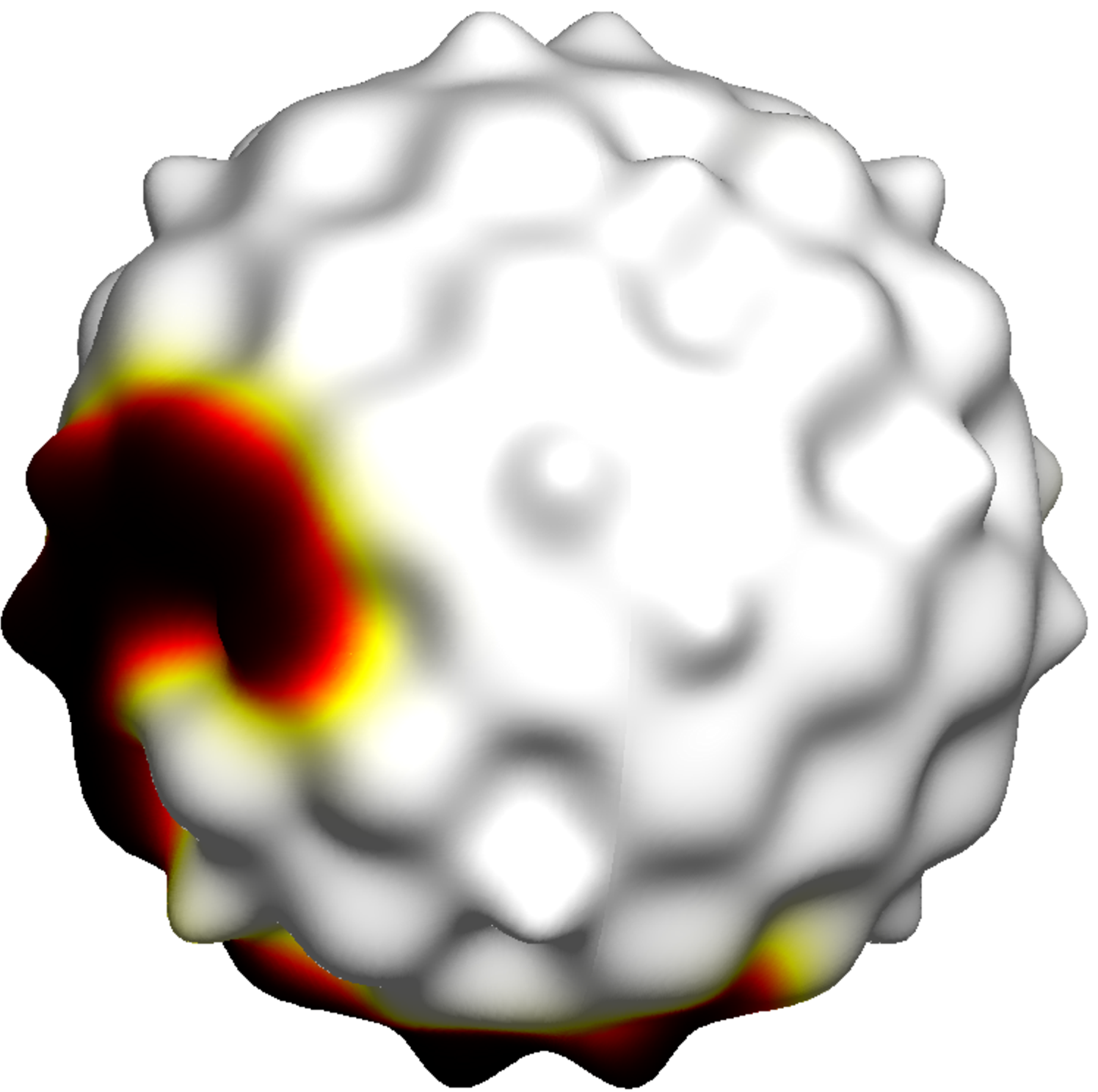}  & \includegraphics[width=0.27\textwidth]{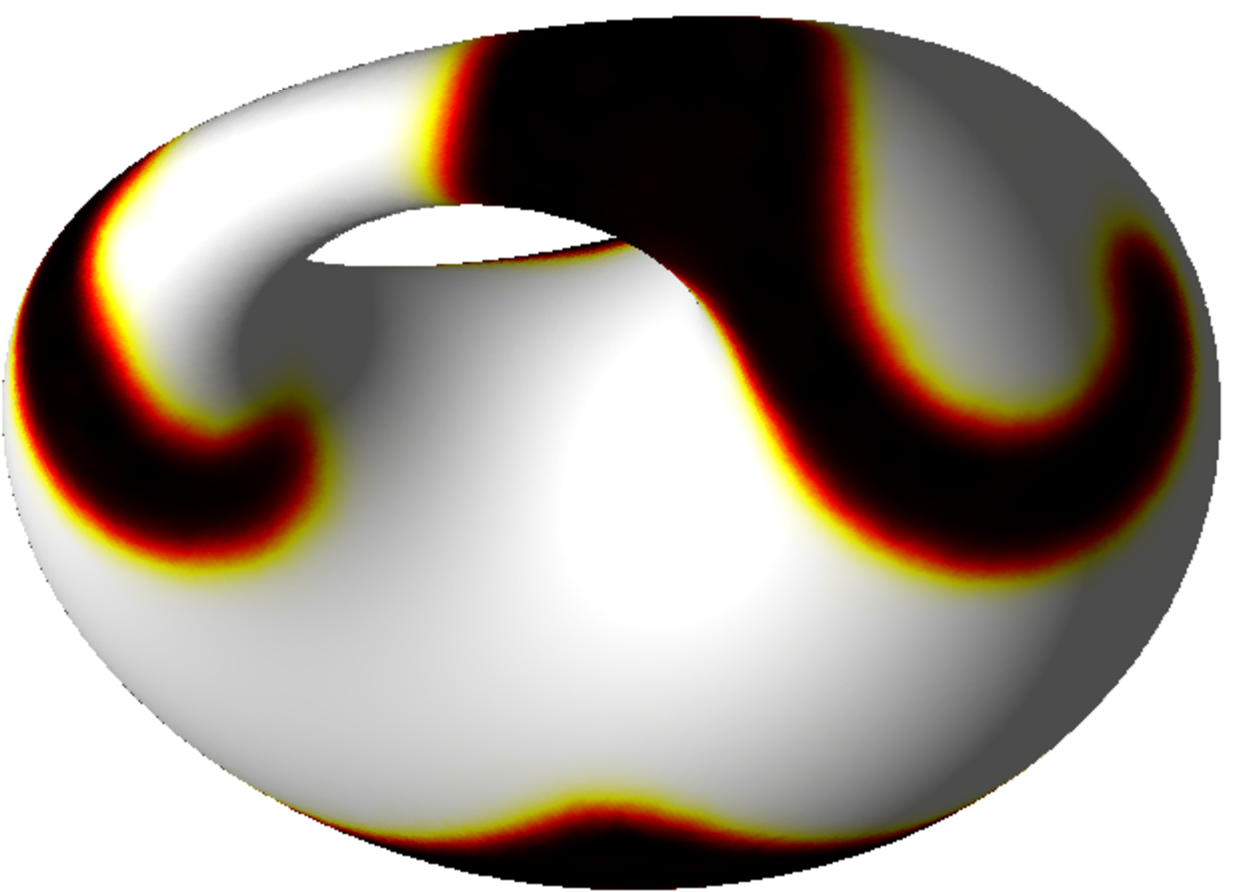}
\end{tabular}
\caption{Snap shots of spiral wave patterns computed from the Fitzhugh-Nagumo model \eqref{eq:barkley} at different times.  The left column shows the excitation variable $u$ on the bumpy sphere at times $t=42.24, 42.80, 43.76$ and $44.8$.  The right column shows $u$ on Dupin's cyclide at times $t=42.76, 43.48, 44.40$ and $45.24$.  The node sets used in the simulations are shown in Figure \ref{fig:ex_surfaces} (c) and (e), respectively, and the time-step was set to $\dt=0.02$.  The colors in the pseudocolor plots change linearly from white ($u=0$) to black ($u=1$).\label{fig:spiral_waves}.}
\end{figure}

Spiral waves can be observed in many excitable chemical, biological, and physical media~\cite{Tyson1988327,keener1998mathematical,EpsteinPojman1998}.  Important examples include Belousov-Zhabotinsky chemical reactions and electrical activity in the membranes of organisms.    While numerical methods have been developed for these models in planar domains (see~\cite{Barkley1991,Barkley:2008} and the corresponding software EZ-Spiral), there has been growing interest in studying these models on non-planar surfaces since most physically relevant problems occur on curved surfaces and curvatrue can effect the wave processes~\cite{DavydovEtAl2000,Grindrod08041991,ManzEtAl2003}.   To illustrate how our kernel method may be applied to these models, we focus on a simple two-variable reaction-diffusion model used in~\cite{Barkley1991}:
\begin{align}
\begin{split}
\frac{\partial u}{\partial t} =& \delta_u \laps u + \underbrace{\frac{1}{\alpha}u\lp 1-u \rp \lp u - \frac{v+b}{a}\rp}_{\ds f_u(u,v)},\\
\frac{\partial v}{\partial t} =& \delta_v \laps v + \underbrace{u-v}_{\ds f_v(u,v)},
\end{split}
\label{eq:barkley}
\end{align}
where $u$ and $v$ can be viewed as some chemical concentrations or as membrane potential and current.  The parameters $a$, $b$, and $\alpha$ govern the reaction kinetics and $\delta_u$ and $\delta_v$ are the diffusivities of the $u$ and $v$ species respectively.  The parameter $\alpha$ is chosen as $\alpha << 1$ so that the $u$ field takes on the values $u=0$ or $u=1$ almost everywhere, with a thin interface (or reaction zone) separating these two regions.  The system \eqref{eq:barkley} is of Fitzhugh-Nagumo type~\cite{FitzHugh1961} and provides a simple model for dynamics of many excitable media.   These types of systems have also been studied on the surface of the sphere~\cite{GomatamAmdjadi,Yagisita1998126}.

Figure \ref{fig:spiral_waves} displays snap shots from the numerical solution of \eqref{eq:barkley} on the bumpy sphere and Dupin's cyclide at different times of the simulation.  For all of these experiments we set the initial conditions to
\begin{align*}
u(0,\vx) &= \frac12\left[1 + \tanh(2x + y)\right], \\
v(0,\vx) &= \frac12\left[1 - \tanh(3z)\right],
\end{align*}
and $\dt = 0.02$.  Values for the parameters that remain fixed for the two surfaces are as follows:  $a = 0.75$, $b=0.02$, $\alpha=0.02$, and $\delta_v=0$.  The values of $\delta_u$ vary depending on the surface as follows:  $\delta_u = 1.5(2\pi/50)^2$ for the bumpy sphere and $\delta_u = 2.5(2\pi/50)^2$.   These values are similar to those used in one of the experiments from~\cite{Barkley:2008} on a 2-D planar domain.  The two quasi-periodic, counter-rotating spiral waves seen on both surfaces in Figure \ref{fig:spiral_waves} are qualitatively similar to those observed on the sphere in~\cite{GomatamAmdjadi}.  These spiral waves remain intact but meander around the surfaces throughout the simulation time.

\section{Concluding remarks} \label{sec:concluding_remarks}

In this paper, we have introduced a new kernel method based on collocation with RBFs for constructing discrete approximations to differential operators on surfaces, with a focus on surfaces of dimension 2 embedded in $\R^3$.  The method is different than other methods for approximating surface differential operators in that it only requires ``scattered'' nodes on the surface and the corresponding normal vectors, it does not require expanding into the $\R^3$ embedding space, and it can give high-orders of accuracy.  We have studied the approximation power of the method and have provided error estimates for target functions of various smoothness.  We have used the method to construct discrete approximations to the surface Laplacian and combined it in a method-of-lines approach to numerically solve diffusion and reaction-diffusion equations on surfaces.  We presented numerical results for two forced diffusion equations, illustrating the convergence of the method for kernels with finite and infinite smoothness.  The numerical results for the finitely smooth kernels indicate that the convergence rates are faster than our theory predicts, which may be related to some form of super-convergence common to spline methods.  Finally, we have illustrated the flexibility of our method by successfully applying it to two important problems in biology and chemistry.

The primary shortcoming of the method is the $\bO(N^2)$ computational cost for applying these operators.  In a follow up study, we intend to address this issue by using a local RBF finite-difference (RBF-FD) method for constructing the approximate surface differential operators.  This method has been used successfully in~\cite{FornbergLehto2011,FlyerLehtoBlaiseWrightStCyr2012} to overcome the computational cost of the global method~\cite{FlyerWright07,FlyerWright09} for the shallow water wave equations on the sphere and has recently been adapted to multi-GPU architectures in~\cite{BolligFlyerErlebacher2012}.  We expect the cost of constructing and applying these discrete operators based on RBF-FD to be reduced to $\bO(N)$, with a minor decrease in the order of accuracy.  These enhancements may make the method competitive for solving diffusion and reaction-diffusion equations on evolving surfaces, which is also important in many biological applications.

%\clearpage

\appendix

%%%%%%%%%%%%%%%%%%%%%%%%%%%%%%%%%%%%%%%%%%%%%%%
\section{Convergence Results}\label{proofappendix}
%%%%%%%%%%%%%%%%%%%%%%%%%%%%%%%%%%%%%%%%%%%%%%%

In this section we present the convergence results for our discrete differential operators. The results presented here are for the special case when $\M\subset \R^3$ is a two dimensional surface, although similar results hold in higher dimensions. The arguments will depend heavily on error estimates for kernel interpolants on smooth, embedded submanifolds of $\R^d$ given recently in \cite{FuselierWright2011}. We will frequently reference results from that paper throughout this section. Sobolev error estimates for the surface gradient and divergence operators will immediately follow the results in \cite{FuselierWright2011}. However, the estimates for the discrete surface Laplacian will require more work. 

\begin{proposition}\label{prop:discretegraddivallkernels}
Let $1\leq q \leq \infty$, $\phi$ be a positive definite kernel satisfying \eqref{eq:fastdecay} with $\tau > 3/2 + 1$, and define $s = \tau - 1/2$. Then there is a constant $h_{\M}$ depending only on $\M$ such that if a finite node set $X\subset\M$ satisfies $h\leq h_\M$, for all $f\in \mathcal{N}_\phi(\M)$ and $\mathbf{f}\in (\mathcal{N}_\phi(\M))^3$ we have
\begin{eqnarray*}
\|G_\M f-\nabla_{\M} f\|_{\mathbf{L}_q(\M)}&\leq &C  h^{s-1-2(1/2-1/q)_+}\|f\|_{\mathcal{N}_\phi(\M)},\\
\|D_{\M}\mathbf{f}-\nabla_{\M}\cdot \mathbf{f}\|_{L_q(\M)}&\leq &C  h^{s-1-2(1/2-1/q)_+}\|\mathbf{f}\|_{\mathcal{N}_\phi(\M)},
\end{eqnarray*}
where $(x)_+ = x$ if $x\geq 0$ and is zero otherwise.
\end{proposition}

\begin{proof}
We will focus on the discrete gradient operator. We have 
\[\|G_\M f-\nabla_{\M} f\|_{\mathbf{L}_q(\M)} =\|\sgrad I_{\phi}f-\nabla_{\M} f\|_{\mathbf{L}_q(\M)} \leq C\|I_{\phi}f - f\|_{W^{1}_q(\M)},\]
where $W^1_q(\M)$ is the $L_q$ Sobolev space of order 1 (see Section 2, \cite{FuselierWright2011}). The assumptions on $\tau$ allow us to apply Theorem 4.6 in \cite{FuselierWright2011}, and the results follow. The error bounds in Theorem 4.6 in \cite{FuselierWright2011} carry over to the vector-valued case, thus bounds for the divergence operator error are obtained in a similar manner.
\end{proof}

When $\phi$ satisfies \eqref{eq:algdecay}, i.e. $\phi$ has finite smoothness, the estimates come in two types: the first concerns targets that may be too rough to be within the native space, and the second applies to functions with additional smoothness. This additional smoothness is measured with the inverse of the integral operator\footnote{With this integral operator comes the pseudodifferential operators $T^{-r}$, $r>0$. A function $f$ is in the native space if and only if $T^{-1/2}f \in L_2(\M)$ \cite[Proposition 4.9]{FuselierWright2011}, so we expect functions such that $T^{-1}f\in L_2(\M)$ to be twice as smooth. }
\[Tf(x) := \int_{\M}\phi(y,x)f(y)d\mu(y).\]
We denote a vector version of this operator by $\mathbf{T}$, which simply applies $T$ to each component of a $3$-dimensional vector field. Similar to the proof above, the result below follows from Theorem 4.12 and Corollary 4.10 in \cite{FuselierWright2011}.

%%%%%%%%%%%%%%%%%%%%%%%%%%%%%%%%%%%%%%%%%%%%%%%%%%%%%
\begin{proposition}\label{prop:discretegraddiv}
Let $1\leq q \leq \infty$, and let $\phi$ be a positive definite kernel satisfying (\ref{eq:algdecay}) with $\tau > 3/2 + 1$, and define $s = \tau - 1/2$. Then there is a constant $h_{\M}$ depending only on $\M$ such that if a finite node set $X\subset\M$ satisfies $h\leq h_\M$, we have the following:
\begin{enumerate}
\item  \textbf{\emph{Rough target functions}}. Let $\beta$ be such that $s \geq \beta > 2$. Then for all $f \in H^\beta(\M)$ and $\mathbf{f}\in\mathbf{H}^\beta(\M)$ we have
\begin{eqnarray*}
\|G_\M f-\nabla_{\M} f\|_{\mathbf{L}_q(\M)}&\leq &C  h^{\beta-1-2(1/2-1/q)_+}\rho^{s -\beta}\|f\|_{H^{\beta}(\M)},\\
\|D_{\M}\mathbf{f}-\nabla_{\M}\cdot \mathbf{f}\|_{L_q(\M)}&\leq &C  h^{\beta-1-2(1/2-1/q)_+}\rho^{s - \beta}\|\mathbf{f}\|_{\mathbf{H}^{\beta}(\M)}. 
\end{eqnarray*}
\item  \textbf{\emph{Smooth target functions}}. For all $f\in L_2(\M)$ such that $T^{-1}f\in L_2(\M)$ and $\mathbf{f}\in \mathbf{L}_2(\M)$ such that $T^{-1}\mathbf{f}\in \mathbf{L}_2(\M)$, we have
\begin{eqnarray*}
\|G_\M f-\nabla_{\M} f\|_{\mathbf{L}_q(\M)}&\leq &C  h^{2s-1-2(1/2 - 1/q)_+}\|T^{-1}f\|_{L_2(\M)},\\
\|D_{\M}\mathbf{f}-\nabla_{\M}\cdot \mathbf{f}\|_{\mathbf{L}_q(\M)}&\leq &C  h^{2s-1-2(1/2 - 1/q)_+}\|\mathbf{T}^{-1}\mathbf{f}\|_{\vL_2(\M)}. 
\end{eqnarray*}
\end{enumerate}
\end{proposition}

Obtaining convergence rates for the discrete Laplace operator is not as straightforward, and will require extra tools. The first is the ``zeros lemma'' from Narcowich, Ward and Wendland \cite{NarcWardWendland:2005ScatteredZeros}. We will use the manifold version stated in \cite[Lemma 4.5]{FuselierWright2011}, with parameters adapted to our situation.

\begin{proposition}\label{zeroslemma}
Let $1\leq q\leq \infty$, $t\in \R$ with $t > 1$. Let $\mu\in \mathbb{N}$ satisfy $0\leq\mu \leq \lceil t - 2(1/2 - 1/q)_{+}\rceil -1$. Also, let $X\subset\M$ be a discrete set with mesh norm $h_{X,\M}$.  Then there is a constant depending only on $\M$ such that if $h_{X,\M}\leq C_{\M}$ and if $u\in H^{t}(\M)$ satisfies $u|_{X}=0$, then
\[|u|_{W_{q}^{\mu}(\M)}\leq Ch^{t-\mu-2(1/2-1/q)_{+}}|u|_{H^{t}(\M)}.\]
\end{proposition}

\noindent Proposition \ref{zeroslemma} also holds with full Sobolev norms, and a similar result holds for vector fields.

Since the operator $L_{\M}$ is a mixture of differential and interpolation operators, it will be beneficial to bound the norm of the interpolation operator and its associated error in several different Sobolev spaces. In particular, we will use the following.

\begin{lemma}\label{interpnorm_rough}
Let $\phi$ satisfy (\ref{eq:algdecay}) with $\tau > 3/2$, and define $s = \tau - 1/2$. Let $\beta$ be such that $s \geq \beta > 1$. Then there is a constant $h_{\M}$ depending only on $\M$ such that if a finite node set $X\subset\M$ satisfies $h\leq h_\M$ then for all $f\in H^{\beta}(\M)$ we have
\[\|I_{\phi}f - f\|_{H^{\beta}(\M)} \leq C \rho^{s - \beta}\|f\|_{H^{\beta}(\M)}\mbox{\hspace{.25in} and \hspace{.25in}} \|I_{\phi}f\|_{H^{\beta}(\M)} \leq C \rho^{s - \beta}\|f\|_{H^{\beta}(\M)}.\]
\end{lemma}

\begin{proof}
The last estimate in the proof of Theorem 4.12 in \cite{FuselierWright2011} is
\[ 
\|f - I_{\phi}f\|_{H^{\beta}(\M)} \leq C \rho^{s - \beta}\|f\|_{H^{\beta}(\M)}.
\]
Applying a triangle inequality and observing that $1\leq \rho$, we get
\[ \|I_{\phi}f\|_{H^{\beta}(\M)} \leq \|f - I_{\phi}f\|_{H^{\beta}(\M)} + \|f\|_{H^{\beta}(\M)} \leq C (\rho^{s - \beta}+1)\|f\|_{H^{\beta}(\M)} \leq C\rho^{s - \beta}\|f\|_{H^{\beta}(\M)}.\]
\end{proof}

Now we are poised to present the following theorem. 

%%%%%%%%%%%%%%%%%%%%%%%%%%%%%%%%
\begin{theorem}\label{theorem:discretelaplaceest_rough}
Let $1\leq q \leq  \infty$, $\phi$ satisfy (\ref{eq:algdecay}) with $\tau > 3/2 + 2$, and define $s = \tau - 1/2$. Let $\beta$ be such that $s \geq \beta > 3$. Then there is a constant $h_{\M}$ depending only on $\M$ such that if a finite node set $X\subset\M$ satisfies $h\leq h_\M$, then for all $f\in H^\beta(\M)$ we have the estimate
\begin{equation}\label{eq:LXerror_rough}
\|L_{\M}f-\Delta_\M f\|_{L_q(\M)}\leq C  h^{\beta-2-2(1/2-1/q)_+}\rho^{2(s - \beta) + 1}\|f\|_{H^{\beta}(\M)}.
\end{equation}
\end{theorem}
%%%%%%%%%%%%%%%%%%%%%%%%%%%%%%%%

\begin{proof}
First, we have
\begin{equation}\label{eq:discretelaplaceest1}
\|L_{\M}f-\Delta_\M f\|_{L_q(\M)}\leq \|L_{\M}f-\Delta_\M I_{\phi}f\|_{L_q(\M)} + \|\Delta_\M f - \Delta_\M I_{\phi}f\|_{L_q(\M)}.
\end{equation}
For the rightmost term, the assumption $\beta > 3$ allows us to use \cite[Theorem 4.12]{FuselierWright2011} to get:
\begin{eqnarray*}
\|\Delta_\M f - \Delta_\M I_{\phi}f\|_{L_q(\M)}&\leq& C \|f - I_{\phi}f\|_{W^{2}_q(\M)} \leq C h^{\beta-2 -2(1/2-1/q)_+}\rho^{s-\beta}\|f\|_{H^{\beta}(\M)}.
\end{eqnarray*}

For the other term in (\ref{eq:discretelaplaceest1}), we have
\begin{eqnarray}
\|L_{\M}f-\Delta_\M I_{\phi}f\|_{L_q(\M)} & = & \|\nabla_\M\cdot (I_{\Phi}(\nabla_\M I_{\phi}f)) -\Delta_\M I_{\phi}f\|_{L_q(\M)} \nonumber\\
& = & \|\nabla_\M\cdot (I_{\Phi}(\nabla_\M I_{\phi}f)) -\nabla_\M \cdot (\nabla_\M I_{\phi}f)\|_{L_q(\M)} \nonumber\\
&\leq & C\|I_{\Phi}(\nabla_\M I_{\phi}f) -\nabla_\M I_{\phi}f\|_{\mathbf{W}^{1}_q(\M)}.\label{eq:discretelaplaceest2}
\end{eqnarray}
Note that the last estimate is the interpolation error for the target function $\vg := \nabla_\M I_{\phi}f$. 

We claim that $\vg$ is in $\mathbf{H}^{t}(\M)$ for all $t < 2s - 2$. Since $\phi$ satisfies \eqref{eq:algdecay}, $I_{\phi}f \in H^{\nu}(\R^3)$ for all $\nu < 2\tau - 3/2$. By the trace theorem restricting to $\M$ puts $I_{\phi}f$ in  $H^{\nu - 1/2}(\M)$. Thus $\vg \in  \mathbf{H}^{\nu - 3/2}(\M)$ for all $\nu < 2\tau - 3/2$, which is equivalent to $\vg\in \mathbf{H}^{t}(\M)$ for all $t < 2s - 2$. In particular, $\vg \in \mathbf{H}^{\beta - 1}(\M)$. Also, since $\beta > 2$ then we have
\[1 \leq  \lceil (\beta - 1) - 2(1/2 - 1/q)_+ \rceil - 1,\]
and that $\beta-1 > 1$. This allows us to estimate (\ref{eq:discretelaplaceest2}) with Proposition \ref{zeroslemma} (with parameter $\mu = 1$ and target smoothness $t = \beta-1$) to get
%Indeed, we have
%\[  2s - (d-1)/2 - 1 = s + s - (d-1)/2 - 1 > s + ((d - 1)/2 + 2)  - (d-1)/2 - 1 = s + 1.\]
%Thus  $\vg \in \mathbf{H}^{s+1}(\M)\subset \mathbf{H}^{\beta}(\M) \subset \mathbf{H}^{\beta-1}(\M)$. 
\begin{eqnarray*}
\|I_{\Phi}(\nabla_\M I_{\phi}f) -\nabla_\M I_{\phi}f\|_{\mathbf{W}^{1}_q(\M)} &\leq& C h^{(\beta - 1) - 1-2(1/2-1/q)_+}\| I_{\Phi}\vg - \vg\|_{\mathbf{H}^{\beta-1}(\M)}.\\
& = & C h^{\beta - 2 -2(1/2-1/q)_+}\| I_{\Phi}\vg - \vg\|_{\mathbf{H}^{\beta-1}(\M)}.
\end{eqnarray*}
Further, since $\beta - 1>1$, we can also apply Lemma \ref{interpnorm_rough}, and with two applications of it we get
\begin{eqnarray}
\| I_{\Phi}\vg - \vg\|_{\mathbf{H}^{\beta-1}(\M)} & \leq & C\rho^{s-\beta+1}\| \vg\|_{\mathbf{H}^{\beta-1}(\M)} =  C\rho^{s-\beta+1}\| \sgrad I_{\phi}f\|_{\mathbf{H}^{\beta-1}(\M)}\label{eq:strongernorm}\\
& \leq  &  C\rho^{s-\beta+1}\|I_{\phi}f\|_{H^{\beta}(\M)}\leq  C\rho^{2(s-\beta)+1}\|f\|_{H^{\beta}(\M)}\nonumber.
\end{eqnarray}
This completes the proof.
\end{proof}
Continuing with our error analysis, we now shift our attention to target functions that are very smooth. First we will need a lemma. 
%%%%%%%%%%%%%%%%%%%%%%%%%%%%%%%%%%%%%%%%%%%%%%%%%%%%%
\begin{lemma}\label{extraorders}
Let $\phi$ satisfy (\ref{eq:algdecay}) with $\tau > 3/2$, and let $s = \tau - 1/2$. Then for all $f\in L_2(\M)$ such that $T^{-1}f \in L_2(\M)$ we have the following estimate 
\[
\|f - I_{\phi}f\|_{H^s(\M)}\leq C h^s\|T^{-1}f\|_{L_2(\M)}.
\]
Also, for all $\vf\in \vL_2(\M)$ such that $\vT^{-1}\vf \in \vL_2(\M)$ we have
\[
\|\vf - I_{\Phi}\vf\|_{\vH^s(\M)}\leq C h^s\|\vT^{-1}\vf\|_{\vL_2(\M)}.
\]
\end{lemma}
%%%%%%%%%%%%%%%%%%%%%%%%%%%%%%%%%%%%%%%%%%%%%%%%%%%%%
\begin{proof}
The first estimate is established in the proof of Corollary 4.10 in \cite{FuselierWright2011}. By working component-wise the second estimate follows from the first.
\end{proof}

%%%%%%%%%%%%%%%%%%%%%%%%%%%%%%%%%%%%%%%%%%%%%%%%%%%%%%%%%%%%%%%%%%
\begin{theorem}\label{theorem:discretelaplaceest_smooth}
Let $1\leq q \leq  \infty$, $\phi$ satisfy (\ref{eq:algdecay}) with $\tau > 3/2 + 2$, and define $s = \tau - 1/2$. %Let $\mu \in \mathbb{N}$ satisfy $0\leq \mu + 2\leq \lceil s - (d-1)(1/2 - 1/q)_+ \rceil - 1$. 
Let $f\in L_2(\M)$ be such that $T^{-1}f \in L_2(\M)$ and $\vT^{-1}\sgrad f \in \vL_2(\M)$. Then there is a constant $h_{\M}$ depending only on $\M$ such that if a finite node set $X\subset\M$ satisfies $h\leq h_\M$, we have the estimate
\begin{equation*}%\label{eq:LXerror_smooth}
\|L_{\M}f-\Delta_\M f\|_{L_q(\M)}\leq C  h^{2s-2-2(1/2-1/q)_+}\rho(\|T^{-1}f\|_{L_2(\M)} + \| \vT^{-1}\sgrad f\|_{\vL_2(\M)}).
\end{equation*}
\end{theorem}
%%%%%%%%%%%%%%%%%%%%%%%%%%%%%%%%%%%%%%%%%%%%%%%%%%%%%%%%%%%%%%%%%%

\begin{proof}
We begin as in the proof of the previous theorem. We have
\begin{equation}\label{eq:discretelaplaceest2_smooth}
\|L_{\M}f-\Delta_\M f\|_{L_q(\M)}\leq \|L_{\M}f-\Delta_\M I_{\phi}f\|_{L_q(\M)} + \|\Delta_\M f - \Delta_\M I_{\phi}f\|_{L_q(\M)}.
\end{equation}
To estimate the rightmost term we can use the error estimates in \cite[Corollary 4.10]{FuselierWright2011} to get:
\begin{eqnarray*}
\|\Delta_\M f - \Delta_\M I_{\phi}f\|_{L_q(\M)}&\leq& C \|f - I_{\phi}f\|_{W^{2}_q(\M)} \leq C h^{2s-2-2(1/2-1/q)_+}\|T^{-1}f\|_{L_2(\M)}.
\end{eqnarray*}

For the other term in \eqref{eq:discretelaplaceest2_smooth}, we proceed as in the proof of Theorem \ref{theorem:discretelaplaceest_rough} to get 
\begin{eqnarray}
\|L_{\M}f-\Delta_\M I_{\phi}f\|_{L_q(\M)} &\leq & C\|I_{\Phi}(\vg) -\vg\|_{\mathbf{W}^{1}_q(\M)},\label{eq:discretelaplaceest3}
\end{eqnarray}
where $\vg = \nabla_\M I_{\phi}f$, and as before we know that this function is in $\mathbf{H}^{t}(\M)$ for all $t < 2s - 2$. In particular, $\vg \in \mathbf{H}^{s-1}(\M)$ and %. Note that the assumption on $\tau$ 
%implies that $s - 1 > 3/2$, and we also have $1\leq \lceil (s-1) - 2(1/2 - 1/q)_+ \rceil - 1$, 
we can estimate (\ref{eq:discretelaplaceest3}) with Proposition \ref{zeroslemma} (with $t=s-1$ and $\mu =1$) to get
\begin{eqnarray*}
\|I_{\Phi}(\vg) -\vg\|_{\mathbf{W}^{1}_q(\M)} &\leq& C h^{s - 2 -2(1/2-1/q)_+}\| I_{\Phi}\vg - \vg\|_{\mathbf{H}^{s-1}(\M)}.
\end{eqnarray*}
This is where the proof detours from that of the previous theorem. To estimate $\| I_{\Phi}\vg - \vg\|_{\mathbf{H}^{s-1}(\M)}$, note that it is bounded it by following quantity:
\[
\underbrace{\|I_{\Phi}(\sgrad I_{\phi}f) - I_{\Phi}(\sgrad f)\|_{\mathbf{H}^{s-1}(\M)}}_{I} + \underbrace{\|I_{\Phi}(\sgrad f) - \sgrad f\|_{\mathbf{H}^{s-1}(\M)}}_{II} + \underbrace{\|\sgrad f - \sgrad I_{\phi}f\|_{\mathbf{H}^{s-1}(\M)}}_{III}.\]
We will bound each term individually. First we concentrate on $I$. An application of Lemma \ref{interpnorm_rough}, which we may apply since $s - 1 > 1$, gives us
\begin{eqnarray*}
\|I_{\Phi}(\sgrad I_{\phi}f) - I_{\Phi}(\sgrad f)\|_{\mathbf{H}^{s-1}(\M)} & = &\|I_{\Phi}(\sgrad I_{\phi}f - \sgrad f)\|_{\mathbf{H}^{s-1}(\M)} \\
& \leq & C\rho \|\sgrad I_{\phi}f - \sgrad f\|_{\mathbf{H}^{s-1}(\M)} = C\rho(III).
\end{eqnarray*}
Thus a bound for $III$ will result in a bound for $I$. To bound $III$, we can apply Lemma \ref{extraorders} to get 
\begin{eqnarray*}
\|\sgrad I_{\phi}f - \sgrad f\|_{\mathbf{H}^{s-1}(\M)}&\leq& C \|I_{\phi}f - f\|_{\mathbf{H}^{s}(\M)} \leq h^{s}\|T^{-1}f\|_{L_2(\M)}.
\end{eqnarray*}
To bound $II$, we again employ Lemma \ref{extraorders}:
\begin{eqnarray*}
\|I_{\Phi}(\sgrad f) - \sgrad f\|_{\mathbf{H}^{s-1}(\M)} & \leq & \|I_{\Phi}(\sgrad f) - \sgrad f\|_{\mathbf{H}^{s}(\M)} \leq C h^s\|\vT^{-1}\sgrad f\|_{\vL_2(\M)}.
\end{eqnarray*}
This completes the proof.
\end{proof}

%\begin{remark}
%Assuming a quasi-uniform point set $X$, in the case of $q=2$, $\mu = 0$, Proposition \ref{prop:discretegraddiv} predicts that the error is at worst $\mathcal{O}(h^{2s-1})$ for the order $1$ differential operators when the target (or its components) satisfy $T^{-1} f \in L_2(\M)$. Since the laplacian is order $2$, we expect an error bound of $\mathcal{O}(h^{2s-2})$, which Theorem \ref{theorem:discretelaplaceest_smooth} predicts. However, here we assume $T^{-1}f\in L_2(\M)$ with the additional assumption $\vT^{-1}\sgrad f \in \vL_2(\M)$. We believe this to be an artifact of the proof; the smoothness assumption can probably be relaxed to $T^{-1}f\in L_2(\M)$.
%\end{remark}

%%%%%%%%%%%%%%%%%%%%%%%%%%%%%%%%%%%%%%%%%%%%%%%%%%%%%%
\section{Surfaces and node sets from the numerical experiments} \label{apndx:Surfaces}
%%%%%%%%%%%%%%%%%%%%%%%%%%%%%%%%%%%%%%%%%%%%%%%%%%%%%%

\subsection{Unit sphere}\label{apndx:sphere}
This manifold is, of course, described implicitly by 
\begin{align}
\mathbb{M} = \left\{\vx = (x,y,z) \in \R^3 \;\left| x^2 + y^2 + z^2 = 1 \right. \right\}.
\label{eq:sphere}
\end{align}
The node sets we use for discretizing the unit sphere are the minimum energy (ME) node sets of Womersley and Sloan~\cite{WomerSloanSpherePts}.  These node sets are approximately uniformly distributed over the surface of the sphere and have the nice property that the mesh norm $h$ and the separation radius $q$ decrease uniformly like the inverse of the square root of the number of nodes $N$, i.e. $h,q \sim \frac{1}{\sqrt{N}}$.  Additionally, the nodes in these sets are not oriented along any vertices or lines, which emphasizes the ability of our method to handle arbitrary node layouts.  They have been used quite successfully in many other RBF applications, e.g.~\cite{NarcWardWright:2007DivFreeRBFsSurfaces,FuselierNarcWardWright:2008DivFreeRBFConvergence,FuselierWright:2009VectorDecomposition,FlyerWright09,FornbergLehto2011,FlyerWright07,WrightFlyerYuen}.

\subsection{Red blood cell}
This manifold is a mathematical model for human red blood cells in static equilibrium conditions.  The model was first proposed in~\cite{EvansFung1972} and has been used in many subsequent studies (e.g.~\cite{LiuLiu2006}).  The model can be described parametrically as follows:
\begin{align}
\mathbb{M} = \left\{(x,y,z)\in\R^3\;\left| x = r_0\cos\lambda\cos\theta,\; y = r_0\sin\lambda\cos\theta,\; z = \frac{1}{2}\sin\theta\lp c_0 + c_2\cos^2\theta + c_4\cos^4\theta\rp\right.\right\},
\label{eq:rbc}
\end{align}
where $-\pi/2 \leq \theta \leq \pi/2$, $-\pi \leq \lambda < \pi$, $r_0 = 3.91/3.39$, $c_0 = 0.81/3.39$, $c_2 = 7.83/3.39$, and $c_4=-4.39/3.39$.
The node sets we used for discretizing this manifold were obtained using a radial projection of the ME points for the surface of the sphere described above.

\subsection{Bumpy sphere}
This manifold is constructed from the ``bumpy sphere'' surface~\cite{BumpySphere} using the following procedure.  First, $N=5256$ points on the bumpy sphere surface were obtained from \url{http://shapes.aimatshape.net}.  This surface is homeomorphic to the unit sphere, so spherical coordinates for each of the $N=5256$ points on the bumpy sphere were obtained.  Upon normalizing the original nodes by the maximum distance of any of the nodes from the origin, a parametric model for the surface was constructed using the RBF geometric modeling technique from~\cite{ShankarEtAl2011}.  The original (normalized) $N=5256$ points on the bumpy sphere are used as the computational nodes in all the numerical experiments and the parametric model is used for computing the normal vectors to the surface.

\subsection{Torus}\label{appndx:torus}
This manifold is given by the implicit equation:
\begin{align}
\mathbb{M} = \left\{\vx = (x,y,z) \in \R^3 \;\left| \left(1 - \sqrt{x^2 + y^2}\right)^2 + z^2 - \frac{1}{9} = 0 \right. \right\}.
\label{eq:Torus}
\end{align}
The node sets used for discretizing this manifold were obtained by arranging the nodes so that their Reisz energy (with a power of 2) is near minimal as described in Hardin and Saff's seminal article~\cite{HarinSaff04}.  As with the ME sphere nodes, these ME torus nodes sets provide a near uniform discretization of the manifold with an approximately uniform decrease in the mesh norm $h$ and the separation radius $q$.  The node sets used in the experiments were kindly given to us by Drs. Douglas Hardin and Edward Saff and Ms. Ayla Gafni from Vanderbilt University.

\subsection{Dupin's cyclide}
This manifold is given by the implicit equation:
\begin{align}
\mathbb{M} = \left\{\vx = (x,y,z) \in \R^3 \;\left| \lp x^2 + y^2 + z^2 - d^2 + b^2 \rp^2 - 4\lp a x + c d \rp^2 - 4 b^2 y^2 = 0 \right. \right\},
\label{eq:cyclide}
\end{align}
where $a=2$, $b=1.9$, $d=1$, and $c^2 = a^2 - b^2$.  The node set for discretizing this manifold was obtained from the algorithm of Palais, Palais, and Karcher (PPK)~\cite{PalaisPalaisKarcher}, which will be included in a future version of the amazing mathematical visualization software package 3-D-XplorMath~\cite{3DXplorMath}.  This algorithm generates ``uniformly random'' point clouds for surfaces and works for both implicit and parametrically defined surfaces.  To obtain the node set displayed in Figure \ref{fig:ex_surfaces}(e) we used the PPK algorithm to generate a point cloud consisting of 9562 points.  We then thinned this point set by removing points that were too close together in order to increase the separation radius.  The final node set ended up at $N=4948$.  Both choices for the starting and ending number of nodes were chosen somewhat arbitrarily. 

\subsection{Bretzel2}
This manifold is given by the implicit equation:
\begin{align}
\mathbb{M} = \left\{\vx = (x,y,z) \in \R^3 \;\left| \lp x^2(1-x^2)-y^2 \rp^2 + \frac12 z^2 = \frac{1}{40} \right. \right\}.
\label{eq:bretzel2}
\end{align}
Unlike the the other surfaces, this one does not have an explicit parameterization.  The node set for this domain was also obtained from the PPK algorithm.   In this case we started with a point cloud of 10278 points and thinned it to end up with a node set of $N=5041$ nodes.

%\bibliographystyle{abbrv}
%\bibliographystyle{spmpsci}
%\bibliography{refs}

\end{document}